\documentclass[11pt,a4paper]{article}

\usepackage[margin=3cm]{geometry}
\usepackage{amsmath}
\usepackage{amssymb,latexsym}
\usepackage{mathabx}
\usepackage{amsthm}
\usepackage[shortlabels]{enumitem}
\usepackage{array}
\usepackage{multirow}
\usepackage[hang,flushmargin,bottom]{footmisc}
\usepackage[margin=0.35in,format=hang,font=small,labelfont=bf]{caption}
\usepackage[normalem]{ulem}
\usepackage{needspace}
\usepackage{graphicx,tikz}
\usepackage[colorlinks=true,urlcolor=blue,linkcolor=blue,citecolor=blue]{hyperref}
\usepackage{palatino}
\usepackage{eulervm}
\usetikzlibrary{patterns}

\usepackage{ifthen}
\usetikzlibrary{calc}
\usetikzlibrary{decorations.pathreplacing,decorations.markings}

\newcommand{\plotptradius}{4pt}
\newcommand{\setplotptradius}[1]{\renewcommand{\plotptradius}{#1}}
\tikzset{permpt/.style={circle, draw, fill=black, inner sep=0pt, minimum width=\plotptradius}}
\tikzset{empty/.style={draw=none, fill=none}}

\newcommand\absdot[2]{
	\node[permpt] at #1 {};
}

\newcommand{\plotperm}[2][black]{ %[colour]{permutation}
	\foreach \j [count=\i] in {#2} {
        \ifnum0=\j {} \else {
 		\node[permpt,fill=#1,draw=#1] (\j) at (\i,\j) {};
	} \fi
	};
}
\newcommand{\plotpermbox}[4]{
	\draw [darkgray, very thick, line cap=round, fill=white]
		({#1-0.5}, {#2-0.5}) rectangle ({#3+0.5}, {#4+0.5});
}

\newcommand{\plotpermborder}[1]{
	\foreach \i [count=\nn] in {#1} {\global\let\n\nn};    % scope of \nn now restricted to loop, resolving tikz issue #702, 3 Jul 2019
	\plotpermbox{1}{1}{\n}{\n};
	\plotperm{#1};
}
\newcommand{\plotpermbordergrid}[1]{
	\foreach \i [count=\nn] in {#1} {\global\let\n\nn};    % scope of \nn now restricted to loop, resolving tikz issue #702, 3 Jul 2019
	\plotpermbox{1}{1}{\n}{\n};
	\draw[step=1cm,gray!50,very thin] (1,1) grid (\n,\n);
	\plotperm{#1};
}
\newcommand{\plotgrid}[1]{
	\draw[step=1cm,gray!50,very thin] (0.5,0.5) grid (#1+0.5,#1+0.5);
}

\newcommand{\plotpermgrid}[1]{
	\foreach \i [count=\nn] in {#1} {\global\let\n\nn};    % scope of \nn now restricted to loop, resolving tikz issue #702, 3 Jul 2019
	\plotgrid{\n};
	\plotperm{#1};
}

\newcommand{\plotpinsequence}[1]{
	\absdot{(0,0)}{};
	\edef\n{0}
	\edef\s{0}
	\edef\e{0}
	\edef\w{0}
	\edef\x{0}
	\edef\y{0}
	\foreach \pin [remember=\pin as \oldpin (initially 1), count=\i] in {#1} {
		\ifthenelse{\pin=1 \OR \pin=2}{%up or down
			\ifthenelse{\oldpin=3}{% previous=right
				\xdef\x{\number\numexpr\e-1}
			}{
				\xdef\x{\number\numexpr\w+1}
			}
			\ifnum\i=1 %expand eastern box by 1 if 1st pin
				\pgfmathparse{\e+1}
 				\xdef\e{\pgfmathresult}
			\fi	
		}{ %left or right
			\ifthenelse{\oldpin=1}{% previous=up
				\xdef\y{\number\numexpr\n-1}
			}{
				\xdef\y{\number\numexpr\s+1}
			}
			\ifnum\i=1 %expand southern boundary by 1 if 1st pin
				\pgfmathparse{\s-1}
 				\xdef\s{\pgfmathresult}
			\fi	
		}
		\ifnum\pin=1 %up
			\pgfmathparse{\n+2}
 			\xdef\n{\pgfmathresult}		
			\absdot{(\x,\n)}{};
			\ifnum\i>1
				\draw (\x,\n) -- (\x,\y-0.5);
			\else
				\draw[gray,very thick] (-0.5,-0.5) rectangle (\x+0.5,\n+0.5);
			\fi
		\fi
		\ifnum\pin=2 % down		
			\pgfmathparse{\s-2}
 			\xdef\s{\pgfmathresult}
			\absdot{(\x,\s)}{};
			\ifnum\i>1
				\draw (\x,\s) -- (\x,\y+0.5);
			\else
				\draw[gray,very thick] (-0.5,0.5) rectangle (\x+0.5,\s-0.5);
			\fi
		\fi
		\ifnum\pin=3 %right
			\pgfmathparse{\e+2}
 			\xdef\e{\pgfmathresult}
			\absdot{(\e,\y)}{};
			\ifnum\i>1
				\draw (\e,\y) -- (\x-0.5,\y);
			\else
				\draw[gray,very thick] (-0.5,+0.5) rectangle (\e+0.5,\y-0.5);
			\fi
		\fi
		\ifnum\pin=4 %left
			\pgfmathparse{\w-2}
 			\xdef\w{\pgfmathresult}
			\absdot{(\w,\y)}{};
			\ifnum\i>1
				\draw (\w,\y) -- (\x+0.5,\y);
			\else
				\draw[gray,very thick] (0.5,0.5) rectangle (\w-0.5,\y-0.5);

			\fi
		\fi		
	};
}

\tikzset{
  on each segment/.style={
    decorate,
    decoration={
      show path construction,
      moveto code={},
      lineto code={
        \path [#1]
        (\tikzinputsegmentfirst) -- (\tikzinputsegmentlast);
      },
      curveto code={
        \path [#1] (\tikzinputsegmentfirst)
        .. controls
        (\tikzinputsegmentsupporta) and (\tikzinputsegmentsupportb)
        ..
        (\tikzinputsegmentlast);
      },
      closepath code={
        \path [#1]
        (\tikzinputsegmentfirst) -- (\tikzinputsegmentlast);
      },
    },
  },
  mid arrow/.style={postaction={decorate,decoration={
        markings,
        mark=at position .5 with {\arrow[#1]{stealth}}
      }}},
}

\newtheorem{thm}{Theorem}[section]
\newtheorem*{thm*}{Theorem}
\newtheorem{prop}[thm]{Proposition}
\newtheorem{cor}[thm]{Corollary}
\newtheorem{lemma}[thm]{Lemma}

\newtheorem{obs}[thm]{Observation}

\newtheorem*{conj*}{Conjecture}

\newtheorem{procedure}[thm]{Procedure}

\theoremstyle{definition}
\newtheorem{defn}[thm]{Definition}
\newtheorem{defn*}{Definition}
\newtheorem{example}[thm]{Example}
\newtheorem*{example*}{Example}
\newtheorem{comment}[thm]{Comment}
\newtheorem*{comment*}{Comment}

\let\C\CCC

\newcommand{\gr}{\mathrm{gr}}

\newcommand{\tikzcircle}[2][black,fill=black]{\tikz[baseline=-0.5ex]\draw[#1,radius=#2] (0,0) circle ;}

\newcommand{\fitellipsis}[2] % first and second node names without parentheses
{\draw let \p1=(#1), \p2=(#2), \n1={atan2(\y2-\y1,\x2-\x1)}, \n2={veclen(\y2-\y1,\x2-\x1)}
    in ($ (\p1)!0.5!(\p2) $) ellipse [x radius=\n2/2+0.5cm, y radius=0.5cm, rotate=\n1];
}

\tikzstyle{vertex}=[circle, draw, fill=black,
                        inner sep=0pt, minimum width=4pt]

\tikzstyle{vertex}=[circle, draw, fill=black,
                        inner sep=0pt, minimum width=4pt]

\newcommand{\smallgrid}{\setplotptradius{2pt}\draw (-.5,0)--++(1,0) (0,-.5)--++(0,1);}
\newcommand{\sept}{\tikz[scale=.5,baseline=-3pt]{\smallgrid\node[permpt] at (.25,-.25) {};}}
\newcommand{\swpt}{\tikz[scale=.5,baseline=-3pt]{\smallgrid\node[permpt] at (-.25,-.25) {};}}
\newcommand{\nept}{\tikz[scale=.5,baseline=-3pt]{\smallgrid\node[permpt] at (.25,.25) {};}}
\newcommand{\nwpt}{\tikz[scale=.5,baseline=-3pt]{\smallgrid\node[permpt] at (-.25,.25) {};}}
\newcommand{\netwo}{\tikz[scale=.5,baseline=-3pt]{\smallgrid\node[permpt] at (.15,.35) {};\node[permpt] at (.35,.15) {};}}
\newcommand{\nwtwo}{\tikz[scale=.5,baseline=-3pt]{\smallgrid\node[permpt] at (-.15,.35) {};\node[permpt] at (-.35,.15) {};}}
\newcommand{\swtwo}{\tikz[scale=.5,baseline=-3pt]{\smallgrid\node[permpt] at (-.15,-.35) {};\node[permpt] at (-.35,-.15) {};}}

\newcommand{\plottwocellperm}[2][black]{ %[colour]{permutation}
    \foreach \i [count=\nn] in {#2} {\global\let\n\nn};
	\foreach \j [count=\i] in {#2} { %The number 0 draws a vertical bar
        \ifnum0=\j {%
          \draw (\i,.5) -- ++(0,\n-1);   
        } \else {%
 		  \node[permpt,fill=#1,draw=#1] (\j) at (\i,\j) {};
	    } \fi
	};
}

\newcommand{\twocell}[1]{\tikz[scale=.4,baseline=-3pt]{\setplotptradius{2pt}
\foreach \i [count=\nn] in {#1} {\global\let\n\nn};
\begin{scope}[yscale=1/(\n-1),xscale=1/3]
	\foreach \j [count=\i] in {#1} { %The number 0 draws a vertical bar
        \ifnum0=\j {%
          \draw (\i,.5-\n/2) -- ++(0,\n-1);   
        } \else {%
 		  \node[permpt] (\j) at (\i,\j-\n/2) {};
	    } \fi
	};
\end{scope}
}}

\title{Pin Classes I: Growth Rates}

\author{Ben Jarvis\footnote{bjarvis93@outlook.com}\\
\small School of Mathematics and Statistics\\
\small The Open University, UK
}

\begin{document}
\maketitle

\begin{abstract}
Pin sequences play an important role in the structural study of permutation classes. In this paper, we study the permutation classes that comprise all the finite subpermutations contained in an infinite pin sequence. We prove that these permutation classes have proper growth rates and establish a procedure for calculating these growth rates.
\end{abstract}

Since their introduction by Brignall, Huczynska and Vatter~\cite{brignall:decomposing-sim:} in relation to the simple permutations, the structures known as pin sequences have appeared in many varied and unexpected places throughout the field of permutation patterns, notably in connection with substitution-closed classes and infinite antichains. Additionally, pin sequences, along with the \emph{pin permutation classes} they naturally define, have proven to be a very efficient means of producing counterexamples, notably being used recently by Brignall and Vatter~\cite{bv:wqo-uncountable:} to demonstrate the existence of well-quasi-ordered permutation classes with non-algebraic generating functions. The study of pin classes as objects of interest in their own right, on the other hand, has been somewhat limited, though precedent certainly exists -- most notably, Bassino, Bouvel and Rossin~\cite{bassino:enumeration-of-:} showed that the class of all pin permutations has a rational generating function and established its growth rate $\omega_{\infty} \approx 5.24112$. In this paper we take up the study of pin classes in a more systematic manner than has previously been attempted. Our main result is that pin classes have \emph{proper} growth rates, not merely upper growth rates as guaranteed by the Marcus-Tardos Theorem~\cite{marcus:excluded-permut:}. We also describe a procedure for finding the growth rate of a given pin class in terms of the recurrent subfactors of the defining pin sequence. 

We begin in Section \ref{sec:1} by collecting the basic notions and results from the field of permutation patterns which we shall require going forward, before concluding with a brief and informal introduction to pin sequences in connection with the simple permutations. In Section \ref{sec:2} we introduce the notion of a \emph{centred permutation}: a permutation along with a defined origin point. This will transpire to be the most natural context in which to situate the study of pin classes, chiefly because it will allow us to utilise the $\boxplus$-sum, an operation on the centred permutations. We begin to study pin permutation classes in earnest in Section \ref{sec:3}, formally defining pin words and the (centred and uncentred) permutation classes they generate. We shall also derive a particuarly nice structure theorem for \emph{recurrent} pin classes (those generated by recurrent pin words) in terms of the $\boxplus$-sum, which will give us an extraordinarily large family of permutation classes whose generating functions we can find by solving a simple combinatorial problem concerning subfactors of the defining pin word. In Section \ref{sec:4} we drop the recurrency condition and deduce the major result of this paper: that every pin permutation class has a proper growth rate. We also describe an effective procedure for determining the growth rate (but not the generating function) of any given pin class by reducing to a simpler combinatorial problem on words and conclude in Section \ref{sec:5} by stating some potential directions for further study.

\tableofcontents

\section{Basic Notions} \label{sec:1}

In this section we collect some of the basic notions and key results from the field of permutation patterns, which we use freely throughout, along with an informal introduction to pin sequences and pin permutations. We refer the reader to Bevan~\cite{bevan2015defs} for background information concerning permutations and permutation classes.

\subsection{Permutations and Permutation Classes} \label{sec:1.1}

A permutation is a bijection from a finite set to itself. As we are generally not interested in the nature of the objects we are permuting, we may take our formal definition to be the following:

\begin{defn}[Permutations]
A \emph{permutation} $\pi$ is a bijection from $\left\{1,2,3,\dots,n\right\}$ to itself, where $n \in \mathbb{N}$. We say that $n$ is the \emph{length} of the permutation $\pi$.
\end{defn}
As a short-hand way of writing down specific permutations we adopt \emph{one-line notation}, in which the images of each number are written out consecutively:
\[
\pi(1)\pi(2)\dots\pi(n).
\]
For example, $\pi = 13524$ refers to the permutation
\[
\pi : \{1,2,3,4,5\} \rightarrow \{1,2,3,4,5\}
\]
defined by $\pi(1)=1,\pi(2)=3,\pi(3)=5,\pi(4)=2,\pi(5)=4$.

One-line notation suggests a natural visual representation of a permutation of length $n$ on an $n$-by-$n$ grid --- see Figure \ref{fig:visrep}.

\begin{figure}[h!]
\begin{center}
\begin{tikzpicture}[scale=0.35]
\draw[step=1cm,gray,very thin] (1,1) grid (5,5);
\foreach \y [count=\x] in {1,3,5,2,4}
{
	\node[permpt] (\y) at (\x,\y) {};
	\node[fill=none,draw=none] at (\x,0) {\y};
}

\end{tikzpicture}
\end{center}
\caption{The standard visual representation of the permutation $\pi = 13524$ on a $5$-by-$5$ grid. $\pi$ maps $1$ to $1$, so there is a point in the grid with coordinates $(1,1)$; $\pi$ maps $2$ to $3$, so there is a point with coordinates $(2,3)$, etc.}
\label{fig:visrep}
\end{figure}

Conversely, we can interpret any finite collection of points in the Cartesian plane as a permutation, as long as no two points share the same $x$- or $y$- coordinate. Figure \ref{fig:visrep2} illustrates an example.

\begin{figure}[h]
\begin{center}
\begin{tikzpicture}[scale=0.35]

\node[permpt] (6) at (1,4) {};
         \node[permpt] (7) at (2,3) {};
         \node[permpt] (8) at (4,6) {}; 
         \node[permpt] (9) at (7,7) {};
         \node[permpt] (10) at (8,1) {};

\node[fill=none,draw=none] at (1,0) {3};
        \node[fill=none,draw=none] at (2,0) {2};
        \node[fill=none,draw=none] at (4,0) {4};
        \node[fill=none,draw=none] at (7,0) {5};
        \node[fill=none,draw=none] at (8,0) {1};

\end{tikzpicture}
\end{center}
\caption{The above collection of points (no pair of which shares either an x- or y-coordinate) can be interpreted as representing the permutation $32451$. We don't care about the distance between points, only their relative orderings horizontally and vertically.}
\label{fig:visrep2}
\end{figure}

There is a particularly natural containment order defined on the set of all permutations:
\begin{defn}[Permutation Containment]
Suppose that $\pi$ is a permutation of length $n$, so that, in one-line notation,
\[
\pi = \pi(1) \pi(2) \cdots \pi(n).
\]
Suppose further that $\sigma$ is a permutation of length $k \leq n$. We say that $\sigma$ is \emph{contained} in $\pi$, denoted $\sigma \leq \pi$, if there exists a sequence of indices $1 \leq i_{1} \leq i_{2} \leq \dots \leq i_{k} \leq n$ such that the sequence $\pi(i_{1})\pi(i_{2}) \dots \pi(i_{k})$ has the same relative ordering as $\sigma(1)\sigma(2) \dots \sigma(k)$. Similarly, $\sigma$ is \emph{strictly contained} in $\pi$, denoted $\sigma < \pi$, if $\sigma \leq \pi$ and $\sigma \neq \pi$.

Conversely, $\pi$ \emph{avoids} $\sigma$ if $\sigma\nleq \pi$.
\end{defn}
Permutation containment is much easier to understand using the visual representation introduced earlier: $\sigma \leq \pi$ if and only if a representation of $\pi$ as a collection of points in the plane contains a representation of $\sigma$ as a subset --- see Figure \ref{fig:permcont}.

\begin{figure}[h]
\begin{center}
\begin{tikzpicture}[scale=0.35]
%\draw[step=1cm,gray,very thin] (1,1) grid (3,3);
\foreach \y [count=\x] in {2,1,3}
{
	\node[permpt] (\y) at (\x,\y) {};
	\node[fill=none,draw=none] at (\x,-1) {\y};
}

\foreach \y [count=\x] in {2,1,3}
{
	
	\node[fill=none,draw=none] at (\x,-1) {\y};
}

\node at (5,-1) {$\leq$};
\begin{scope}[shift={(6,0)}]
%\draw[step=1cm,gray,very thin] (1,1) grid (5,5);
\foreach \y [count=\x] in {1,3,5,2,4}
{
	\node[permpt] (\y) at (\x,\y) {};
	\node[fill=none,draw=none] at (\x,-1) {\y};
}

\foreach \y [count=\x] in {1,3,5,2,4}
{
	
	\node[fill=none,draw=none] at (\x,-1) {\y};
}

\node[fill=none,draw=none,blue] at (2,-1) {3};
\node[fill=none,draw=none,blue] at (4,-1) {2};
\node[fill=none,draw=none,blue] at (5,-1) {4};

\node at (7,-1) {$\leq$};

\foreach \y in {3,2,4}
	\node[permpt,blue] at (\y) {};
\end{scope}
\begin{scope}[shift={(14,0)}]
%\draw[step=1cm,gray,very thin] (1,1) grid (8,8);
\foreach \y [count=\x] in {4,1,2,6,3,8,5,7}
{
	\node[permpt] (\y) at (\x,\y) {};
	\node[fill=none,draw=none] at (\x,-1) {\y};
}

\foreach \y [count=\x] in {4,1,2,6,3,8,5,7}
{
	
	\node[fill=none,draw=none] at (\x,-1) {\y};
}

\node[fill=none,draw=none,orange] at (2,-1) {1};
\node[fill=none,draw=none,orange] at (4,-1) {6};
\node[fill=none,draw=none,orange] at (6,-1) {8};
\node[fill=none,draw=none,orange] at (7,-1) {5};
\node[fill=none,draw=none,orange] at (8,-1) {7};

\foreach \y in {1,6,8,5,7}
	\node[permpt,orange] at (\y) {};
\end{scope}
\end{tikzpicture}
\end{center}
\caption{The permutation $213$ is contained in $13524$ (as the blue points form another copy of $213$); in turn, $13524$ is contained in $41263857$ (as the orange points form another copy of $13524$).}
\label{fig:permcont}
\end{figure}

Permutation containment allows us to define the primary objects of study in the field of permutation patterns.

\begin{defn}[Permutation Classes]
A \emph{permutation class} $\mathcal{C}$ is a set of permutations which is closed under permutation containment. That is:
\[
\left(\pi \in \C\text{ and }\sigma\leq \pi\right)\text{ implies }\sigma\in \C \ .
\]
A permutation class is \emph{proper} if it does not contain all permutations.
\end{defn}

Given a permutation class we are primarily concerned with what we can say about its enumeration, meaning that we wish to study its enumeration sequence, defined as follows.

\begin{defn} \label{enumseqforpermclass}
Suppose $\mathcal{C}$ is a permutation class, and let $\mathcal{C}_{n}$ denote the set of permutations of length $n$ in $\mathcal{C}$. We denote
\[
C_{n} = |\mathcal{C}_{n}| \ \ \ \text{for all $n \in \mathbb{N}$,}
\]
and refer to this sequence as the \emph{enumeration sequence} of $\mathcal{C}$.

The \emph{generating function} of $\mathcal{C}$ is then defined to be the formal power series
\[
f(z) = 1 + C_{1}z + C_{2}z^2 + C_{3}z^3 + C_{4}z^4 + \dots
\]
\end{defn}

In practice, when we say we have enumerated a permutation class we mean that we have obtained a closed form for its generating function, not necessarily that we have an $n$th term formula for its enumeration sequence. This focus on generating functions is partly justified by the fact that we can read off asymptotic information about the enumeration sequence from its generating function. Specifically, studying analytic properties of the generating function of a permutation class will enable us to determine its \emph{(upper and lower) growth rates}.

\begin{defn}[Growth Rates] \label{def:growthrates}
Let $\mathcal{C}$ be a permutation class with enumeration sequence $(C_{n})$. The \emph{upper} and \emph{lower growth rates} of $\mathcal{C}$ are defined, respectively, to be 
\[
\overline{gr}(\mathcal{C}) = \limsup_{n\to\infty}\sqrt[n]{C_n}, \quad\text{ and }\quad \underline{gr}(\mathcal{C})=\liminf_{n\to\infty}\sqrt[n]{C_n}.
\]

If both the upper and lower growth rates of $\mathcal{C}$ are finite and
\[
\overline{gr}(\mathcal{C})=\underline{gr}(\mathcal{C}),
\]
then we say that $\mathcal{C}$ has a \emph{(proper) growth rate} $\gr(\mathcal{C})$, denoting this common value. Equivalently,
\[
gr(\mathcal{C}) = \lim_{n \rightarrow \infty}\sqrt[n]{C_{n}}
\]
if such a quantity exists.
\end{defn}

\begin{comment}[Growth rate respects subclass containment]
We note that this formulation of the growth rate immediately implies that growth rates respect subclass containment, a fact to which we refer freely throughout. To be precise, if $\mathcal{C}$ is a permutation class and $\mathcal{D}$ is a subclass (meaning that $\mathcal{D}$ is a subset of $\mathcal{C}$ which is itself a permutation class) then the (upper, lower, proper) growth rate of $\mathcal{D}$ is less than or equal to the (upper, lower, proper) growth rate of $\mathcal{C}$. This is of course immediate from the fact that the enumeration sequence of $\mathcal{D}$ is dominated by the enumeration sequence of $\mathcal{C}$.
\end{comment}

Suppose $\mathcal{C}$ is a permutation class. By basic real analysis the quantities $\overline{gr}(\mathcal{C})$ and $\underline{gr}(\mathcal{C})$ both exist in $[0,\infty]$. It is certainly possible for both to be infinite: the class of all permutations grows factorially, and therefore has $\overline{gr}(\mathcal{C}) = \underline{gr}(\mathcal{C}) = \infty$. Relatively early in the study of permutation classes it was noted, independently by Stanley and Wilf amongst others, that this appears to be the only instance of a permutation class with infinite upper or lower growth rate. This observation motivated the famous \emph{Stanley-Wilf conjecture}, first stated in the early $1990$s, that every \emph{proper} permutation class has finite upper (and hence also lower) growth rate. This conjecture remained open until $2004$ when Marcus and Tardos settled it in the affirmative:

\begin{thm}[The Marcus-Tardos Theorem~\cite{marcus:excluded-permut:}]
Suppose that $\mathcal{C}$ is a proper permutation class. Then there is some constant $\mu > 0$ for which
\[
C_{n} < \mu^{n},
\]
for all $n \in \mathbb{N}$.

Hence $\overline{gr}(\mathcal{C})$ and $\underline{gr}(\mathcal{C})$ are both finite.
\end{thm}

The Marcus-Tardos Theorem is a truly remarkable result: essentially it tells us that if we throw \emph{any} element out of the class of all permutations then the containment condition forces us to throw \emph{almost everything} out, as exponential growth to any base is negligible compared to factorial growth. On the other hand, the Marcus-Tardos Theorem only classifies the asymptotics of permutation classes in the broadest possible sense, and leaves open many questions we could ask about the finer asymptotics. Most notably, Marcus-Tardos does not guarantee that the upper and lower growth rates of any given proper permutation class are equal, though this is true in all cases which have been checked. The question of whether every proper permutation class has a proper growth rate remains one of the major open problems in the field of permutation patterns. One of the major results of this paper, presented in Theorem \ref{greqbox}, is that every member of a large family of permutation classes called the \emph{pin permutation classes} does indeed have a proper growth rate. It is our hope that this result may shed some light on the more general problem.

\subsection{The Exponential Growth Theorem} \label{sec:1.2}

As mentioned in the previous section, one of the primary reasons we are interested in calculating generating functions of permutation classes is that we can easily read asymptotic information about the enumeration sequence $(C_{n})$ of a permutation class $\mathcal{C}$ from its generating function $f(z)$. We state a result derived by Flajolet \& Sedgewick in consequence of Pringsheim's Theorem, perhaps the most important analytic tool used in this paper:

\begin{thm}[Exponential Growth Theorem, {Flajolet and Sedgewick~\cite{flajolet:analytic-combin:}}]\label{EGT0}
Suppose that $f(z)$ is the generating function of a combinatorial class\footnote{A \emph{combinatorial class} is a set $\mathcal{C}$ along with a length function $\varphi: \mathcal{C} \rightarrow \mathbb{N}_{0}$ for which the preimage of any element of $\mathbb{N}_{0}$ is finite. A permutation class is of course a type of combinatorial class, as is any set of permutations. The enumeration sequence and generating function of a general combinatorial class is defined analogously to Definition \ref{enumseqforpermclass}.} $\mathcal{C}$, and that $f(z)$ is analytic at $0$. Then the upper growth rate of $\mathcal{C}$ is equal to the reciprocal of the radius of convergence of $f(z)$. 
\end{thm}

We refer to Flajolet \& Sedgewick~\cite[Theorem IV.7]{flajolet:analytic-combin:} for a formal proof of this fact, but the intuition here is fairly basic: the faster the coefficients of a power series grow the smaller we expect the radius of convergence to be, and we can make this relationship explicit by comparison with the geometric series
\[
1 + \lambda z + \lambda^{2}z^{2} + \lambda^{3}z^{3} + \lambda^{4}z^{4} + \dots = \frac{1}{1 - \lambda z}
\]
whose radius of convergence is the smallest root of its denominator, namely $\lambda^{-1}$. Thus if a power series $f(z)$ can be shown to have radius of convergence $\lambda^{-1}$ we expect its coefficients to, in some limited sense, grow `like powers of $\lambda$'.

We now specialise Theorem \ref{EGT0} to the combinatorial classes we are most interested in:

\begin{thm}[Exponential Growth Theorem for Permutation Classes]\label{EGT1}
Suppose that $f(z)$ is the generating function of a proper permutation class $\mathcal{C}$. Then the upper growth rate of $\mathcal{C}$ is equal to the reciprocal of 
\[
R = \sup\left\{r \geq 0 \mid \ \text{$f(z)$ converges on $[0,r)$}\right\} \ .
\]

Furthermore, if
\[
f(z) = \frac{a(z)}{b(z)}
\]
is a rational function, where the polynomials $a(z)$ and $b(z)$ share no polynomial factor, then the upper growth rate of $\mathcal{C}$ is equal to the reciprocal of the smallest positive real root of $b(z)$.
\end{thm}

\begin{proof}
We can apply Theorem \ref{EGT0} in this instance as Marcus-Tardos tells us that the coefficients of $f(z)$ are bounded by powers of some constant $\lambda > 0$. Hence, by the Ratio Test, $f(z)$ converges to an analytic function on the unit disc of radius $\lambda^{-1}$. As the coefficients of $f(z)$ are non-negative, the singularity of $f(z)$ with smallest absolute value must occur on the positive real axis, so nothing changes by defining $R$ as in the statement of the theorem. If we further assume that $f(z)$ is a rational function then the only possible singularities are at the zeroes of the denominator, as required. \qedhere
\end{proof}

Theorem \ref{EGT1} (as well as its analogue for centred permutation classes stated as Theorem \ref{EGT}) will be the primary means by which we find (upper) growth rates of the classes we are interested in throughout this paper.

\subsection{Intervals} \label{sec:1.3}

\begin{defn}[Intervals and Simple Permutations] \label{def:uncentredinvtervals}
An \emph{interval} of a permutation
\[
\pi \ = \pi(1) \ \pi(2) \ \dots \ \pi(n) \ ,
\]
is a set of consecutive indices $k, k+1, \dots , \ell$ such that the set of values $\{\pi(i) : k \leq i \leq \ell\}$ is also contiguous. Visually, an interval in a
permutation is a bounding rectangle that is not sliced horizontally or vertically by any point not in it --- see Figure \ref{fig:interval}.

In any permutation the set of all points is an interval, as are the singleton sets consisting of only one point each. A permutation which contains no intervals other than these is called \emph{simple}.
\end{defn}

\begin{figure}[h]
\begin{center}
\begin{tikzpicture}[scale=0.35]

%\node[circle, draw, fill=none, inner sep=0pt, minimum width=\plotptradius] (0) at (7,7) {};
\node[permpt] (1) at (8,9) {};
\node[permpt] (2) at (5,8) {};
\node[permpt] (3) at (6,5) {};
\node[permpt] (4) at (4,6) {};
\node[permpt] (1) at (3,3) {};
\node[permpt] (2) at (1,4) {};
\node[permpt] (3) at (2,1) {};
\node[permpt] (4) at (9,11) {};
\node[permpt] (4) at (10,10) {};

%\draw[thin] (7,0.5) -- ++ (0,12);
%\draw[thin] (0.5,7) -- ++ (12,0);

\draw[dotted] (8.5,4.5) -- ++ (0,5);
\draw[dotted] (3.5,4.5) -- ++ (0,5);
\draw[dotted] (3.5,9.5) -- ++ (5,0);
\draw[dotted] (3.5,4.5) -- ++ (5,0);

%\draw[dotted] (3.5,4.5) rectangle (8.5,9.5);

\end{tikzpicture}
\end{center}
\caption{In the permutation $312564798$ the contiguous block $5647$ consists of four contiguous integers, so this forms an interval --- as can be seen by the fact that the bounding rectangle is not sliced by any outside point in either the horizontal or vertical direction. This permutation is therefore not simple.}
\label{fig:interval}
\end{figure}

The existence and number of simple permutations in a permutation class (or in its \emph{basis}, the set of permutations it avoids which are minimal with respect to this property) has important structural consequences for the class as a whole. It may easily be demonstrated, for example, that a permutation class is substitution-closed if and only if its basis consists only of simple permutations. Moreover, Albert \& Atkinson~\cite{albert:simple-permutat:} proved that a permutation class containing only finitely-many simple permutations must have a finite basis as well as an algebraic generating function which may be computed via a recursive procedure. 

\subsection{Pin Sequences} \label{sec:1.4}

When presented with a new permutation class $\mathcal{C}$, then, one of the most important structural properties of interest is whether $\mathcal{C}$ contains only finitely-many simple permutations --- if it does, and if we can list these simples, then Albert \& Atkinson provide a scheme to enumerate the class. The search for a decision procedure for determining whether a given permutation class $\mathcal{C}$ contains finitely- or infinitely-many simples therefore became an object of research, and it was this problem that first motivated the construction that will be the main focus of this paper: \emph{pin permutations}.

\begin{defn}[Pin Sequences and Permutations]
A \emph{pin word} $w$ is a finite or infinite word over the alphabet $\mathbb{A}_{P} = \left\{1,2,3,4,u,d,l,r\right\}$ which obeys the following construction rules:
\begin{enumerate}
\item The first letter is one of the numerals $1$, $2$, $3$ or $4$.
\item The second letter is one of the letters $u$, $d$, $l$, $r$.
\item If the $n$th letter of $w$ is in $\left\{u,d\right\}$ then the $(n+1)$-th letter (if it exists) is in $\left\{l,r\right\}$.
\item If the $n$th letter of $w$ is in $\left\{l,r\right\}$ then the $(n+1)$-th letter (if it exists) is in $\left\{u,d\right\}$.
\end{enumerate}

A finite pin word $w$ can be converted into a set of points in the plane, known as a \emph{proper pin sequence}, by the following procedure (see Figure \ref{fig:example} for an illustration of this process):

\begin{enumerate}
\item Choose a point in the plane to be the origin. By drawing horizontal and vertical axes through this point, split the plane into four quadrants, numbered $1$ to $4$ anti-clockwise from the top-right.
\item Place an initial point in the quadrant specified by the initial numeral of $w$.
\item At each subsequent step, place a point either up, down, left or right (depending on the letter u, d, l, or r in the relevant position of $w$) of the bounding rectangle of all previous points and the origin at the end of a 'pin' which separates the previous point from all points prior.
\end{enumerate}

We may then ignore the grid-lines, the origin and the pins to read off the permutation given by the points. This is the \emph{pin permutation generated by $w$}. More generally, a \emph{pin permutation} is any permutation which is contained in a permutation generated by some finite pin word $w$ through this process.

\begin{figure}[h]
\begin{center}
\begin{tikzpicture}[scale=0.5]
\node[circle, draw, fill=none, inner sep=0pt, minimum width=\plotptradius] (0) at (6,4) {}; 
\node[permpt,label={\small 1}] (1) at (5,6) {};
\node[permpt,label={\small 2}] (2) at (3,5) {}; \draw[thin] (2) -- ++ (2.5,0); 
\node[permpt,label={\small 3}] (3) at (4,8) {}; \draw[thin] (3) -- ++ (0,-3.5);
\node[permpt,label={\small 4}] (4) at (8,7) {}; \draw[thin] (4) -- ++ (-4.5,0);
\node[permpt,label={[yshift=-18pt]{\small 5}}] (5) at (7,2) {}; \draw[thin] (5) -- ++ (0,5.5);
\node[permpt,label={\small 6}] (6) at (1,3) {}; \draw[thin] (6) -- ++ (6.5,0);
\node[permpt,label={[yshift=-18pt]{\small 7}}] (7) at (2,1) {}; \draw[thin] (7) -- ++ (0,2.5);

\draw[thick] (0,4) -- ++ (9,0);
\draw[thick] (6,0) -- ++ (0,9);
\end{tikzpicture}
\end{center}
\caption{The permutation $3147526$, constructed from the pin word $2lurdld$. The numbers refer to the order in which the points are placed: the first point is placed in quadrant $2$ (due to the $2$ at the start of the pin word); then, the second point is placed to the left (due to the $l$) of the bounding rectangle of the first point and the origin, at the end of a pin separating point $1$ from the origin; next, point $3$ is placed above (or 'up', due to the u) the bounding rectangle of the first two points and the origin, at the end of a pin separating point $2$ from point $1$ and the origin; and so on...}
\label{fig:example}
\end{figure}
\end{defn}

Pin sequences were first introduced by Brignall, Huczynska \& Vatter~\cite{brignall:decomposing-sim:}. The connection with simple permutations may be appreciated immediately: the definition of a proper pin sequence seems almost designed to produce a simple permutation, with each point placed in such a way as to ensure that the previously placed points do not form an interval. This is not in fact quite true --- see Figure \ref{fig:nonsimplepinperm} for a pin word which generates a non-simple permutation --- but Brignall \emph{et al.} did show that every permutation generated by a sufficiently long finite pin word is either simple \emph{or} can be made simple by removing a single point. More importantly, they demonstrated a partial converse: every sufficiently long simple permutation contains either a long pin sequence or a long permutation in one of two families called parallel alternations and wedge simples. This result provides a procedure for determining whether a given permutation class contains only finitely-many simples, and thus establishes pin sequences as a crucial tool in the structural study of permutation classes.

\begin{figure}[h!]
\begin{center}
\begin{tikzpicture}[scale=0.35]

\node[circle, draw, fill=none, inner sep=0pt, minimum width=\plotptradius] (0) at (7,7) {}; 
\node[permpt] (1) at (8,9) {};
\node[permpt] (2) at (5,8) {}; \draw[thin] (2) -- ++ (3.5,0); 
\node[permpt] (3) at (6,5) {}; \draw[thin] (3) -- ++ (0,3.5);
\node[permpt] (4) at (3,6) {}; \draw[thin] (4) -- ++ (3.5,0);
\node[permpt] (5) at (4,3) {}; \draw[thin] (5) -- ++ (0,3.5);
\node[permpt] (6) at (1,4) {}; \draw[thin] (6) -- ++ (3.5,0);
\node[permpt] (7) at (2,1) {}; \draw[thin] (7) -- ++ (0,3.5);
\node[permpt] (8) at (10,2) {}; \draw[thin] (8) -- ++ (-8.5,0);
\node[permpt] (9) at (9,10) {}; \draw[thin] (9) -- ++ (0,-8.5);

\draw[thick] (0,7) -- ++ (11,0);
\draw[thick] (7,0) -- ++ (0,11);

\node at (16,7) {$\longrightarrow$};

\begin{scope}[shift={(21,0)}]

%\node[circle, draw, fill=none, inner sep=0pt, minimum width=\plotptradius] (0) at (7,7) {}; 
\node[permpt] (1) at (8,9) {};
\node[permpt] (2) at (5,8) {}; %\draw[thin] (2) -- ++ (3.5,0); 
\node[permpt] (3) at (6,5) {}; %\draw[thin] (3) -- ++ (0,3.5);
\node[permpt] (4) at (3,6) {}; %\draw[thin] (4) -- ++ (3.5,0);
\node[permpt] (5) at (4,3) {}; %\draw[thin] (5) -- ++ (0,3.5);
\node[permpt] (6) at (1,4) {}; %\draw[thin] (6) -- ++ (3.5,0);
\node[permpt] (7) at (2,1) {}; %\draw[thin] (7) -- ++ (0,3.5);
\node[permpt] (8) at (10,2) {}; %\draw[thin] (8) -- ++ (-8.5,0);
\node[permpt] (9) at (9,10) {}; %\draw[thin] (9) -- ++ (0,-8.5);

\draw[dashed] (7.5,8.5) rectangle (9.5,10.5);

\end{scope}

\end{tikzpicture}
\end{center}
\caption{The proper pin sequence $1ldldldru$ generates the permutation $416375892$. This is not a simple permutation due to the proper interval contained in the marked rectangle, but note that removing either one of the two points in this rectangle results in a simple permutation.}
\label{fig:nonsimplepinperm}
\end{figure}

Though pin sequences were initially introduced in connection with the simple permutations, they have since cropped up in various places, suggesting a deeper importance. Shortly after their introduction, Brignall~\cite{brignall:wreath-products:} used pin sequences to investigate preservation of the finite basis property under the \emph{wreath product}\footnote{Usually referred to as the \emph{inflation} of one permutation class by another in more recent publications.}. Given two permutation classes $\mathcal{C}$ and $\mathcal{D}$, the wreath product $\mathcal{C}[\mathcal{D}]$ is the permutation class consisting of permutations in $\mathcal{C}$ with any individual point replaced with a copy of some permutation in $\mathcal{D}$. Since the introduction of the wreath product by Atkinson \& Stitt~\cite{atkinson:restricted-perm:a} it has been known that this operation does not in general preserve finite bases: Atkinson \& Stitt give an example of two finitely-based permutation classes $\mathcal{C}$ and $\mathcal{D}$ such that $\mathcal{C}[\mathcal{D}]$ has an infinite basis. Brignall was able to demonstrate, however, that the wreath product $\mathcal{C}[\mathcal{D}]$ is finitely-based if both $\mathcal{C}$ and $\mathcal{D}$ are finitely-based \emph{and} $\mathcal{D}$ does not contain arbitrarily long pin sequences.

Another point of contact between pin sequences and the structural study of permutation classes concerns \emph{infinite antichains}. An infinite antichain is an infinite set of permutations which are pairwise-incomparable under the $\leq$-order. The presence of an infinite antichain in a permutation class has significant structural consequences, so much so that a name has been given to those classes which do \emph{not} contain such an antichain: we call these \emph{well-quasi-ordered} permutation classes. The existence of an infinite antichain known as the \emph{antichain of oscillations} in a permutation class with growth rate $\kappa \approx 2.20557$ was used by Vatter~\cite{vatter:small-permutati:} to construct uncountably-many distinct permutation classes at this growth rate, and later work by Bevan~\cite{bevan:intervals} (extending Vatter~\cite{vatter:every-growth-rate:}) used this same antichain to construct permutation classes at \emph{every} growth rate greater than $\lambda_{B} \approx 2.35698$. The connection between infinite antichains and pin sequences is that the \emph{oscillations} (shown in Figure \ref{fig:osc0}), a particular family of permutations used in the construction of the antichain of oscillations, can be thought of as the simplest possible pin permutations, those for which only up- and right-steps are allowed. Seen in this light, pin permutations are generalisations of the oscillations and can also be used to generate infinite antichains by an analogous construction.

With these points in mind we may therefore confidently state that the presence or otherwise of long pin sequences in a given permutation class is a property of considerable interest. Pin sequences have therefore become an object of study in their own right, beginning with a paper of Bassino, Bouvel \& Rossin~\cite{bassino:enumeration-of-:} in which the authors showed that the class $\mathcal{P}$ consisting of \emph{all} pin permutations has a rational generating function, namely
\begin{equation} \label{eq:completepinclass}
f(z) = 1 + \frac{z - 6z^2 + 9z^3 - 12z^4 + 4z^5 + 20z^6 - 8z^7}{1 - 8z + 19z^2 - 26z^3 + 14z^4 - 12z^5 - 8z^6 + 20z^7 - 8z^8} \ ,
\end{equation}
thus confirming a conjecture of Brignall, Ru\v{s}kuc \& Vatter~\cite{brignall:simple-permutat:b}, as well as demonstrating that this class has growth rate approximately equal to $5.24112$. More recently, Brignall and Vatter~\cite{bv:wqo-uncountable:} used pin permutations to construct uncountably many well-quasi-ordered permutation classes with distinct enumeration sequences.

In this paper we take up the study of permutation classes generated by infinite pin words --- so-called \emph{pin classes} --- in a more systematic manner than has previously been attempted. Our main result is that these classes have \emph{proper} growth rates, not merely upper growth rates as guaranteed by Marcus-Tardos. We also describe a procedure for finding the growth rate of such a permutation class in terms of the recurrent subfactors of its defining infinite pin word. It is hoped that, in addition to potential applications to simple permutations and infinite antichains, pin classes prove to be objects of interest in their own right.

\begin{figure}[h]
\begin{center}
\begin{tikzpicture}[scale=0.23]

%\node[circle, draw, fill=none, inner sep=0pt, minimum width=\plotptradius] (0) at (3,0) {};
\node[permpt] (1) at (1,2) {};
%\node[permpt] (1a) at (1.2,0.6) {};
%\node[permpt] (1b) at (0.6,1.2) {};
\node[permpt] (2) at (3,1) {}; \draw[thin] (2) -- ++ (-2.5,0);
\node[permpt] (3) at (2,4) {}; \draw[thin] (3) -- ++ (0,-3.5);
\node[permpt] (4) at (5,3) {}; \draw[thin] (4) -- ++ (-3.5,0);
\node[permpt] (5) at (4,6) {}; \draw[thin] (5) -- ++ (0,-3.5);
%\node[permpt] (6a) at (5.6,6.6) {}; %\draw[thin] (6) -- ++ (0,-3.5);
%\node[permpt] (6b) at (6.2,7.2) {}; \draw[thin] (6.2,6.5) -- ++ (0,-3);
\node[permpt] (6) at (7,5) {}; \draw[thin] (6) -- ++ (-3.5,0);

%\fitellipsis {1a} {1b};

%\fitellipsis {6a} {6b} ;

%\draw[thick] (0,0) -- ++ (8,0);
%\draw[thick] (3,0) -- ++ (0,7);

\node at (10.5,3) {$\leq$};

\begin{scope}[shift={(14,0)}]

%\node[circle, draw, fill=none, inner sep=0pt, minimum width=\plotptradius] (0) at (3,0) {};
\node[permpt] (1) at (1,2) {};
%\node[permpt] (1a) at (1.2,0.6) {};
%\node[permpt] (1b) at (0.6,1.2) {};
\node[permpt] (2) at (3,1) {}; \draw[thin] (2) -- ++ (-2.5,0);
\node[permpt] (3) at (2,4) {}; \draw[thin] (3) -- ++ (0,-3.5);
\node[permpt] (4) at (5,3) {}; \draw[thin] (4) -- ++ (-3.5,0);
\node[permpt] (5) at (4,6) {}; \draw[thin] (5) -- ++ (0,-3.5);
\node[permpt] (6) at (7,5) {}; \draw[thin] (6) -- ++ (-3.5,0);
\node[permpt] (7) at (6,8) {}; \draw[thin] (7) -- ++ (0,-3.5);
\node[permpt] (8) at (9,7) {}; \draw[thin] (8) -- ++ (-3.5,0);
%\node[permpt] (6a) at (5.6,6.6) {}; %\draw[thin] (6) -- ++ (0,-3.5);
%\node[permpt] (6b) at (6.2,7.2) {}; \draw[thin] (6.2,6.5) -- ++ (0,-3);
%\fitellipsis {1a} {1b};

%\fitellipsis {6a} {6b} ;

%\draw[thick] (0,0) -- ++ (8,0);
%\draw[thick] (3,0) -- ++ (0,7);

\node at (12.5,3) {$\leq$};

\end{scope}

\begin{scope}[shift={(30,0)}]

%\node[circle, draw, fill=none, inner sep=0pt, minimum width=\plotptradius] (0) at (3,0) {};
\node[permpt] (1) at (1,2) {};
%\node[permpt] (1a) at (1.2,0.6) {};
%\node[permpt] (1b) at (0.6,1.2) {};
\node[permpt] (2) at (3,1) {}; \draw[thin] (2) -- ++ (-2.5,0);
\node[permpt] (3) at (2,4) {}; \draw[thin] (3) -- ++ (0,-3.5);
\node[permpt] (4) at (5,3) {}; \draw[thin] (4) -- ++ (-3.5,0);
\node[permpt] (5) at (4,6) {}; \draw[thin] (5) -- ++ (0,-3.5);
\node[permpt] (6) at (7,5) {}; \draw[thin] (6) -- ++ (-3.5,0);
\node[permpt] (7) at (6,8) {}; \draw[thin] (7) -- ++ (0,-3.5);
\node[permpt] (8) at (9,7) {}; \draw[thin] (8) -- ++ (-3.5,0);
\node[permpt] (9) at (8,10) {}; \draw[thin] (9) -- ++ (0,-3.5);
\node[permpt] (10) at (11,9) {}; \draw[thin] (10) -- ++ (-3.5,0);
%\node[permpt] (6a) at (5.6,6.6) {}; %\draw[thin] (6) -- ++ (0,-3.5);
%\node[permpt] (6b) at (6.2,7.2) {}; \draw[thin] (6.2,6.5) -- ++ (0,-3);
%\fitellipsis {1a} {1b};

%\fitellipsis {6a} {6b} ;

%\draw[thick] (0,0) -- ++ (8,0);
%\draw[thick] (3,0) -- ++ (0,7);

\node at (14.5,3) {$\leq$};

\node[circle,fill,inner sep=0.5pt] () at (16.5,3) {};
\node[circle,fill,inner sep=0.5pt] () at (17.5,3) {};
\node[circle,fill,inner sep=0.5pt] () at (18.5,3) {};
\node[circle,fill,inner sep=0.5pt] () at (19.5,3) {};

\end{scope}

\begin{scope}[shift={(0,-15)}]

%\node[circle, draw, fill=none, inner sep=0pt, minimum width=\plotptradius] (0) at (3,0) {};
\node[permpt] (1a) at (0.6,1.6) {};
\node[permpt] (1b) at (1.2,2.2) {};
%\node[permpt] (1a) at (1.2,0.6) {};
%\node[permpt] (1b) at (0.6,1.2) {};
\node[permpt] (2) at (3,1) {}; \draw[thin] (2) -- ++ (-2.5,0);
\node[permpt] (3) at (2,4) {}; \draw[thin] (3) -- ++ (0,-3.5);
\node[permpt] (4) at (5,3) {}; \draw[thin] (4) -- ++ (-3.5,0);
\node[permpt] (5) at (4,6) {}; \draw[thin] (5) -- ++ (0,-3.5);
%\node[permpt] (6a) at (5.6,6.6) {}; %\draw[thin] (6) -- ++ (0,-3.5);
%\node[permpt] (6b) at (6.2,7.2) {}; \draw[thin] (6.2,6.5) -- ++ (0,-3);
\node[permpt] (6a) at (6.6,4.6) {}; %\draw[thin] (6) -- ++ (-2.5,0);
\node[permpt] (6b) at (7.2,5.2) {}; \draw[thin] (6.5,5.2) -- ++ (-3,0);

\fitellipsis {1a} {1b};

\fitellipsis {6a} {6b} ;

%\draw[thick] (0,0) -- ++ (8,0);
%\draw[thick] (3,0) -- ++ (0,7);

\node at (10.5,3) {$\nleq$};

\begin{scope}[shift={(14,0)}]

%\node[circle, draw, fill=none, inner sep=0pt, minimum width=\plotptradius] (0) at (3,0) {};
\node[permpt] (1a) at (0.6,1.6) {};
\node[permpt] (1b) at (1.2,2.2) {};
%\node[permpt] (1a) at (1.2,0.6) {};
%\node[permpt] (1b) at (0.6,1.2) {};
\node[permpt] (2) at (3,1) {}; \draw[thin] (2) -- ++ (-2.5,0);
\node[permpt] (3) at (2,4) {}; \draw[thin] (3) -- ++ (0,-3.5);
\node[permpt] (4) at (5,3) {}; \draw[thin] (4) -- ++ (-3.5,0);
\node[permpt] (5) at (4,6) {}; \draw[thin] (5) -- ++ (0,-3.5);
\node[permpt] (6) at (7,5) {}; \draw[thin] (6) -- ++ (-3.5,0);
\node[permpt] (7) at (6,8) {}; \draw[thin] (7) -- ++ (0,-3.5);
\node[permpt] (8a) at (8.6,6.6) {};
\node[permpt] (8b) at (9.2,7.2) {}; \draw[thin] (8.5,7.2) -- ++ (-3,0);
%\node[permpt] (6a) at (5.6,6.6) {}; %\draw[thin] (6) -- ++ (0,-3.5);
%\node[permpt] (6b) at (6.2,7.2) {}; \draw[thin] (6.2,6.5) -- ++ (0,-3);

\fitellipsis {1a} {1b};

\fitellipsis {8a} {8b} ;

%\draw[thick] (0,0) -- ++ (8,0);
%\draw[thick] (3,0) -- ++ (0,7);

\node at (12.5,3) {$\nleq$};

\end{scope}

\begin{scope}[shift={(30,0)}]

%\node[circle, draw, fill=none, inner sep=0pt, minimum width=\plotptradius] (0) at (3,0) {};
\node[permpt] (1a) at (0.6,1.6) {};
\node[permpt] (1b) at (1.2,2.2) {};
%\node[permpt] (1a) at (1.2,0.6) {};
%\node[permpt] (1b) at (0.6,1.2) {};
\node[permpt] (2) at (3,1) {}; \draw[thin] (2) -- ++ (-2.5,0);
\node[permpt] (3) at (2,4) {}; \draw[thin] (3) -- ++ (0,-3.5);
\node[permpt] (4) at (5,3) {}; \draw[thin] (4) -- ++ (-3.5,0);
\node[permpt] (5) at (4,6) {}; \draw[thin] (5) -- ++ (0,-3.5);
\node[permpt] (6) at (7,5) {}; \draw[thin] (6) -- ++ (-3.5,0);
\node[permpt] (7) at (6,8) {}; \draw[thin] (7) -- ++ (0,-3.5);
\node[permpt] (8) at (9,7) {}; \draw[thin] (8) -- ++ (-3.5,0);
\node[permpt] (9) at (8,10) {}; \draw[thin] (9) -- ++ (0,-3.5);
\node[permpt] (10a) at (10.6,8.6) {};
\node[permpt] (10b) at (11.2,9.2) {}; \draw[thin] (10.5,9.2) -- ++ (-3,0);
%\node[permpt] (6a) at (5.6,6.6) {}; %\draw[thin] (6) -- ++ (0,-3.5);
%\node[permpt] (6b) at (6.2,7.2) {}; \draw[thin] (6.2,6.5) -- ++ (0,-3);

\fitellipsis {1a} {1b};

\fitellipsis {10a} {10b} ;

%\draw[thick] (0,0) -- ++ (8,0);
%\draw[thick] (3,0) -- ++ (0,7);

\node at (14.5,3) {$\nleq$};

\node[circle,fill,inner sep=0.5pt] () at (16.5,3) {};
\node[circle,fill,inner sep=0.5pt] () at (17.5,3) {};
\node[circle,fill,inner sep=0.5pt] () at (18.5,3) {};
\node[circle,fill,inner sep=0.5pt] () at (19.5,3) {};

\end{scope}

\end{scope}

\end{tikzpicture}
\end{center}
\caption{The chain of increasing oscillations (above) along with the antichain of increasing oscillations (below). The increasing oscillations are the permutations that can be drawn on 'staircase' diagrams like these. Note that these are actually pin permutations defined by pin words with only up- and right-steps. The permutation class of increasing oscillations, denoted $\mathcal{O}$, is the set of all permutations that can be found somewhere in a diagram in the chain.}
\label{fig:osc0}
\end{figure}

\section{Centred Permutation Classes} \label{sec:2}

We introduce the concept of a  \emph{centred permutation class}, one of the primary combinatorial objects we will deal with in this paper, due to these being in a sense the most natural context in which to study pin permutation classes. In Section \ref{sec:2.1} we define centred permutations: permutations with an identified centre point, known as the \emph{origin}. We show that there exists a natural containment order of the set of centred permutations, which will allow us to define a centred analogy to permutation classes. We shall see that there is a close relationship between the asymptotic enumeration of a centred permutation class and that of its \emph{underlying} class of regular permutations. In Section \ref{sec:2.2} we introduce an operation on centred permutations, the $\boxplus$-sum, somewhat analogous to the $\oplus$-sum for regular permutations. This operation will go some way to explaining our interest in centred permutation classes, as many permutation classes of interest can be conceptualised as the underlying classes of $\boxplus$-closed centred permutation classes. In Section \ref{sec:2.3} we therefore investigate $\boxplus$-closed centred permutation classes, and demonstrate that they can be enumerated with relative ease using the sequent operator. In Section \ref{sec:2.4} we briefly specialise the Exponential Growth Theorem to centred permutation classes, as this will be the main means by which we find the growth rates of classes throughout this paper. We conclude in Section \ref{sec:2.5} with a small number of examples, demonstrating the means by which we may determine the generating functions and growth rates of $\boxplus$-closed centred permutation classes from scratch.

\subsection{Centred Permutations} \label{sec:2.1}

Our main focus in this paper will be a particular method for converting a binary sequence into a permutation class; the resulting permutation classes will be known as \emph{pin classes}. It shall transpire that pin classes are in fact most naturally understood not as classes of permutations in the ordinary sense, but as \emph{centred permutations}. A centred permutation is essentially a permutation with an extra point, designated as the \emph{origin}. A formal definition follows:

\begin{defn}[Centred Permutations]
A centred permutation of length $n$ is a permutation of length $n+1$ in which one point is designated as the origin, which does not contribute to the length of the centred permutation.
\end{defn}

\begin{figure}[h]
\begin{center}
\begin{tikzpicture}[scale=0.35]

\node[circle, draw, fill=none, inner sep=0pt, minimum width=\plotptradius] (0) at (0,0) {};
\node[permpt] (1) at (-2,-1) {}; 
\node[permpt] (2) at (-1,2) {};
\node[permpt] (3) at (1,1) {};

\begin{scope}[shift={(12,0)}]

\node[circle, draw, fill=none, inner sep=0pt, minimum width=\plotptradius] (0) at (0,0) {};
\node[permpt] (1) at (-1,-1) {}; 
\node[permpt] (2) at (1,2) {};
\node[permpt] (3) at (2,1) {};

\end{scope}

\end{tikzpicture}
\end{center}
\caption{Two centred permutations of length $3$. Note that a centred permutation of length $3$ consists of $4$ points (no two of which share either an $x$- or $y$-coordinate) in the plane, of which one is designated as the origin, which does not contribute to the length. Here we denote the origin with an empty circle, whereas the `true' points are represented by solid circles.}
\label{fig:centred0}
\end{figure}

We usually distinguish between centred and uncentred permutations by placing a circle in the superscript: so $\pi$ is a regular permutation and $\pi^{\circ}$ is a centred permutation\footnote{Similarly, when these concepts have been defined, we shall write $\mathcal{C}^{\circ}$ for a class of centred permutations, and  $\mathcal{C}$ for a class of uncentred permutations}.
\begin{comment}
By drawing horizontal and vertical axes through the origin, we can identify a centred permutation with a $2$-by-$2$-gridded permutation\footnote{Informally, this is a permutation along with one horizontal and one vertical line in positions specified relative to the points of the permutation. See ~{Vatter~\cite[Proposition 2.1]{vatter:small-permutati:}} for a formal introduction to gridded permutations.}; we shall usually draw these axes in addition to the origin for clarity. This splits the plane into four \emph{quadrants}, numbered $1$ to $4$ anticlockwise --- see Figure \ref{fig:quadnumbering}. Centred permutations containing points in only one of these four quadrants will play a special role in the theory --- we call these \emph{one-quadrant permutations}.
\end{comment}
\begin{figure}[h]
\begin{center}
\begin{tikzpicture}[scale=0.35]

\node[circle, draw, fill=none, inner sep=0pt, minimum width=\plotptradius] (0) at (0,0) {};
\node at (1,1) {1};
\node at (-1,1) {2};
\node at (-1,-1) {3};
\node at (1,-1) {4};

\draw[thick] (0,-2) -- ++ (0,4);
\draw[thick] (-2,0) -- ++ (4,0);

\end{tikzpicture}
\end{center}
\caption{The standard quadrant numbering: the origin point of a centred permutation splits the plane into four quadrants, numbered anticlockwise.}
\label{fig:quadnumbering}
\end{figure}
Every centred permutation $\sigma^{\circ}$ of length $n$ is naturally associated with an \emph{underlying} permutation $\sigma$ of length $n$, obtained by removing the origin point from $\sigma^{\circ}$. We also have the \emph{filled-in} permutation $\sigma^{\tikzcircle{1.5pt}}$, the permutation of length $n+1$ obtained by replacing the origin of $\sigma^{\circ}$ with a true point. Though we are ultimately studying centred permutations in order to understand their associated underlying permutations, it is the filled-in permutation which provides us with the most convenient notation for a centred permutation:
\begin{defn}[One-Line Notation for Centred Permutations]
Suppose that $\sigma^{\circ}$ is a centred permutation of length $n$, with corresponding filled-in permutation $\sigma^{\tikzcircle{1.5pt}}$ of length $n+1$. Suppose that the $k$th entry of  $\sigma^{\tikzcircle{1.5pt}}$ is the origin of $\sigma^{\circ}$. Then we can write $\sigma^{\circ}$ in one-line notation as follows:
\[
\sigma^{\circ} = \sigma^{\tikzcircle{1.5pt}}(1)\sigma^{\tikzcircle{1.5pt}}(2)\dots\underline{\sigma^{\tikzcircle{1.5pt}}(k)}\dots\sigma^{\tikzcircle{1.5pt}}(n)\sigma^{\tikzcircle{1.5pt}}(n+1)
\]
In other words, this is the one-line notation for $\sigma^{\tikzcircle{1.5pt}}$ with the centre point underlined. See Figure \ref{fig:centred1} for an example.
\end{defn}

\begin{figure}[h]
\begin{center}
\begin{tikzpicture}[scale=0.35]

\node[circle, draw, fill=none, inner sep=0pt, minimum width=\plotptradius] (0) at (4,3) {};
\node[permpt] (1) at (1,4) {}; 
\node[permpt] (2) at (2,2) {};
\node[permpt] (3) at (3,6) {};
\node[permpt] (4) at (5,5) {};
\node[permpt] (5) at (6,1) {};

\draw[thick] (4,0.5) -- ++ (0,6);
\draw[thick] (0.5,3) -- ++ (6,0);

\end{tikzpicture}
\end{center}
\caption{The centred permutation $426\underline{3}51$.}
\label{fig:centred1}
\end{figure}

\begin{comment}
When dealing with centred permutations, the origin does not contribute to the length, but \emph{does} need to be in the same place for two centred permutations to be the same. See Figure \ref{fig:centred2} for an example.
\end{comment}

\begin{figure}[h]
\begin{center}
\begin{tikzpicture}[scale=0.35]

\node[circle, draw, fill=none, inner sep=0pt, minimum width=\plotptradius] (0) at (0,0) {};
\node[permpt] (1) at (-2,-1) {}; 
\node[permpt] (2) at (-1,2) {};
\node[permpt] (3) at (1,1) {};

\draw[thick] (0,-3) -- ++ (0,6);
\draw[thick] (-3,0) -- ++ (6,0);

\node at (6,0) {$\neq$};

\begin{scope}[shift={(12,0)}]

\node[circle, draw, fill=none, inner sep=0pt, minimum width=\plotptradius] (0) at (0,0) {};
\node[permpt] (1) at (-1,-1) {}; 
\node[permpt] (2) at (1,2) {};
\node[permpt] (3) at (2,1) {};

\draw[thick] (0,-3) -- ++ (0,6);
\draw[thick] (-3,0) -- ++ (6,0);

\end{scope}

\end{tikzpicture}
\end{center}
\caption{The centred permutations $14\underline{2}3$ and $1\underline{2}43$ are both of length $3$ (the origin point doesn't count) and both have the same underlying (uncentred) permutation $132$, but they are distinct as centred permutations as the origins are in different relative positions.}
\label{fig:centred2}
\end{figure}
Analogously with the uncentred case we define a notion of centred permutation containment:
\begin{defn}[Centred Permutation Containment]
Suppose that $\sigma^{\circ}$ is a centred permutation of length $n$ and $\pi^{\circ}$ is a centred permutation of length $m \leq n$. We say that $\pi^{\circ}$ is \emph{contained} in $\sigma^{\circ}$, denoted $\pi^{\circ} \leq \sigma^{\circ}$, if $\pi^{\tikzcircle{1.5pt}}$ can be embedded in $\sigma^{\tikzcircle{1.5pt}}$ in such a way that the origins match up. Explicitly, if
\[
\pi^{\circ} = \pi^{\tikzcircle{1.5pt}}(1)\pi^{\tikzcircle{1.5pt}}(2)\dots\underline{\pi^{\tikzcircle{1.5pt}}(k)}\dots\pi^{\tikzcircle{1.5pt}}(m)\pi^{\tikzcircle{1.5pt}}(m+1)
\]
and
\[
\sigma^{\circ} = \sigma^{\tikzcircle{1.5pt}}(1)\sigma^{\tikzcircle{1.5pt}}(2)\dots\underline{\sigma^{\tikzcircle{1.5pt}}(\ell)}\dots\sigma^{\tikzcircle{1.5pt}}(n)\sigma^{\tikzcircle{1.5pt}}(n+1)
\]
then $\pi^{\circ} \leq \sigma^{\circ}$ if and only if there is an ascending sequence of $m+1$ indices
\begin{equation*}
1 \leq i_{1} < i_{2} < \dots < i_{m+1} \leq n + 1
\end{equation*}
such that the sequence $\sigma^{\tikzcircle{1.5pt}}(i_{1})\sigma^{\tikzcircle{1.5pt}}(i_{2})\dots\sigma^{\tikzcircle{1.5pt}}(i_{m+1})$ is order-isomorphic to $\pi^{\tikzcircle{1.5pt}}$ \emph{and} $i_{k} = \ell$.
\end{defn}

As with regular permutation containment, centred permutation containment is much easier to understand graphically --- see Figure \ref{fig:centcont} for an example. This is one of the reasons that we generally prefer to draw centered permutations out rather than work with one-line notation.

\begin{figure}[h]
\begin{center}
\begin{tikzpicture}[scale=0.5]

\node[circle, draw, fill=none, inner sep=0pt, minimum width=\plotptradius] (0) at (3,2) {};
\node[permpt] (1) at (1,1) {}; 
\node[permpt] (2) at (2,4) {};
\node[permpt] (3) at (4,3) {};

\draw[thick] (0,2) -- ++ (5,0);
\draw[thick] (3,0) -- ++ (0,5);

\node at (7,2) {$\leq$};

\begin{scope}[shift={(9,-2)}]

\node[circle, blue, draw, fill=none, inner sep=0pt, minimum width=\plotptradius] (0) at (5,4) {};
\node[permpt] (1) at (1,7) {}; 
\node[permpt,blue] (2) at (2,1) {};
\node[permpt,blue] (3) at (3,8) {};
\node[permpt] (4) at (4,3) {};
\node[permpt,blue] (5) at (6,5) {};
\node[permpt] (6) at (7,2) {};
\node[permpt] (7) at (8,6) {};

\draw[thick] (0,4) -- ++ (9,0);
\draw[thick] (5,0) -- ++ (0,9);

\end{scope}

\end{tikzpicture}
\end{center}
\caption{The centred permutation $14\underline{2}3$ is contained in $7183\underline{4}526$; the relevant embedding is shown by the points marked in blue in the diagram of $7183\underline{4}526$. Note that the embedding must make the origins match up.}
\label{fig:centcont}
\end{figure}

Centred permutation containment, of course, immediately allows us to define the centred analogue of a permutation class, along with basic related notions:

\begin{defn}[Centred Permutation Class] \label{def:centclass}
A \emph{centred permutation class} is a non-empty set $\mathcal{C}^{\circ}$ of centred permutations which is closed under centred permutation containment. We call a centred permutation class \emph{proper} if the set of all underlying (uncentred) permutations of centred permutations in $\mathcal{C}^{\circ}$ does not contain every (uncentred) permutation.

Given a centred permutation class $\mathcal{C}^{\circ}$, we denote by $\mathcal{C}^{\circ}_{n}$ the set of centred permutations in $\mathcal{C}^{\circ}$ of length $n$. We refer to $(C_{n}) = (|\mathcal{C}^{\circ}_{n}|)$ as the \emph{enumeration sequence} of $\mathcal{C}^{\circ}$ and define the \emph{generating function} of $\mathcal{C}^{\circ}$ to be the formal power series
\begin{equation} \label{lkfjaskz8hjk}
f(z) = 1 + C_{1}z + C_{2}z^{2} + C_{3}z^{3} + C_{4}z^{4} + \dots
\end{equation}

The \emph{upper} and \emph{lower growth rates} of a centred permutation class $\mathcal{C}^{\circ}$ are defined, respectively, to be 
\[
\overline{gr}(\mathcal{C}^{\circ}) = \limsup_{n\to\infty}\sqrt[n]{C_n}, \quad\text{ and }\quad \underline{gr}(\mathcal{C}^{\circ})=\liminf_{n\to\infty}\sqrt[n]{C_n} \ ,
\]
both of which always exist in $[0,\infty]$ by basic real analysis. If the quantities $\overline{gr}(\mathcal{C}^{\circ})$ and $\underline{gr}(\mathcal{C}^{\circ})$ are finite and equal then we refer to the common value as the \emph{growth rate} of $\mathcal{C}^{\circ}$, denoted $gr(\mathcal{C}^{\circ})$.
\end{defn}

\begin{comment}[Non-empty generation functions]
We will usually, as in (\ref{lkfjaskz8hjk}), include the single \emph{empty} centred permutation $(\circ)$ in the generating function. This distinction is usually unimportant, but if we explictly wish to exclude the empty centred permutation we refer instead to the \emph{non-empty} generating function of $\mathcal{C}^{\circ}$. Thus the non-empty generating function of $\mathcal{C}^{\circ}$ is given by
\begin{equation*}
f(z) = C_{1}z + C_{2}z^{2} + C_{3}z^{3} + C_{4}z^{4} + \dots
\end{equation*}
\end{comment}

Note that $\sigma^{\circ} \leq \pi^{\circ}$ implies that $\sigma \leq \pi$. This means that, for every centred permutation class $\mathcal{C}^{\circ}$, the set $\mathcal{C}$ consisting of the underlying permutations of $\mathcal{C}^{\circ}$ is itself a permutation class. We call $\mathcal{C}$ the \emph{underlying} permutation class of $\mathcal{C}^{\circ}$. We wish to use centred permutation classes as a means to study regular permutation classes: there are many interesting permutation classes that are most succinctly described as the underlying class of some given centred permutation class. In particular, we are interested in the growth rates of (centred and uncentred) permutation classes: recall from the introduction that the Marcus-Tardos Theorem guarantees that the upper and lower growth rates of a proper permutation class are always finite. We plan to use centred permutation classes to study the growth rates of their underlying permutation classes, so the following is crucial to note:

\begin{prop} \label{eqgrs}
Suppose that $\mathcal{C}^{\circ}$ is a proper centred permutation class whose underlying permutation class is $\mathcal{C}$. Then:
\begin{enumerate}
\item The upper growth rate of $\mathcal{C}^{\circ}$ exists and is equal to the upper growth rate of $\mathcal{C}$.
\item The lower growth rate of $\mathcal{C}^{\circ}$ exists and is equal to the lower growth rate of $\mathcal{C}$.
\item The proper growth rate $gr(\mathcal{C}^{\circ})$ exists if and only if $gr(\mathcal{C})$ exists; if so then $gr(\mathcal{C}^{\circ}) = gr(\mathcal{C})$.
\end{enumerate}
\end{prop}
\begin{proof}
As $\mathcal{C}^{\circ}$ is a proper centred permutation class, $\mathcal{C}$ is, by definition, a proper permutation class. Hence by the Marcus-Tardos Theorem~\cite{marcus:excluded-permut:}, $\mathcal{C}$ has an upper growth rate $\rho < \infty$. 

Denote by $\mathcal{C}_{n}$ and $\mathcal{C}^{\circ}_{n}$ respectively the set of elements of $\mathcal{C}$ and $\mathcal{C}^{\circ}$ of length $n$. The map that removes the origin maps $\mathcal{C}^{\circ}_{n}$ onto $\mathcal{C}_{n}$, though not injectively, as distinct centred permutations can have the same underlying permutation. In fact, an element of $\mathcal{C}_{n}$ could have been produced by up to $(n+1)^2$ distinct elements of $\mathcal{C}^{\circ}_{n}$, corresponding to the $(n+1)^2$ different positions an origin point can be placed in a permutation of length $n$. These two facts combine to give us the inequality:
\[
|\mathcal{C}_{n}| \leq |\mathcal{C}^{\circ}_{n}| \leq (n+1)^2 |\mathcal{C}_{n}|,
\]
which in turn implies that
\[
\sqrt[n]{|\mathcal{C}_{n}|} \leq \sqrt[n]{|\mathcal{C}^{\circ}_{n}|} \leq \sqrt[n]{(n+1)^2} \sqrt[n]{|\mathcal{C}_{n}|}
\]
for all $n$. As $\sqrt[n]{(n+1)^2} \rightarrow 1$ as $n \rightarrow \infty$ we have thus sandwiched $\sqrt[n]{|\mathcal{C}^{\circ}_{n}|}$ between two sequences which both have lim sup $\rho$; we conclude that $\sqrt[n]{|\mathcal{C}^{\circ}_{n}|}$ also has lim sup $\rho$, which is to say that $\mathcal{C}^{\circ}$ has upper growth rate $\rho$.

Precisely the same proof goes through for lower growth rates, from which the equality of proper growth rates follows.
\end{proof}

\begin{comment}
On noting that a centred permutation can be identified with a $2$-by-$2$ gridded permutation, the fact that $\mathcal{C}^{\circ}$ has the same upper and lower growth rates as $\mathcal{C}$ may also be deduced immediately from a more general result of~{Vatter~\cite[Proposition 2.1]{vatter:small-permutati:}} dealing with \emph{gridded permutation classes} in general.
\end{comment}

We also note the following related result:
\begin{lemma} \label{eqgrs2}
Let $\mathcal{C}^{\circ}$ be a proper centred permutation class and let $\mathcal{C}^{\circ+n}$ denote the class consisting of centred permutations in $\mathcal{C}^{\circ}$ with at most $n$ points added anywhere. Then $\overline{gr}(\mathcal{C}^{\circ+n})$ exists and is equal to $\overline{gr}(\mathcal{C}^{\circ})$. Similarly, $\underline{gr}(\mathcal{C}^{\circ+n}) = \underline{gr}(\mathcal{C}^{\circ})$.
\end{lemma}

\begin{proof}
The case $\mathcal{C}^{\circ+1}$ is identical to the proof of Proposition \ref{eqgrs}. The general result for $\mathcal{C}^{\circ+n}$ then follows by induction. \qedhere
\end{proof}

We shall also require a centred analogue of the notion of an interval of a permutation:
\begin{defn}[$\circ$-intervals of a Centred Permutation]
A \emph{centred-interval} (also written \emph{$\circ$-interval}) of a centred permutation
\[
\sigma^{\circ} = \sigma^{\tikzcircle{1.5pt}}(1)\sigma^{\tikzcircle{1.5pt}}(2)\dots\underline{\sigma^{\tikzcircle{1.5pt}}(k)}\dots\sigma^{\tikzcircle{1.5pt}}(n)\sigma^{\tikzcircle{1.5pt}}(n+1)
\]
is a set of consecutive indices $\ell, \ell+1, \dots , m$ \emph{including $k$} such that the set of values
\[
\{\sigma^{\tikzcircle{1.5pt}}(i) \mid \ell \leq i \leq m\}
\]
is itself an interval of integers.

Informally, a $\circ$-interval is an interval (in the sense of Definition \ref{def:uncentredinvtervals}) which contains the origin. Visually, a $\circ$-interval of $\sigma^{\circ}$ is a bounding rectangle containing the origin which is not sliced horizontally or vertically by any outside point.

We call a $\circ$-interval \emph{non-trivial} if it contains at least one point in addition to the origin, and \emph{proper} if it does not contain every point of the centred permutation. A non-trivial $\circ$-interval is \emph{minimal} if it contains no smaller non-trivial $\circ$-interval. We note that every non-empty centred permutation contains at least one minimal $\circ$-interval, as the permutation itself is a non-trivial $\circ$-interval and it is either itself minimal or contains a smaller non-trivial $\circ$-interval; we can then iterate until we obtain a minimal $\circ$-interval.
\end{defn}

Minimal $\circ$-intervals will be especially important in the theory going forward, so we note the following restriction on their form and number:

\begin{lemma}[Minimal $\circ$-intervals in a centred permutation] \label{minint}
Let $\pi^{\circ}$ be a non-empty centred permutation. Then $\pi^{\circ}$ has either:
\begin{itemize}
\item one minimal $\circ$-interval $\mathcal{I}$;\ \ \ \emph{or}
\item two minimal $\circ$-intervals $\mathcal{I}_{1}$, $\mathcal{I}_{2}$. In this case, each $\mathcal{I}_{i}$ will contain points in one quadrant only, and the quadrant containing all of the points in $\mathcal{I}_{1}$ will diagonally opposite to the quadrant containing all of the points in $\mathcal{I}_{2}$.
\end{itemize}
\end{lemma}

\begin{proof}
Suppose that $\pi^{\circ}$ has two distinct minimal $\circ$-intervals $\mathcal{I}_{1}$, $\mathcal{I}_{2}$. Then, as the intersection of two $\circ$-intervals is necessarily a $\circ$-interval, $\mathcal{I}_{1} \cap \mathcal{I}_{2}$ is a $\circ$-interval.

Hence $\mathcal{I} = \mathcal{I}_{1} \cap \mathcal{I}_{2}$ is a $\circ$-interval contained in both $\mathcal{I}_{1}$ and $\mathcal{I}_{2}$. As these were assumed to be distinct we can deduce that $\mathcal{I}$ is \emph{strictly} contained in at least one of these. Hence, by the assumed minimality of $\mathcal{I}_{1}$ and $\mathcal{I}_{2}$,
\[
\mathcal{I} = \mathcal{I}_{1} \cap \mathcal{I}_{2} = \{\circ\}.
\]
Hence $\mathcal{I}_{1}, \mathcal{I}_{2}$ are non-trivial $\circ$-intervals with trivial overlap. Suppose that $\mathcal{I}_{1}, \mathcal{I}_{2}$ both have at least one point in the upper half-plane: say $p_{1} \in \mathcal{I}_{1},p_{2} \in \mathcal{I}_{2}$ are both points in the upper half-plane. Note that $p_{1} \notin \mathcal{I}_{2},p_{2} \notin \mathcal{I}_{1}$ (by $\mathcal{I}_{1} \cap \mathcal{I}_{2} = \{\circ\}$) and that one of $p_{1},p_{2}$ must be lower than the other --- without loss of generality, say that $p_{1}$ is lower than $p_{2}$. But then $p_{1}$ is vertically between the origin and $p_{2}$ and hence slices the rectangle $\mathcal{I}_{2}$, contradicting the assumption that $\mathcal{I}_{2}$ is a $\circ$-interval. Hence our supposition is impossible, and $\mathcal{I}_{1}, \mathcal{I}_{2}$ cannot both have points in the upper half-plane. See Figure \ref{fig:minintervals} for an illustration of this argument.

By precisely the same logic $\mathcal{I}_{1}, \mathcal{I}_{2}$ cannot both have points in any given half-plane, and thus the only possibility remaining is that $\mathcal{I}_{1}$ and $\mathcal{I}_{2}$ are both one-quadrant intervals occupying diagonally opposite quadrants, as required. Note that this argument immediately precludes the existence of a third minimal $\circ$-interval $\mathcal{I}_{3}$ as this would have to be a one-quadrant interval diagonally opposite to both $\mathcal{I}_{1}$ and $\mathcal{I}_{2}$. \qedhere
\end{proof}

\begin{figure}[h]
\begin{center}
\begin{tikzpicture}[scale=0.5]

\node[circle, draw, fill=none, inner sep=0pt, minimum width=\plotptradius] (0) at (0,0) {};
\node[permpt,label={\tiny $p_{1}$}] (1) at (-2,1) {}; 
\node[permpt,label={\tiny $p_{2}$}] (2) at (1,2) {};
\node[empty,label={\tiny $\mathcal{I}_{2}$}] (3) at (2,0) {};

\draw[thick] (0,-1) -- ++ (0,4);
\draw[thick] (-3,0) -- ++ (6,0);

\draw[dotted] (-0.5,-0.5) rectangle (1.5,2.5);

\end{tikzpicture}
\end{center}
\caption{The 'interval' $\mathcal{I}_{2}$ is intersected horizontally by the point $p_{1}$ --- and hence is not an interval at all.}
\label{fig:minintervals}
\end{figure}

The two possible cases for the number and type of minimal $\circ$-intervals in a centred permutation, as described in Lemma \ref{minint}, are illustrated by the two centred permutations in Figure \ref{fig:centredintervals2types}.

\begin{figure}[h]
\begin{center}
\begin{tikzpicture}[scale=0.35]

\node[circle, draw, fill=none, inner sep=0pt, minimum width=\plotptradius] (0) at (5,3) {};
\node[permpt] (1) at (1,1) {}; 
\node[permpt] (2) at (2,7) {};
\node[permpt] (3) at (3,4) {};
\node[permpt] (4) at (4,6) {};
\node[permpt] (5) at (6,5) {};
\node[permpt] (6) at (7,2) {};
\node[permpt] (7) at (8,8) {};

\draw[thick] (0,3) -- ++ (9,0);
\draw[thick] (5,0) -- ++ (0,9);

\draw[dotted] (2.5,2.5) rectangle (6.5,6.5);

\begin{scope}[shift={(12,0)}]

\node[circle, draw, fill=none, inner sep=0pt, minimum width=\plotptradius] (0) at (3,3) {};
\node[permpt] (1) at (1,7) {}; 
\node[permpt] (2) at (2,2) {};
\node[permpt] (3) at (4,5) {};
\node[permpt] (4) at (5,4) {};
\node[permpt] (5) at (6,8) {};
\node[permpt] (6) at (7,1) {};
\node[permpt] (7) at (8,6) {};

\draw[thick] (0,3) -- ++ (9,0);
\draw[thick] (3,0) -- ++ (0,9);

\draw[dotted] (2.5,2.5) rectangle (5.5,5.5);
\draw[dotted] (1.5,1.5) rectangle (3.5,3.5);

\end{scope}

\end{tikzpicture}
\end{center}
\caption{The centred permutations $\pi^{\circ}_{1} = 1746\underline{3}528$ and $\pi^{\circ}_{2} = 72\underline{3}54816$, along with their minimal $\circ$-intervals. Note that $\pi^{\circ}_{1}$ has a unique minimal $\circ$-interval, whereas $\pi^{\circ}_{2}$ has two minimal $\circ$-intervals which (as is guaranteed to be the case by Lemma \ref{minint}) both contain points from one quadrant only, and from diagonally opposite quadrants.}
\label{fig:centredintervals2types}
\end{figure}

\subsection{The $\boxplus$-decomposition} \label{sec:2.2}

The following operation on centred permutations will be fundamental in describing the structure of the permutation classes we work with in this paper:

\begin{defn}[The Box Sum]
Given two centred permutations $\pi^{\circ}$ and $\sigma^{\circ}$ their \emph{$\boxplus$-sum} (read: \emph{box-sum}), written $\pi^{\circ} \boxplus \sigma^{\circ}$, is obtained by inflating the origin of $\sigma^{\circ}$ with a copy of $\pi^{\circ}$ --- see Figure \ref{fig:boxsumdef} for an illustration.
\end{defn}

\begin{figure}[h]
\begin{center}
\begin{tikzpicture}[scale=0.35]

\begin{scope}[shift={(0,4)}]
\node[circle, draw, fill=none, inner sep=0pt, minimum width=\plotptradius] (0) at (4,3) {};
\node[permpt] (1) at (5,5) {};
\node[permpt] (2) at (2,4) {};
\node[permpt] (3) at (3,1) {};
\node[permpt] (4) at (1,2) {};

\draw[thin] (4,0.5) -- ++ (0,6);
\draw[thin] (0.5,3) -- ++ (6,0);

\node at (9,3) {$\boxplus$};
\end{scope}

\begin{scope}[shift={(11,2)}]
\node[circle, draw, fill=none, inner sep=0pt, minimum width=\plotptradius] (0) at (4,5) {};
\node[permpt] (1) at (3,3) {};
\node[permpt] (2) at (1,4) {};
\node[permpt] (3) at (2,1) {};
\node[permpt] (4) at (5,2) {};

\draw[thin] (4,0.5) -- ++ (0,8);
\draw[thin] (0.5,5) -- ++ (7,0);

\node at (10,5) {=};
\end{scope}

\begin{scope}[shift={(23,0)}]
\node[circle, draw, fill=none, inner sep=0pt, minimum width=\plotptradius] (0) at (7,7) {};
\node[permpt] (1) at (8,9) {};
\node[permpt] (2) at (5,8) {};
\node[permpt] (3) at (6,5) {};
\node[permpt] (4) at (4,6) {};
\node[permpt] (1) at (3,3) {};
\node[permpt] (2) at (1,4) {};
\node[permpt] (3) at (2,1) {};
\node[permpt] (4) at (9,2) {};

\draw[thin] (7,0.5) -- ++ (0,12);
\draw[thin] (0.5,7) -- ++ (12,0);

\draw[dotted] (8.5,0.5) -- ++ (0,12);
\draw[dotted] (3.5,0.5) -- ++ (0,12);
\draw[dotted] (0.5,9.5) -- ++ (12,0);
\draw[dotted] (0.5,4.5) -- ++ (12,0);

%\draw[dotted] (3.5,4.5) rectangle (8.5,9.5);

\end{scope}

\end{tikzpicture}
\end{center}
\caption{The $\boxplus$-sum of the centred permutations $\pi^{\circ} = 241\underline{3}5$ and $\sigma^{\circ} = 413\underline{5}2$ is $413685\underline{7}92$; the origin of $\sigma^{\circ}$ is simply replaced ('inflated') by a copy of $\pi^{\circ}$.}
\label{fig:boxsumdef}
\end{figure}

\begin{comment}
As the visual representation makes clear, the $\boxplus$-sum is associative. On the other hand, the $\boxplus$-sum is certainly \emph{not} commutative, as discussed in more detail below.
\end{comment}

\begin{comment} \label{comment:directsum}
The $\boxplus$-sum can be thought of as a generalisation of the direct and skew sums of (uncentred) permutations, two of the most fundamental operations in the study of permutation patterns --- see Bevan~\cite{bevan2015defs} for the definitions of these operations. When $\pi^{\circ}$ and $\sigma^{\circ}$ are centred permutations whose points are in the first quadrant only, $\pi^{\circ} \boxplus \sigma^{\circ}$ is equivalent to the usual direct sum $\pi \oplus \sigma$ of the corresponding uncentred permutations. If, on the other hand, $\pi^{\circ}$ and $\sigma^{\circ}$ are centred permutations whose points are in the \emph{second} quadrant only, then $\pi^{\circ} \boxplus \sigma^{\circ}$ is the usual skew sum $\sigma \ominus \pi$ of the uncentred permutations (note the reversed order).
\end{comment}

Figure \ref{fig:boxsumdef} suggests a close relationship between the $\boxplus$-sum and $\circ$-intervals. Note that if $\mathcal{I}$ is a $\circ$-interval of a centred permutation $\pi^{\circ}$ then $\mathcal{I}$ contains the origin, and so the set of points contained in $\mathcal{I}$ forms a centred permutation $\sigma^{\circ}$. We say that $\mathcal{I}$ \emph{encloses} the centred permutation $\sigma^{\circ}$, and we immediately observe the following:

\begin{obs}[The $\boxplus$-sum and centred intervals] \label{intsum}
Let $\pi^{\circ},\sigma^{\circ}$ be centred permutations, with $\sigma^{\circ}$ having strictly smaller length than $\pi^{\circ}$. Then:
\begin{enumerate}
\item $\pi^{\circ} = \sigma^{\circ} \boxplus \tau^{\circ}$ for some centred permutation $\tau^{\circ}$ if and only if $\pi^{\circ}$ contains a (necessarily unique) $\circ$-interval $\mathcal{I}_{\sigma}$ enclosing $\sigma^{\circ}$.
\item If $\pi^{\circ} = \sigma^{\circ} \boxplus \tau^{\circ}$ then $\tau^{\circ}$ is the centred permutation obtained from $\pi^{\circ}$ by deleting all non-origin points contained in the (unique) $\circ$-interval $\mathcal{I}_{\sigma}$.
\end{enumerate}
\end{obs}

We note the following immediate consequence of this result:

\begin{cor}[Left-cancellation]
Suppose $\sigma^{\circ},\tau^{\circ}_{1},\tau^{\circ}_{2}$ are centred permutations such that
\[
\sigma^{\circ} \ \boxplus \ \tau^{\circ}_{1} \ = \ \sigma^{\circ} \ \boxplus \ \tau^{\circ}_{2} \ .
\]
Then $\tau^{\circ}_{1} = \tau^{\circ}_{2}$.
\end{cor}

\begin{proof}
Let $\pi^{\circ} = \sigma^{\circ} \boxplus \tau^{\circ}_{1} = \sigma^{\circ} \boxplus \tau^{\circ}_{2}$. Then $\pi^{\circ}$ contains a unique $\circ$-interval $\mathcal{I}_{\sigma}$ enclosing $\sigma^{\circ}$ by Observation \ref{intsum}. Deleting all (non-origin) points in $\mathcal{I}_{\sigma}$ from $\pi^{\circ}$ will result in $\tau^{\circ}_{1}$, but by the same token will also result in $\tau^{\circ}_{2}$; hence $\tau^{\circ}_{1} = \tau^{\circ}_{2}$.
\end{proof}

Right-cancellation is also easy to prove, but will not be required here.

\begin{defn}[$\boxplus$-decomposables and -indecomposables] \label{def:boxinds}
We say that a centred permutation $\sigma^{\circ}$ is $\boxplus$-\emph{decomposable} if it can be written as $\pi_{1}^{\circ} \boxplus \pi_{2}^{\circ}$ for two shorter centred permutations $\pi_{1}^{\circ}$ and $\pi_{2}^{\circ}$. Conversely, $\sigma^{\circ}$ is $\boxplus$-\emph{indecomposable} if it cannot be split up like this.
\end{defn}

It will be important to note for later that $\sigma^{\circ}$ is $\boxplus$-decomposable if and only if it contains a proper, non-trivial $\circ$-interval; this is an immediate consequence of Observation \ref{intsum}. For example, the inner rectangle in Figure \ref{fig:boxsumdef} is an interval because there are no points in the 'shadow' of the rectangle in either the $x$- or $y$-direction.

Any $\boxplus$-decomposable centred permutation can be decomposed as a $\boxplus$-sum of $\boxplus$-indecomposables. This decomposition is not necessarily unique, since certain elements commute, as in Figure \ref{fig:commpair}.

\begin{figure}[h]
\begin{center}
\begin{tikzpicture}[scale=0.35]

\node[circle, draw, fill=none, inner sep=0pt, minimum width=\plotptradius] (0) at (2,2) {};
%\node[permpt] (1) at (1,1) {};
\node[permpt] (2) at (3,4) {};
\node[permpt] (3) at (4,3) {};

\draw[thin] (2,-1) -- ++ (0,6);
\draw[thin] (-1,2) -- ++ (6,0);

\node at (7,2) {$\boxplus$};

\begin{scope}[shift={(10,0)}]
\node[circle, draw, fill=none, inner sep=0pt, minimum width=\plotptradius] (0) at (2,2) {};
\node[permpt] (1) at (1,1) {};
%\node[permpt] (2) at (3,4) {};
%\node[permpt] (3) at (4,3) {};

\draw[thin] (2,-1) -- ++ (0,6);
\draw[thin] (-1,2) -- ++ (6,0);

\node at (7,2) {=};
\end{scope}

\begin{scope}[shift={(20,0)}]
\node[circle, draw, fill=none, inner sep=0pt, minimum width=\plotptradius] (0) at (2,2) {};
\node[permpt] (1) at (1,1) {};
\node[permpt] (2) at (3,4) {};
\node[permpt] (3) at (4,3) {};

\draw[thin] (2,-1) -- ++ (0,6);
\draw[thin] (-1,2) -- ++ (6,0);

\node at (7,2) {=};
\end{scope}

\begin{scope}[shift={(30,0)}]
\node[circle, draw, fill=none, inner sep=0pt, minimum width=\plotptradius] (0) at (2,2) {};
\node[permpt] (1) at (1,1) {};
%\node[permpt] (2) at (3,4) {};
%\node[permpt] (3) at (4,3) {};

\draw[thin] (2,-1) -- ++ (0,6);
\draw[thin] (-1,2) -- ++ (6,0);

\node at (7,2) {$\boxplus$};
\end{scope}

\begin{scope}[shift={(40,0)}]
\node[circle, draw, fill=none, inner sep=0pt, minimum width=\plotptradius] (0) at (2,2) {};
%\node[permpt] (1) at (1,1) {};
\node[permpt] (2) at (3,4) {};
\node[permpt] (3) at (4,3) {};

\draw[thin] (2,-1) -- ++ (0,6);
\draw[thin] (-1,2) -- ++ (6,0);
\end{scope}

\end{tikzpicture}
\end{center}
\caption{The two $\boxplus$-indecomposable centred permutations $\underline{1}32$ and $1\underline{2}$ commute under the $\boxplus$-sum, due to being one-quadrant permutations from opposite quadrants.}
\label{fig:commpair}
\end{figure}

More generally, we have:

\begin{lemma}[Commuting elements under the $\boxplus$-sum] \label{commuters}\ \\
Suppose $\pi^{\circ}$ and $\sigma^{\circ}$ are non-empty $\boxplus$-indecomposable centred permutations.
Then
\[
\pi^{\circ} \ \boxplus \ \sigma^{\circ} \ = \ \sigma^{\circ} \ \boxplus \ \pi^{\circ}
\]
if and only if either $\pi^{\circ} = \sigma^{\circ}$ or $\pi^{\circ}$ and $\sigma^{\circ}$ are one-quadrant centred permutations from diagonally opposite quadrants.
\end{lemma}

\begin{proof}

One direction here is obvious: if either $\pi^{\circ} = \sigma^{\circ}$ or $\pi^{\circ}$ and $\sigma^{\circ}$ are one-quadrant centred permutations from opposite quadrants then clearly $\pi^{\circ}$ and $\sigma^{\circ}$ commute under the $\boxplus$-sum. For the converse, we assume that the centred permutations $\pi^{\circ}$ and $\sigma^{\circ}$ commute and aim to deduce that they must be of one of these two forms:

Consider the centred permutation $\tau^{\circ} = \pi^{\circ} \boxplus \sigma^{\circ} = \sigma^{\circ} \boxplus \pi^{\circ}$. This can be generated in two ways: either by inflating the origin of $\sigma^{\circ}$ by $\pi^{\circ}$ or vice versa. Thus there is a $\circ$-interval $\mathcal{I}_{1}$ of $\tau^{\circ}$ enclosing the centred permutation $\pi^{\circ}$ and another $\circ$-interval $\mathcal{I}_{2}$ of $\tau^{\circ}$ enclosing the centred permutation $\sigma^{\circ}$. As $\pi^{\circ}, \sigma^{\circ}$ are both $\boxplus$-indecomposable, $\mathcal{I}_{1}$ and $\mathcal{I}_{2}$ are both minimal $\circ$-intervals. Hence, by Lemma \ref{minint}, either $\mathcal{I}_{1} = \mathcal{I}_{2}$, in which case $\pi^{\circ} = \sigma^{\circ}$, or $\mathcal{I}_{1}$ and $\mathcal{I}_{2}$ each contain points in only one quadrant, and from opposite quadrants to each other, in which case $\pi^{\circ},\sigma^{\circ}$ are one-quadrant centred permutations from opposite quadrants, as required.\end{proof}

We shall henceforth call a pair $\{\pi^{\circ},\sigma^{\circ}\}$ of one-quadrant centred permutations from opposite quadrants a \emph{commuting pair}.

In fact, commutativity of these commuting pairs is the \emph{only} way that uniqueness of the $\boxplus$-decomposition can fail, as is demonstrated by the following result:

\begin{thm} \label{uniqueness}
Let $\pi^{\circ}$ be a centred permutation. Then there exist $\boxplus$-indecomposables $\sigma_{1}^{\circ}, \sigma_{2}^{\circ}, \dots, \sigma_{n}^{\circ}$ such that
\[
\pi^{\circ} = \sigma_{1}^{\circ} \ \boxplus \ \sigma_{2}^{\circ} \ \boxplus \ \dots \ \boxplus \ \sigma_{n}^{\circ}.
\]
Furthermore, this decomposition is unique up to commutativity of adjacent commuting pairs. In other words, if 
\[
\pi^{\circ} = \tau_{1}^{\circ} \ \boxplus \ \tau_{2}^{\circ} \ \boxplus \ \dots \ \boxplus \ \tau_{k}^{\circ}.
\]
where $\tau_{1}^{\circ}, \tau_{2}^{\circ}, \dots, \tau_{k}^{\circ}$ are $\boxplus$-indecomposables, then:
\begin{enumerate}
\item k = n;
\item $[\sigma_{1}^{\circ}, \sigma_{2}^{\circ}, \dots, \sigma_{n}^{\circ}] = [\tau_{1}^{\circ}, \tau_{2}^{\circ}, \dots, \tau_{k}^{\circ}]$ as multisets;
\item the $n$-tuple $(\sigma_{1}^{\circ}, \sigma_{2}^{\circ}, \dots, \sigma_{n}^{\circ})$ can be transformed into $(\tau_{1}^{\circ}, \tau_{2}^{\circ}, \dots, \tau_{k}^{\circ})$ by a sequence of transpositions of adjacent commuting pairs.
\end{enumerate}

\end{thm}

\begin{proof}
On noting that centred permutations are equivalent to $2$-by-$2$-gridded permutations, this follows from a much stronger result of~{Bevan, Brignall \& Ru\v{s}kuc~\cite[Lemma 3.4]{bbr:unicyclicgrids:}}. For completeness we outline a proof here:

The existence of a $\boxplus$-decomposition is clear: we may take minimal $\circ$-intervals inductively and apply Observation \ref{intsum}.

For uniqueness the key is to prove that if $\sigma_{1}^{\circ} \neq \tau_{1}^{\circ}$ then there is some $\ell \in \{2,3, \dots, k\}$ such that $\sigma_{1}^{\circ} = \tau_{\ell}^{\circ}$ and $\sigma_{1}^{\circ} = \tau_{\ell}^{\circ}$ commutes with $\tau_{j}^{\circ}$ for $j = 1,2, \dots, \ell-1$. All three claims follow inductively from this.

To prove this, note that, as
\[
\pi^{\circ} = \tau_{1}^{\circ} \ \boxplus \ \tau_{2}^{\circ} \ \boxplus \ \dots \ \boxplus \ \tau_{k}^{\circ} \ ,
\]
there is, by Lemma \ref{intsum} an increasing sequence of $\circ$-intervals $\mathcal{J}_{1}, \mathcal{J}_{2}, \dots, \mathcal{J}_{k}$ where $\mathcal{J}_{j}$ encloses the centred permutation
\[
\tau_{1}^{\circ} \ \boxplus \ \tau_{2}^{\circ} \ \boxplus \ \dots \ \boxplus \ \tau_{j}^{\circ} \ .
\]
Also by Observation \ref{intsum} there is a unique $\circ$-interval $\mathcal{I}$ enclosing $\sigma_{1}^{\circ}$. Note that $\mathcal{J}_{1} \cap \mathcal{I} = \{\circ\}$ (as these are distinct and $\mathcal{I}$ is minimal due to $\boxplus$-indecomposability of $\sigma_{1}^{\circ}$), and so by Lemma \ref{minint} these are both one-quadrant intervals from opposite quadrants --- without loss of generality, say that $\sigma_{1}^{\circ}$ is entirely contained in quadrant $1$ and $\tau_{1}^{\circ}$ is in quadrant $3$. Now, as $\mathcal{I}$ is minimal $\mathcal{J}_{j} \cap \mathcal{I}$ is either $\{\circ\}$ or $\mathcal{I}$ for all $j$. Further $\mathcal{J}_{k} \cap \mathcal{I} = \mathcal{I}$. As $\mathcal{J}_{j}$ is an increasing sequence of intervals, this implies that there is some $\ell$ such that $\mathcal{J}_{\ell} \cap \mathcal{I} = \mathcal{I}$ but $\mathcal{J}_{j} \cap \mathcal{I} = \{\circ\}$ for all $j \leq \ell-1$. As $\mathcal{J}_{\ell-1} \cap \mathcal{I} = \{\circ\}$, Lemma \ref{minint} implies that $\tau_{1}^{\circ} \boxplus \tau_{2}^{\circ} \boxplus \dots \boxplus \tau_{\ell-1}^{\circ}$ is entirely contained in quadrant $3$, so $\sigma_{1}^{\circ}$ commutes with all of $\tau_{j}^{\circ}$ for $j = 1,2, \dots, \ell-1$. And as $\mathcal{J}_{\ell} \cap \mathcal{I} = \mathcal{I}$, $\tau_{1}^{\circ} \boxplus \tau_{2}^{\circ} \boxplus \dots \boxplus \tau_{\ell}^{\circ}$ contains a $\circ$-interval enclosing $\sigma_{1}^{\circ}$.

But then $\tau_{1}^{\circ} \boxplus \tau_{2}^{\circ} \boxplus \dots \boxplus \tau_{\ell}^{\circ}$ can be thought of as '$\tau_{\ell}^{\circ}$ plus some points in the third quadrant', whereas $\sigma_{1}^{\circ}$ is entirely contained in the first quadrant. Hence the interval enclosing $\sigma_{1}^{\circ}$ is in fact entirely contained in $\tau_{\ell}^{\circ}$, so $\sigma_{1}^{\circ} \leq \tau_{\ell}^{\circ}$. But by $\boxplus$-indecomposability of $\tau_{\ell}^{\circ}$ we conclude that $\sigma_{1}^{\circ} = \tau_{\ell}^{\circ}$.\end{proof}

\subsection{The Generating Function Specification} \label{sec:2.3}

Having introduced the $\boxplus$-sum, we now define in this section the notion of a $\boxplus$-closed centred permutation class, along with the $\boxplus$-closure of any set of centred permutations. Our main aim is to outline a procedure for enumerating a $\boxplus$-closed centred permutation class $\mathcal{C}^{\circ}$, given the enumeration sequence of the $\boxplus$-indecomposables in $\mathcal{C}^{\circ}$. The natural way to do this is by identifying elements of a $\boxplus$-closed class with tuples of $\boxplus$-indecomposables from the same class. This runs into problems resulting from the non-uniqueness of the $\boxplus$-decomposition, but having classifed this behaviour in the previous section we are able to tweak this idea in order to describe a general enumeration scheme for a $\boxplus$-closed class.

\begin{defn} \label{def:boxclosed}
A centred permutation class $\mathcal{C}^{\circ}$ is \emph{$\boxplus$-closed} if, for any $\pi^{\circ},\sigma^{\circ} \in \mathcal{C}^{\circ}$, $\pi^{\circ} \boxplus \sigma^{\circ} \in \mathcal{C}^{\circ}$.

For any non-empty set of centred permutations $S$ we define its \emph{$\boxplus$-closure} $\boxplus(S)$ to be the intersection of all $\boxplus$-closed centred permutation classes containing $S$ (this is of course non-empty as the class of \emph{all} centred permutations is itself $\boxplus$-closed).
\end{defn}

We collect some basic results about $\boxplus$-closures:

\begin{lemma}[$\boxplus$-closure properties] \label{lemma:boxprops}
Suppose $S$ is a non-empty set of centred permutations. Then:
\begin{enumerate}
\item The $\boxplus$-closure $\boxplus(S)$ is itself a $\boxplus$-closed centred permutation class. \label{lemma:boxprops1}
\item The $\boxplus$-closure $\boxplus(S)$ is the set of all centred permutations of the form \label{lemma:boxprops2}
\[
\sigma^{\circ}_{1} \ \boxplus \ \sigma^{\circ}_{2} \ \boxplus \ \dots \ \boxplus \ \sigma^{\circ}_{n} \ ,
\]
where each $\sigma^{\circ}_{i}$ is a $\boxplus$-indecomposable centred permutation contained in some element of $S$.
\item If $S$ is itself a centred permutation class, say $S = \mathcal{C}^{\circ}$, then the $\boxplus$-closure $\boxplus(\mathcal{C}^{\circ})$ is the set of all centred permutations of the form \label{lemma:boxprops3}
\[
\sigma^{\circ}_{1} \ \boxplus \ \sigma^{\circ}_{2} \ \boxplus \ \dots \ \boxplus \ \sigma^{\circ}_{n} \ ,
\]
where each $\sigma^{\circ}_{i}$ is a $\boxplus$-indecomposable centred permutation in $\mathcal{C}^{\circ}$.
\item A centred permutation class $\mathcal{C}^{\circ}$ is $\boxplus$-closed if and only if $\mathcal{C}^{\circ} = \boxplus(\mathcal{C}^{\circ})$. \label{lemma:boxprops4}
\end{enumerate}
\end{lemma}

\begin{proof}
\begin{enumerate}
\item The intersection of \emph{any} collection of centred permutation classes is a centred permutation class, so $\boxplus(S)$ is certainly a centred permutation class. To see that $\boxplus(S)$ is $\boxplus$-closed, suppose that $\pi^{\circ},\sigma^{\circ} \in \boxplus(S)$, and suppose that $\mathcal{D}^{\circ}$ is a $\boxplus$-closed permutation class containing $S$ as a subset. By definition, $\pi^{\circ}$ and $\sigma^{\circ}$ must be in $\mathcal{D}^{\circ}$, and so, by the fact that $\mathcal{D}^{\circ}$ is assumed to be $\boxplus$-closed, $\pi^{\circ} \boxplus \sigma^{\circ} \in \mathcal{D}^{\circ}$. But this is true for \emph{any} $\boxplus$-closed class containing $S$, and so $\pi^{\circ} \boxplus \sigma^{\circ} \in \boxplus(S)$, as required.
\item Consider the set
\[
\widetilde{S} = \left\{\sigma^{\circ}_{1} \ \boxplus \ \sigma^{\circ}_{2} \ \boxplus \ \dots \ \boxplus \ \sigma^{\circ}_{n} \ \mid \text{$\sigma^{\circ}_{i}$ is a $\boxplus$-indecomposable contained in an element of $S$}\right\}.
\]
This is clearly a $\boxplus$-closed centred permutation class containing $S$ as a subset, and so by definition must contain $\boxplus(S)$ as a subset. But the fact that $\boxplus(S)$ is itself $\boxplus$-closed implies that it must contain all of the elements in $\widetilde{S}$. Thus $\widetilde{S}$ is in fact equal to $\boxplus(S)$.
\item Immediate from part 2 on noting that $\mathcal{C}^{\circ}$ is $\leq$-closed, and so any $\boxplus$-indecomposable contained in an element of $\mathcal{C}$ is itself in $\mathcal{C}^{\circ}$.
\item Immediate from part 1 and the definition.\qedhere
\end{enumerate}
\end{proof}

Thus we may reasonably describe $\boxplus(S)$ as `the smallest $\boxplus$-closed centred permutation class containing $S$ as a subset'.

We note that Lemma \ref{lemma:boxprops} implies that the $\boxplus$-closure of a centred permutation class $\mathcal{C}^{\circ}$ is entirely determined by the set of $\boxplus$-indecompoables in $\mathcal{C}^{\circ}$. Our aim in this section is to describe a general method for determining the generating function of a $\boxplus$-closed centred permutation class from the enumeration sequence of its $\boxplus$-indecomposables.

Suppose then that we have a $\boxplus$-closed centred permutation class $\mathcal{C}^{\circ}$. Then every element of $\mathcal{C}^{\circ}$ is of the form
\begin{equation} \label{sdfjl789nnn}
\pi^{\circ} \ = \ \sigma_{1}^{\circ} \ \boxplus \ \sigma_{2}^{\circ} \ \boxplus \ \dots \ \boxplus \ \sigma_{n}^{\circ} \ ,
\end{equation}
where each $\sigma_{i}^{\circ}$ is a $\boxplus$-indecomposable in $\mathcal{C}^{\circ}$. Suppose further that we have enumerated the $\boxplus$-indecomposables in $\mathcal{C}^{\circ}$, meaning that we have found the enumeration sequence $(D_{n})$, where $D_{n}$ is the number of $\boxplus$-indecomposables of length $n$ in $\mathcal{C}^{\circ}$. We may summarise this information as the \emph{generating function $g(z)$ of the (non-empty) $\boxplus$-indecomposables in $\mathcal{C}^{\circ}$}:
\[
g(z) = D_{1}z + D_{2}z^2 + D_{3}z^3 + D_{4}z^4 + D_{5}z^5 + \dots
\]
Crucially, for the family of permutation classes that will be the main focus of this paper, this generating function will be relatively easy to find. We note, simply by multiplying out, that $g(z)^2$ is the generating function of ordered pairs of $\boxplus$-indecomposables from $\mathcal{C}^{\circ}$ (where the length of an ordered pair of $\boxplus$-indecomposables is defined to be the sum of the lengths of its two elements), $g(z)^3$ is the generating function of ordered triples of $\boxplus$-indecomposables, and, more generally, $g(z)^n$ is the generating function of ordered $n$-tuples of $\boxplus$-indecomposables from $\mathcal{C}^{\circ}$.

Hence, \emph{if the decomposition given in (\ref{sdfjl789nnn}) were unique} we could, by identifying elements of $\mathcal{C}^{\circ}$ with tuples of $\boxplus$-indecomposables, immediately obtain the generating function of the entire class by applying the \emph{sequent operator}\footnote{Explicitly, the sequent operator $Seq$ is the operator that sends a non-empty generating function $g(z)$ to $Seq(g(z)) = (1-g(z))^{-1}$.}:
\[
\begin{split}
f(z) = Seq(g(z)) & = 1 + g(z) + g(z)^{2} + g(z)^{3} + \dots \\
 & = \frac{1}{1 - g(z)} \ .
\end{split}
\]
Of course, as was demonstrated in the previous section, the decomposition (\ref{sdfjl789nnn}) is \emph{not} in general unique: the issue is that pairs of one-quadrant $\boxplus$-indecomposables from opposite quadrants commute, and so elements of $\mathcal{C}^{\circ}$ cannot be uniquely identified with $n$-tuples of $\boxplus$-indecomposables from the class. As it happens, however, we can amend this approach to provide a specification of the generating function of $\mathcal{C}^{\circ}$ in terms of $g(z)$, though, unsurprisingly, we shall also have to include the generating functions $g_{i}(z)$ of the one-quadrant $\boxplus$-indecomposables of $\mathcal{C}^{\circ}$ contained entirely in the $i$th quadrant, for each $i \in \{1,2,3,4\}$, as these are the elements that introduce commutativity into the $\boxplus$-decomposition:

\begin{thm}[Generating Function Specification for a $\boxplus$-closed Class]\label{genfuncspec}
Let $\mathcal{C}^{\circ}$ be a centred permutation class, with $\boxplus$-closure $\boxplus\left(\mathcal{C}^{\circ}\right)$. Let
\[
G(z) = g(z) - g_{1}(z)g_{3}(z) - g_{2}(z)g_{4}(z) \ ,
\]
where $g(z)$ is the generating function of the (non-empty) $\boxplus$-indecomposables of $\mathcal{C}^{\circ}$, and $g_{i}(z)$ is the generating function of the (non-empty) one-quadrant $\boxplus$-indecomposables of $\mathcal{C}^{\circ}$ contained entirely in the $i$th quadrant.
Then
\[
f(z) = \frac{1}{1 - G(z)}
\]
is the generating function of $\boxplus\left(\mathcal{C}^{\circ}\right)$.

Furthermore, if $\mathcal{C}^{\circ}$ is itself $\boxplus$-closed then $f(z)$ is in fact the generating function of $\mathcal{C}^{\circ}$.
\end{thm}

\begin{comment}
This can be derived almost immediately from a much more general result of Cartier and Foata~\cite{cartier:problemes-combi:} on trace monoids; see also Flajolet \& Sedgewick~\cite[Note V.10, p.307]{flajolet:analytic-combin:} for a comment on this result as applied to words with commutativity. For the sake of completeness we shall outline a direct proof here.
\end{comment}

\begin{proof}
Suppose $\tau^{\circ} \in \boxplus\left(\mathcal{C}^{\circ}\right)$. Then $\tau^{\circ}$ can be expressed in the form
\begin{equation} \label{kkkkjj2}
\tau^{\circ} = \sigma^{\circ}_{1} \ \boxplus \sigma^{\circ}_{2} \ \boxplus \ \dots \ \boxplus \ \sigma^{\circ}_{m} \ ,
\end{equation}
where each $\sigma^{\circ}_{i}$ is a $\boxplus$-indecomposable in $\mathcal{C}^{\circ}$, and, conversely, every centred permutation of this form is in $\boxplus\left(\mathcal{C}^{\circ}\right)$. The expression in (\ref{kkkkjj2}) is not in general unique, however: by Lemma \ref{commuters}, if $\sigma^{\circ}_{i}, \sigma^{\circ}_{i+1}$ are one-quadrant centred permutations from opposite quadrants then they can be switched around. Lemma \ref{commuters} also tells us, however, that this is the only thing that can go wrong with commutativity. This implies that each $\tau^{\circ} \in \boxplus\left(\mathcal{C}^{\circ}\right)$ does in fact have a unique \emph{canonical $\boxplus$-decomposition}, in which no one-quadrant $\boxplus$-indecomposable from the $1$st quadrant immediately follows a one-quadrant $\boxplus$-indecomposable from the $3$rd quadrant, and no one-quadrant $\boxplus$-indecomposable from the $2$nd quadrant immediately follows a one-quadrant $\boxplus$-indecomposable from the $4$th quadrant. We aim to enumerate these canonical decompositions.

To this end, we refer to a $\boxplus$-indecomposable with points in more than one quadrant a \emph{multi-quadrant $\boxplus$-indecomposable} (note that these are precisely the $\boxplus$-indecomposables that don't commute with any other centred permutation). We write $g_{*}(z)$ for the generating function of the multi-quadrant $\boxplus$-indecomposables in $\mathcal{C}^{\circ}$.

Next, let $d_{i}(z)$ denote the generating function of the centred permutations in $\mathcal{C}^{\circ}$ whose canonical $\boxplus$-decomposition ends in a one-quadrant $\boxplus$-indecomposable from the $i$th quadrant, and $d_{*}(z)$ denote the generating function of the centred permutations in $\mathcal{C}^{\circ}$ whose canonical $\boxplus$-decomposition ends in a multi-quadrant $\boxplus$-indecomposable.

We then have the following set of functional equations:

\begin{align*}
f(z) & = 1 + d_1(z) + d_2(z) + d_3(z) + d_4(z) + d_*(z) \\
g(z) & = g_{1}(z) + g_{2}(z) + g_{3}(z) + g_{4}(z) + g_{*}(z) \\
d_1(z) & = \big(f(z) - d_3(z)\big) g_1(z) \\
d_2(z) & = \big(f(z) - d_4(z)\big) g_2(z) \\
d_3(z) & = f(z) g_3(z) \\
d_4(z) & = f(z) g_4(z) \\
d_*(z) & = f(z) g_*(z) .
\end{align*}

Solving for $f(z)$ in terms of the $g_{i}(z)$ and $g(z)$ yields the required result. \qedhere
\end{proof}

\begin{comment}
We call the power series
\[
G(z) = g(z) - g_{1}(z)g_{3}(z) - g_{2}(z)g_{4}(z)
\]
the \emph{amended $G$-sequence} of $\mathcal{C}^{\circ}$. The generating function of $\boxplus\left(\mathcal{C}^{\circ}\right)$ is then given by the sequent operator not of $g(z)$ but of $G(z)$. We note something curious however: our usual intuition for the sequent operator is that it takes as its input a sequence of positive integers describing the number of distinguishable types of indecomposable `block' of each length $n$ in some collection; the sequent operator then gives us the enumeration sequence of the set of compound structures we can form by gluing any number of these blocks together. But this intuition breaks down when allowing amended $G$-sequences as inputs for the sequent operator as $G(z)$ can have negative coefficients. Of course, if $G(z)$ is the amended $G$-sequence of a centred permutation class then $Seq(G(z))=(1-G(z))^{-1}$ necessarily has non-negative coefficients by dint of the fact that this is the generating function of a combinatorial class. There are in fact $\boxplus$-closed centred permutation classes whose amended $G$-sequences become unbounded and negative and yet (of course) their sequent operators give entirely positive sequences.
\end{comment}

\subsection{The Exponential Growth Theorem} \label{sec:2.4}

We now have a method for determining the generating function of a $\boxplus$-closed centred permutation class. This may seem initially to be of limited interest, as we are mostly interested in the underlying (uncentred) classes, and these will of course have different generating functions to their centred counterparts. We recall, however, that if $\mathcal{C}^{\circ}$ is a centred permutation class with underlying class $\mathcal{C}$ then $\mathcal{C}^{\circ}$ and $\mathcal{C}$ have the same (upper) growth rate by Proposition \ref{eqgrs}, and we can calculate this from the generating function of the centred class $\mathcal{C}^{\circ}$ by the following specialisation of Theorem \ref{EGT0}.

\begin{thm}[Exponential Growth Theorem for Centred Permutation Classes]\label{EGT}
Suppose that $f(z)$ is the generating function of a proper centred permutation class $\mathcal{C}^{\circ}$. Then the upper growth rate of $\mathcal{C}^{\circ}$ is equal to the reciprocal of the radius of convergence of $f(z)$. This radius of convergence is equal to
\[
R = \sup\left\{r \geq 0 \mid \ \text{$f(z)$ converges on $[0,r)$}\right\}.
\]
\end{thm}

\begin{proof}
Immediate from Theorem \ref{EGT0}, on noting that $f(z)$ has non-negative coefficients. \qedhere
\end{proof}

Suppose now that $\mathcal{C}^{\circ}$ is a proper centred permutation class with amended $G$-sequence $G(z)$. Then the generating function of $\boxplus\left(\mathcal{C}^{\circ}\right)$ is given by
\[
f(z) = \frac{1}{1 - G(z)}.
\]
By Theorem \ref{EGT} the growth rate of $\boxplus\left(\mathcal{C}^{\circ}\right)$ is then equal to the reciprocal of the radius of convergence of $f(z)$. At first glace it looks like this radius of convergence should be equal to the modulus of the smallest solution of $G(z)=1$ (and we will see that this is true for the classes that we are interested in in this paper) but we need to be careful about $G(z)$ itself possibly having singularities of smaller magnitude than any solutions of $G(z)=1$. For the time-being we will deal with this issue on a case-by-case basis.

\subsection{Examples} \label{sec:2.5}

We conclude this section by determining the growth rates of some specific $\boxplus$-closed centred permutation classes. The process outlined by Theorems \ref{genfuncspec}  and \ref{EGT} for determining the generating functions and growth rates of $\boxplus$-closed centred permutation classes will be used freely without comment going forward, for example in the many growth rate calculations of $\boxplus$-closures we will undertake in the long proofs in Sections \ref{sec:3} and \ref{sec:4}.

Theorems \ref{genfuncspec}  and \ref{EGT} are remarkably useful in tandem, suggesting as they do a powerful method for constructing a large number of centred permutation classes whose growth rates we can easily calculate. For example, the $\boxplus$-closure of a given finite set of $\boxplus$-indecomposables is now easy to calculate, as demonstrated by the following examples:

\begin{example}[$\boxplus$-closure of $41\underline{3}52$]
We begin by noting that $\pi^{\circ} = 41\underline{3}52$ is a $\boxplus$-indecomposable centred permutation --- see Figure \ref{fig:pi41352} for an illustration.

\begin{figure}[h]
\begin{center}
\begin{tikzpicture}[scale=0.35]

\node[circle, draw, fill=none, inner sep=0pt, minimum width=\plotptradius] (0) at (3,3) {}; 
\node[permpt] (1) at (1,4) {};
\node[permpt] (2) at (2,1) {};
\node[permpt] (3) at (4,5) {};
\node[permpt] (4) at (5,2) {};

\draw[thick] (0,3) -- ++ (6,0);
\draw[thick] (3,0) -- ++ (0,6);
	
\end{tikzpicture}
\end{center}
\caption{The centred permutation $\pi^{\circ} = 41\underline{3}52$. Note that $\pi^{\circ}$ is a $\boxplus$-indecomposable centred permutation which strictly contains precisely four $\boxplus$-indecomposables, namely the four single point centred permutations.}
\label{fig:pi41352}
\end{figure}

 We shall determine the growth rate of its $\boxplus$-closure $\mathcal{C}^{\circ} = \boxplus(\pi^{\circ})$, that is, the smallest $\boxplus$-closed class containing $\pi^{\circ}$. As $\mathcal{C}^{\circ}$ is (by definition) $\boxplus$-closed we need only enumerate the $\boxplus$-indecomposables in $\mathcal{C}^{\circ}$ in each quadrant.

First, the only $\boxplus$-indecomposables contained in $\pi^{\circ}$ (and hence in $\mathcal{C}^{\circ}$) are $\pi^{\circ}$ itself and the four single point centred permutations. Hence the generating function of the $\boxplus$-indecomposables in $\mathcal{C}^{\circ}$ is $g(z) = 4z + z^4$. Of course, as the only single-quadrant $\boxplus$-indecomposables in this list are the four single point permutations, the generating functions $g_{i}(z)$ of the single-quadrant $\boxplus$-indecomposables in $\mathcal{C}^{\circ}$ are given by $g_{1}(z) = g_{2}(z) = g_{3}(z) = g_{4}(z) = z$. Hence the generating function of $\mathcal{C}^{\circ}$ is given by:
\[
\begin{split}
f(z) & = \frac{1}{1 - [(4z + z^4) -  2z^2]} \\
 & = \frac{1}{1 - 4z + 2z^2 - z^4}
\end{split}
\]
And so by Theorem \ref{EGT} we can deduce that the growth rate of $gr(\mathcal{C}^{\circ})$ is the reciprocal of the modulus of the smallest root of the denominator, which is approximately $3.44372$.
\end{example}

\begin{example}[The $\mathcal{X}$-class] \label{ex:Xclass}
Next, we consider the centred permutation class $\mathcal{X}^{\circ}$, defined to be the $\boxplus$-closure of the four single point centred permutations: $\mathcal{X}^{\circ} = \boxplus\{(\nept),(\nwpt),(\swpt),(\sept)\}$. Then for $\mathcal{X}^{\circ}$ we have
\[
g(z) = 4z
\]
and
\[
g_{1}(z) = g_{2}(z) = g_{3}(z) = g_{4}(z) = z.
\]
Hence $\mathcal{X}^{\circ}$ has generating function
\[
f(z) = \frac{1}{1 - 4z + 2z^2} \ .
\]
We can now apply the Exponential Growth Theorem \ref{EGT} to deduce that $\mathcal{X}^{\circ}$ has growth rate $2 + \sqrt{2} \approx 3.41421$.

\begin{comment}
The underlying permutation class $\mathcal{X}$ of $\mathcal{X}^{\circ}$ has been studied previously in the literature: this is the $\mathcal{X}$-class, introduced by Waton~\cite{waton:on-permutation-:} in the context of picture classes. Informally, this is the permutation class consisting of all permutations which can be drawn on an $X$. Waton showed that this class has generating function
\[
f(z) = \frac{z(1 - 2z)}{1 - 4z + 2z^2} \ .
\]
Note that this has the same denominator as the centred class $\mathcal{X}^{\circ}$, demonstrating, as expected by Proposition \ref{eqgrs}, that these two classes have the same growth rate.
\end{comment}

\end{example}

\section{Pin Words and Pin Classes} \label{sec:3}

In this section we introduce pin permutation classes and outline a scheme for the enumeration of a large subfamily: the recurrent pin classes. We begin in Section \ref{sec:3.1} by recalling the definition of pin words as (finite or infinite) words over the pin alphabet $\mathbb{A}_{P}$ subject to some specific construction rules, and we explain a process, known as the \emph{$\pi$-map}, by which a finite pin word can be used to generate a permutation, and an infinite pin word can generate a (centred or regular) permutation class. We describe how the pin permutations generated by this procedure are in a sense generalisations of the oscillations, as defined in the introduction. In Section \ref{sec:3.2} we define a factor containment relation on the set of pin words, and describe how this maps on to centred permutation containment. This will allow us to describe a particularly important subfamily of pin classes: the \emph{recurrent pin classes}, those generated by the $\pi$-map from a recurrent infinite pin word. In Section \ref{sec:3.3} we use the $\boxplus$-sum to describe a structure theorem for centred pin classes, the \emph{pin decomposition}, which we use to prove that every recurrent pin class is $\boxplus$-closed. We use this fact in Section \ref{sec:3.4} to demonstrate that every recurrent pin class has a proper growth rate. In Section \ref{sec:3.5} we state a classification (whose proof is relegated to the appendix to this paper) of certain sets of pin words which behave unexpectedly under the $\pi$-map, namely collisions (sets of pin words which generate the same centred permutation) and pin words which generate $\boxplus$-decomposable centred permutations, which we use in Section \ref{sec:3.6} to outline a procedure for determining the growth rate and generating function of any recurrent pin class. We conclude by applying this procedure to obtain examples of recurrent pin class enumerations, focusing on pin classes contained entirely in two quadrants in Section \ref{sec:3.7} before moving on to three and four quadrants in Section \ref{sec:3.8}.

\subsection{Pin Words} \label{sec:3.1}

We recall the definition of a finite pin word, following Brignall, Ru\v{s}kuc and Vatter~\cite{brignall:simple-permutat:b}:

\begin{defn}[Finite Pin Words] \label{def:finpinword}
A finite word $w$ over the pin alphabet $\mathbb{A}_{P} = \left\{1,2,3,4,u,d,l,r\right\}$ is a \emph{finite pin word} if:
\begin{itemize}
\item The first letter of $w$ is one of the numerals $1,2,3,4$.
\item Every further letter of $w$ is one of $u,d,l,r$.
\item If the $n$th letter of $w$ is in $\left\{u,d\right\}$ then the $(n+1)$th letter (if it exists) of $w$ is in $\left\{l,r\right\}$.
\item If the $n$th letter of $w$ is in $\left\{l,r\right\}$ then the $(n+1)$th letter (if it exists) of $w$ is in $\left\{u,d\right\}$.
\end{itemize}
We refer to the set of all finite pin words as the \emph{language of finite pin words}, denoted $\mathcal{L}_{P}$.
\end{defn}

We plan to use pin words to construct centred permutations, a construction described below in Definition \ref{def:pimap}. Under this correspondence we think of the letters $1,2,3$ and $4$ as referring to the four quadrants of a centred permutation --- we thus refer to these as the \emph{quadrant numerals}. Similarly, we think of the letters $u,d,l$ and $r$ as representing the directions up, down, left and right, respectively. We then say that the letters $u$ and $d$ have \emph{vertical alignment} and the letters $l$ and $r$ have \emph{horizontal alignment}. We can then informally describe a pin word as consisting of an initial quadrant numeral from $\{1,2,3,4\}$ followed by a sequence of letters from $\{u,l,d,r\}$ which must alternate between horizontal and vertical alignment. Note that this implies that there are $4$ pin words of length $1$ and $2^{n+2}$ pin words of every length $n \geq 2$.

\begin{example}
The words $w_{1} = 3ldrdrdlurdl$, $w_{2} = 4urdldr$ and $w_{3} = 2$ are all finite pin words, whereas $w_{4} = urd3ldr$ (which has the quadrant numeral in the wrong place) and $w_{5} = 1uruldur$ (which has two consecutive letters of the same alignment) are not.
\end{example}

We will also have cause to refer to infinite pin words:

\begin{defn}[Infinite Pin Words]
An \emph{infinite pin word} $w$ is an infinite word over the alphabet $\left\{1,2,3,4,u,d,l,r\right\}$ for which every finite prefix is a finite pin word.
We write $\mathcal{L}^{\infty}_{P}$ for the set of all infinite pin words.
\end{defn}

Informally, an infinite pin word is simply an infinite word following the same construction rules given in Definition \ref{def:finpinword}.

\begin{example}
The word
\[
w_{1} = 4\overline{urul} = 4urulurulurulurul\dots
\]
is an infinite pin word\footnote{We use the notation $\overline{f}$ to denote the word $f$ recurring indefinitely.}, as is
\[
w_{2} = 2ldluldldluldldldluldldldldlu\dots,
\]
in which $l$ is the only horizontally-aligned letter, and $u$'s are separated by increasing sequences of $d$'s.
\end{example}

We shall also have cause to refer to the sets $\overline{\mathcal{L}}_{P}$ and $\overline{\mathcal{L}}^{\infty}_{P}$, consisting of all finite and infinite pin words, respectively, with the initial quadrant numeral removed. Note that a finite or infinite pin word consists of an initial numeral between $1$ and $4$ followed by a (possibly empty) pin word in $\overline{\mathcal{L}}_{P}$ and $\overline{\mathcal{L}}^{\infty}_{P}$.

As discussed informally in the introduction, there is a natural procedure by which a finite pin word $w$ can be converted into a centred permutation $\pi_{w}^{\circ}$ of the same length (see Figure \ref{fig:pinpermexample} for an illustration of this process):

\begin{defn}[The $\pi$-map and pin permutations] \label{def:pimap}

Let $S^{\circ}$ denote the set of all centred permutations. The $\pi$-\emph{map} is the map
\[
\pi^{\circ}: \mathcal{L}_{P} \rightarrow S^{\circ}
\]
defined by the following procedure. Given $w \in \mathcal{L}_{P}$ of length $n=|w|$:

\begin{enumerate}
\item Place the initial origin point $p_{0}$ somewhere in the plane. Use this point to split the plane into four numbered quadrants, as in Figure \ref{fig:quadnumbering} on page \pageref{fig:quadnumbering}.
\item Place the first point $p_{1}$ in the quadrant specified by the initial numeral of the pin word $w$.
\item Once the first $k-1$ points have been placed, place the point $p_{k}$ either up, down, left or right (depending on the letter $u$, $d$, $l$, or $r$ appearing in the $k$-th place in the pin word) of the bounding rectangle of all previous points $\{p_{0},p_{1}, \dots , p_{k-1} \}$ at the end of a 'pin' which separates the last point $p_{k-1}$ from all previous points.
\item Once all $n$ points have been placed, read off the centred permutation $\pi^{\circ}_{w}$ given by the points  $\{p_{0}, p_{1}, \dots , p_{n} \}$, with $p_{0}$ as the origin.
\end{enumerate}

Note that the $\pi$-map is length-preserving: if $w$ is a finite pin word of length $n$ then $\pi^{\circ}_{w}$ is a centred permutation of length $n$. We refer to $\pi^{\circ}_{w}$ as the \emph{centred pin permutation generated by $w$}; similarly, we write $\pi_{w}$ for the underlying uncentred permutation of $\pi^{\circ}_{w}$ and refer to this as the \emph{pin permutation generated by $w$}.

We refer to a centred permutation $\sigma^{\circ}$ as a \emph{centred pin permutation} if there is some pin word $w$ such that $\sigma^{\circ} \leq \pi^{\circ}_{w}$. Similarly, an uncentred permutation $\sigma$ is a \emph{pin permutation} if it is the underlying permutation of a centred pin permutation.

Finally, we refer to a centred permutation $\sigma^{\circ}$ as a \emph{contiguous} centred pin permutation if there is some pin word $w$ such that $\sigma^{\circ} = \pi^{\circ}_{w}$.

\end{defn}

\begin{figure}[h]
% Here you can include a picture (.pdf, .png,...) with \includegraphics[]{} or a tikz plot:
\begin{center}
\begin{tikzpicture}[scale=0.5]

\node[circle, draw, fill=none, inner sep=0pt, minimum width=\plotptradius] (0) at (6,4) {};

\node[permpt,label={\tiny $p_{1}$}] (1) at (5,6) {}; 
\node[permpt,label={\tiny $p_{2}$}] (2) at (3,5) {}; \draw[thin] (2) -- ++ (2.5,0);
\node[permpt,label={\tiny $p_{3}$}] (3) at (4,8) {}; \draw[thin] (3) -- ++ (0,-3.5);
\node[permpt,label={\tiny $p_{4}$}] (4) at (8,7) {}; \draw[thin] (4) -- ++ (-4.5,0);
\node[permpt,label={[yshift=-16pt]{\tiny $p_{5}$}}] (5) at (7,2) {}; \draw[thin] (5) -- ++ (0,5.5);
\node[permpt,label={\tiny $p_{6}$}] (6) at (1,3) {}; \draw[thin] (6) -- ++ (6.5,0);
\node[permpt,label={[yshift=-16pt]{\tiny $p_{7}$}}] (7) at (2,1) {}; \draw[thin] (7) -- ++ (0,2.5);

\draw[thick] (6,0) -- ++ (0,9);
\draw[thick] (0,4) -- ++ (9,0);

\end{tikzpicture}
\end{center}
\caption{The contiguous centred pin permutation $31586\underline{4}27$ (and its underlying uncentred permutation $3147526$), constructed from the pin word $2lurdld$.}
\label{fig:pinpermexample}
\end{figure}

\begin{obs} \label{slicing}
By definition, the point $p_{k}$ slices the bounding rectangle $rec(p_{0},p_{1}, \dots , p_{k-1})$ of all previous points, including the origin. In fact, the definition implies that $p_{k}$ is the \emph{only} point after $p_{k-1}$ that slices this rectangle; this means that if we remove $p_{k}$ from the corresponding centred permutation then $rec(p_{0},p_{1}, \dots , p_{k-1})$ becomes a $\circ$-interval, and hence, if $p_{k}$ is not the last point, the resulting centred permutation is $\boxplus$-decomposable. This will be important in our later structure theorem for a pin class.
\end{obs}

We denote the class consisting of \emph{all} centred pin permutations $\mathcal{P}^{\circ}$, and its uncentred counterpart by $\mathcal{P}$. We call this the \emph{complete pin class}. As noted in the introduction, Bassino, Bouvel and Rossin~\cite{bassino:enumeration-of-:} proved that $\mathcal{P}$ has rational generating function
\[ 
f(z) = 1 + \frac{z - 6z^2 + 9z^3 - 12z^4 + 4z^5 + 20z^6 - 8z^7}{1 - 8z + 19z^2 - 26z^3 + 14z^4 - 12z^5 - 8z^6 + 20z^7 - 8z^8},
\]
and (by applying the Exponential Growth Theorem \ref{EGT1} to this function) growth rate $\omega_{\infty} \approx 5.24112$. We shall later find the generating function of $\mathcal{P}^{\circ}$, which of course has the same growth rate as $\mathcal{P}$ by Proposition \ref{eqgrs}. We are primarily interested in this paper in certain subclasses of $\mathcal{P}$, known as \emph{pin classes}. These are defined somewhat analogously to pin permutations: just as we can convert a finite pin word into a (centred or uncentred) permutation, we can also convert an infinite pin word into a (centred or uncentred) pin class:

\begin{defn}[Pin Classes]
Suppose that $w \in \mathcal{L}^{\infty}_{P}$. Let $w_{1,n}$ denote the prefix of $w$ of length $n$ (note that this is itself a pin word) and let $\pi^{\circ}_{n}$ be the centred pin permutation generated by $w_{n}$ via the process described in Definition \ref{def:pimap}. Consider the set
\[
\Pi = \{ \pi^{\circ}_{1}, \pi^{\circ}_{2}, \pi^{\circ}_{3}, \dots \} \ .
\]
The downward closure of $\Pi$ under the pattern containment order $\leq$ forms a centred permutation class $\mathcal{C}^{\circ}_{w}$, known as the \emph{centred pin class} of $w$. We similarly define the \emph{uncentred pin class} of $w$, $\mathcal{C}_{w}$, to be the underlying permutation class of $\mathcal{C}^{\circ}_{w}$.
\end{defn}

Informally, we can think of this as using the infinite pin word $w$ to draw an 'infinite diagram' by the same process as in the finite case; the pin class $\mathcal{C}_{w}$ is then the class of all permutations that can be found somewhere within this infinite diagram (by picking a finite collection of points and forgetting everything else).

\begin{example}
Taking $w=1\overline{(ldrdluru)}$ we generate the pin diagram shown in Figure \ref{fig:fountain}. We could of course extend this diagram indefinitely. The class $\mathcal{C}_{w}$ then consists of all permutations that can be found somewhere in this infinite diagram (with $\mathcal{C}^{\circ}_{w}$ being the class of all \emph{centred} permutations that can be found).
\end{example}

\begin{figure}[h]
% Here you can include a picture (.pdf, .png,...) with \includegraphics[]{} or a tikz plot:
\begin{center}
\begin{tikzpicture}[scale=0.35]

\node[circle, draw, fill=none, inner sep=0pt, minimum width=\plotptradius] (0) at (9,9) {};

\node[permpt] (1) at (10,11) {}; 
\node[permpt] (2) at (7,10) {}; \draw[thin] (2) -- ++ (3.5,0);
\node[permpt] (3) at (8,7) {}; \draw[thin] (3) -- ++ (0,3.5);
\node[permpt] (4) at (12,8) {}; \draw[thin] (4) -- ++ (-4.5,0);
\node[permpt] (5) at (11,5) {}; \draw[thin] (5) -- ++ (0,3.5);
\node[permpt] (6) at (5,6) {}; \draw[thin] (6) -- ++ (6.5,0);
\node[permpt] (7) at (6,13) {}; \draw[thin] (7) -- ++ (0,-7.5);
\node[permpt] (8) at (14,12) {}; \draw[thin] (8) -- ++ (-8.5,0);
\node[permpt] (9) at (13,15) {}; \draw[thin] (9) -- ++ (0,-3.5);
\node[permpt] (10) at (3,14) {}; \draw[thin] (10) -- ++ (10.5,0);
\node[permpt] (11) at (4,3) {}; \draw[thin] (11) -- ++ (0,11.5);
\node[permpt] (12) at (16,4) {}; \draw[thin] (12) -- ++ (-12.5,0);
\node[permpt] (13) at (15,1) {}; \draw[thin] (13) -- ++ (0,3.5);
\node[permpt] (14) at (1,2) {}; \draw[thin] (14) -- ++ (14.5,0);
\node[permpt] (15) at (2,17) {}; \draw[thin] (15) -- ++ (0,-15.5);
\node[permpt] (16) at (17,16) {}; \draw[thin] (16) -- ++ (-15.5,0);

\draw[thick] (9,0) -- ++ (0,18);
\draw[thick] (0,9) -- ++ (18,0);

\end{tikzpicture}
\end{center}
\caption{The beginning of the pin diagram generated by the infinite pin word $w = 1\overline{(ldrdluru)}$. The pin class $\mathcal{C}_{w}$ consists of all permutations that can be found somewhere in this infinite diagram, whereas $\mathcal{C}^{\circ}_{w}$ consists of all centred permutations which can be found (with the origin staying in place).}
\label{fig:fountain}
\end{figure}

\subsubsection*{Oscillations} \label{sec:3.1.1.1}

One particular collection of pin permutations is already well-known and will play a key role in our later theory: the \emph{oscillations}, as briefly introduced in Figure \ref{fig:osc0} on page \pageref{fig:osc0}. In fact, we have so far only mentioned the \emph{increasing} oscillations. We therefore define more generally:

\begin{defn}[Oscillations]
Let $w$ be a finite pin word that stays in its initial quadrant (for example, $1ururur$, $1rururur$, $3ldl$ and $4drdrd$). We call the centred permutation $\pi^{\circ}_{w}$ generated by $w$ an \emph{oscillation} (sometimes a \emph{one-quadrant oscillation} for emphasis).
\end{defn}

Oscillations have natural `staircase' structures --- see Figure \ref{fig:oscexamples} for examples. Note that there are precisely two oscillations of each length $n\geq3$ in each quadrant (e.g. those generated by $1ururu$ and $1rurur$ of length $6$ in the first quadrant), but only one of length $2$ (as $1u$ and $1r$ in fact generate the same centred permutation (\netwo). See Figure \ref{fig:2osccollision} for an illustration).

\begin{figure}[h]
\begin{center}
\begin{tikzpicture}[scale=0.3]

\node[circle, draw, fill=none, inner sep=0pt, minimum width=\plotptradius] (0) at (0,0) {};
\node[permpt] (1) at (2,1) {}; 
\node[permpt] (2) at (1,3) {}; \draw[thin] (2) -- ++ (0,-2.5);
\node[permpt] (3) at (4,2) {}; \draw[thin] (3) -- ++ (-3.5,0);
\node[permpt] (4) at (3,5) {}; \draw[thin] (4) -- ++ (0,-3.5);
\node[permpt] (5) at (6,4) {}; \draw[thin] (5) -- ++ (-3.5,0);
\node[permpt] (6) at (5,6) {}; \draw[thin] (6) -- ++ (0,-2.5);

\draw[thick] (-1,0) -- ++ (8,0);
\draw[thick] (0,-1) -- ++ (0,8);

\node at (0,-2) {$1ururu$};
\node at (0,-4) {$\underline{1}426375$};

\begin{scope}[shift={(12,0)}]

\node[circle, draw, fill=none, inner sep=0pt, minimum width=\plotptradius] (0) at (0,0) {};
 
\node[permpt] (1) at (1,2) {};
\node[permpt] (2) at (3,1) {}; \draw[thin] (2) -- ++ (-2.5,0);
\node[permpt] (3) at (2,4) {}; \draw[thin] (3) -- ++ (0,-3.5);
\node[permpt] (4) at (5,3) {}; \draw[thin] (4) -- ++ (-3.5,0);
\node[permpt] (5) at (4,6) {}; \draw[thin] (5) -- ++ (0,-3.5);
\node[permpt] (6) at (6,5) {}; \draw[thin] (6) -- ++ (-2.5,0);

\draw[thick] (-1,0) -- ++ (8,0);
\draw[thick] (0,-1) -- ++ (0,8);

\node at (0,-2) {$1rurur$};
\node at (0,-4) {$\underline{1}352746$};

\end{scope}

\begin{scope}[shift={(29,0)}]

\node[circle, draw, fill=none, inner sep=0pt, minimum width=\plotptradius] (0) at (0,0) {};
\node[permpt] (1) at (-1,2) {}; 
\node[permpt] (2) at (-3,1) {}; \draw[thin] (2) -- ++ (2.5,0);
\node[permpt] (3) at (-2,4) {}; \draw[thin] (3) -- ++ (0,-3.5);
\node[permpt] (4) at (-5,3) {}; \draw[thin] (4) -- ++ (3.5,0);
\node[permpt] (5) at (-4,5) {}; \draw[thin] (5) -- ++ (0,-2.5);

\draw[thick] (1,0) -- ++ (-7,0);
\draw[thick] (0,-1) -- ++ (0,8);

\node at (0,-2) {$2lulu$};
\node at (0,-4) {$46253\underline{1}$};

\end{scope}

\end{tikzpicture}
\end{center}
\caption{Three oscillations and the pin words that generate them. Note the two distinct oscillations of length $6$ in the first quadrant. In general, there are two distinct oscillations of each length $n \geq 3$ in each quadrant.}
\label{fig:oscexamples}
\end{figure}

\begin{figure}[h]
\begin{center}
\begin{tikzpicture}[scale=0.3]

\begin{scope}[shift={(10,0)}]

\node[circle, draw, fill=none, inner sep=0pt, minimum width=\plotptradius] (0) at (0,0) {};
\node[permpt] (1) at (1,2) {};
\node[permpt] (2) at (2,1) {}; \draw[thin] (2) -- ++ (-1.5,0);

\node[empty] (-1) at (0,4) {};

\draw[thick] (-3,0) -- ++ (6,0);
\draw[thick] (0,-3) -- ++ (0,6);

\node[] (4) at (0,-5) {$1r$};

%\draw[dotted] (0.5,8.5) rectangle (1.5,9.5);

\end{scope}

\node[circle, draw, fill=none, inner sep=0pt, minimum width=\plotptradius] (0) at (0,0) {};
\node[permpt] (1) at (2,1) {};
\node[permpt] (2) at (1,2) {}; \draw[thin] (2) -- ++ (0,-1.5);

\draw[thick] (-3,0) -- ++ (6,0);
\draw[thick] (0,-3) -- ++ (0,6);

\node[] (3) at (5,0) {=};

%\node[] (3) at (5,0) {=};

\node[] (4) at (0,-5) {$1u$};

%\draw[dotted] (0.5,8.5) rectangle (1.5,9.5);

\end{tikzpicture}
\end{center}
\caption{Whilst there are two distinct oscillations in each quadrant for all lengths $n \geq 3$, there is only one at length $2$. This is because the pin words $1u$ and $1r$ in fact generate the \emph{same} centred permutation $\underline{1}32$. This is an example of a \emph{collision} of pin factors, a phenomenon that we will study in detail later.}
\label{fig:2osccollision}
\end{figure}

We can also consider the pin class $\mathcal{O}^{\circ}$ of \emph{increasing oscillations}, defined to be $\mathcal{C}^{\circ}_{w}$ for the pin word $1\overline{(ru)}$ (see Figure \ref{fig:oscclass}). This class\footnote{Or more precisely, the underlying uncentred class $\mathcal{O}$, but this is essentially identical as $\mathcal{O}^{\circ}$ is contained entirely in the first quadrant.} has been shown by Vatter~\cite[Proposition 1.10]{vatter:small-permutati:} to have growth rate $\kappa \approx 2.20557$, a fact which we shall be able to prove later. We will also be able to prove that $\kappa$ is in fact the \emph{smallest} possible growth rate of a pin class, and is only achievable by pin classes which are `essentially' $\mathcal{O}^{\circ}$.

\begin{figure}[h]
\begin{center}
\begin{tikzpicture}[scale=0.25]

\node[circle, draw, fill=none, inner sep=0pt, minimum width=\plotptradius] (0) at (0,0) {};
\node[permpt] (1) at (1,2) {};
%\node[permpt] (1a) at (1.2,0.6) {};
%\node[permpt] (1b) at (0.6,1.2) {};
\node[permpt] (2) at (3,1) {}; \draw[thin] (2) -- ++ (-2.5,0);
\node[permpt] (3) at (2,4) {}; \draw[thin] (3) -- ++ (0,-3.5);
\node[permpt] (4) at (5,3) {}; \draw[thin] (4) -- ++ (-3.5,0);
\node[permpt] (5) at (4,6) {}; \draw[thin] (5) -- ++ (0,-3.5);
\node[permpt] (6) at (7,5) {}; \draw[thin] (6) -- ++ (-3.5,0);
\node[permpt] (7) at (6,8) {}; \draw[thin] (7) -- ++ (0,-3.5);
\node[permpt] (8) at (9,7) {}; \draw[thin] (8) -- ++ (-3.5,0);
\node[permpt] (9) at (8,10) {}; \draw[thin] (9) -- ++ (0,-3.5);
\node[permpt] (10) at (11,9) {}; \draw[thin] (10) -- ++ (-3.5,0);
%\node[permpt] (6a) at (5.6,6.6) {}; %\draw[thin] (6) -- ++ (0,-3.5);
%\node[permpt] (6b) at (6.2,7.2) {}; \draw[thin] (6.2,6.5) -- ++ (0,-3);
%\fitellipsis {1a} {1b};

%\fitellipsis {6a} {6b} ;

\draw[thick] (0,0) -- ++ (10,0);
\draw[thick] (0,0) -- ++ (0,12);

\end{tikzpicture}
\end{center}
\caption{The class $\mathcal{O}^{\circ}$ of \emph{increasing oscillations}, generated by the infinite pin word $1\overline{ru}$. Note that, due to being confined to only one quadrant, this is `essentially' the same as the uncentred class $\mathcal{O}$, which has been studied extensively in the literature.}
\label{fig:oscclass}
\end{figure}

Oscillations have been studied extensively in the literature. In fact, to a significant extent pin permutations are worth studying because they generalise the oscillations and end up being interesting for many of the same reasons. For example, Bevan~\cite{bevan:intervals} demonstrated that an infinite antichain contained in the two-point extension $\mathcal{O}^{+2}$ can be used to construct permutation classes (all of which contain long oscillations) at \emph{every} growth rate $\lambda_{B} \geq 2.35526$.

We have thus established a method for converting an infinite pin word into a (centred or uncentred) permutation class. This allows us to describe an extraordinarily large number of classes which would be awkward to describe by other means (it would, for example, take a lot of work to describe the class in Figure \ref{fig:fountain} in terms of its basis). There are two main reasons that this construction will prove useful. First, controlling the pin word allows us to control various properties of the permutation class produced (for example, the connection with simple permutations mentioned in Section \ref{sec:1.4} allows us to control the number of simples in a given pin class). Second, and most importantly, we have a process for determining the growth rates of the pin permutation classes produced: this will depend on an in-built structure possessed by pin permutation classes, the \emph{$\boxplus$-decomposition}.

\subsection{Pin Factors} \label{sec:3.2}

We will need to have a notion of a 'factor' of a pin word, one that should hopefully preserve pattern containment in both the centred and uncentred case:

\begin{defn}[Pin Factor]
Suppose that $w$ is a (finite or infinite) pin word. A \emph{pin factor} of $w$ is either:
\begin{itemize}
\item a finite prefix of $w$; or:
\item a finite non-initial contiguous subsequence of $w$ in which the first letter is replaced by the quadrant in which the corresponding point appears in $\pi_{w}^{\circ}$.
\end{itemize}
\end{defn}

Note that a pin factor of a given pin word is by definition itself a pin word. We write $w_{i,j}$ for the pin factor obtained by taking the contiguous subsequence of $w$ between the $i$th and $j$th place.

\begin{example}
Consider the pin word $w = 2ruldlurdru$. Then $w_{1,2} = 2r$, $w_{1,4} = 2rul$, $w_{1,9} = 2ruldlurd$ and $w$ itself are all pin factors of $w$ as they are all finite initial subsequences. If instead we take out the contiguous subsequence from the fifth to ninth terms we obtain the word $dlurd$. This is a factor in the word sense, but it is not a pin factor as it does not begin with a numeral (and so is not a pin word). However, if we draw out the centred permutation $\pi_{w}^{\circ}$ generated by $w$ (see Figure \ref{fig:pinfactorexample}) we can see that the fifth point placed is in the $3$rd quadrant, so we replace the initial $d$ in $dlurd$ with a $3$ to obtain the pin word $3lurd$, which is now a pin factor of $w$; specifically $w_{5,9} = 3lurd$. For an internal pin word we can always work our what this initial numeral should be by looking at the previous letter: for example, if we wish to work out $w_{9,11}$ we first take out the subword $dru$ (between the $9$th and $11$th places); to work out which numeral we replace the initial $d$ with, note that it was preceded by an $r$ and the $d$ in any $rd$ clearly must be placed in the $4$th quadrant. Hence $w_{9,11} = 4ru$. This suggests a correspondence between pin factors of length $n$ and $\overline{\mathcal{L}}_{P}$-subwords of length $n+1$, stated below in Lemma \ref{pfsubwcorr}.

\begin{figure}[h]
\begin{center}
\begin{tikzpicture}[scale=0.35]

\node[circle, draw, fill=none, inner sep=0pt, minimum width=\plotptradius] (0) at (6,5) {};
\node[permpt] (1) at (5,7) {}; 
\node[permpt] (2) at (8,6) {}; \draw[thin] (2) -- ++ (-3.5,0);
\node[permpt] (3) at (7,9) {}; \draw[thin] (3) -- ++ (0,-3.5);
\node[permpt] (4) at (3,8) {}; \draw[thin] (4) -- ++ (4.5,0);
\node[permpt,blue] (5) at (4,3) {}; \draw[thin] (5) -- ++ (0,5.5);
\node[permpt,blue] (6) at (1,4) {}; \draw[thin] (6) -- ++ (3.5,0);
\node[permpt,blue] (7) at (2,11) {}; \draw[thin] (7) -- ++ (0,-7.5);
\node[permpt,blue] (8) at (10,10) {}; \draw[thin] (8) -- ++ (-8.5,0);
\node[permpt,blue] (9) at (9,1) {}; \draw[thin] (9) -- ++ (0,9.5);
\node[permpt] (10) at (12,2) {}; \draw[thin] (10) -- ++ (-3.5,0);
\node[permpt] (11) at (11,12) {}; \draw[thin] (11) -- ++ (0,-10.5);

\draw[thick] (6,0.5) -- ++ (0,12);
\draw[thick] (0.5,5) -- ++ (12,0);

\draw[dotted] (3.5,2.5) rectangle (4.5,3.5);

\end{tikzpicture}
\end{center}
\caption{The construction of the centred permutation $\pi_{w}^{\circ}$ from the pin word $w = 2ruldlurdru$. Note that the fifth-placed point, $p_{5}$, corresponding to the first $d$ in the word, is in the $3$rd quadrant, so when we extract the word factor $dlurd$ from $w$ we replace the initial $d$ with a $3$ to obtain the pin factor $\tilde{w} = 3lurd$. Note that the centred permutation $\sigma_{\tilde{w}}^{\circ}$ (highlighted in blue) is a subpermutation of $\pi_{w}^{\circ}$.}
\label{fig:pinfactorexample}
\end{figure}

\end{example}

As Figure \ref{fig:pinfactorexample} suggests, this notion preserves pattern containment:

\begin{obs} \label{pinfactorcontainment}
Suppose $w_{1},w_{2}$ are finite pin words and that $w_{1}$ is a pin factor of $w_{2}$. Then:
\[
\pi^{\circ}_{w_{1}} \leq \pi^{\circ}_{w_{2}} \ .
\]
\end{obs}

It's also worth noting that, while pin factors of $w$ of length $n$ cannot in general be identified with subwords of $w$ of length $n$ (for example, the subword $ulur$ could instantiate either $1lur$ or $2lur$ depending on the context) they can \emph{almost} be identified with subwords of length $n+1$ (with the first letter helping to determine the quadrant numeral) --- this is the content of Lemma \ref{pfsubwcorr}, stated below.

We now have the terminology to define an important category of infinite pin words, whose associated pin classes (we shall see) have a structure which will be particularly amenable to the enumeration methods developed in Section \ref{sec:2}:

\begin{defn}[Recurrent Pin Words] \label{def:pinwordrecurrence}
We refer to an infinite pin word $w$ as \emph{recurrent} if every pin factor that occurs in $w$ occurs infinitely-often; similarly, $w$ is \emph{eventually recurrent} if every pin factor that occurs after a certain point occurs infinitely-often. We refer to a pin class $\mathcal{C}^{\circ}_{w}$ generated by a recurrent pin word $w$ as a \emph{recurrent pin class}.
\end{defn}

Note that this definition is subtly different from recurrence in the usual (factor) sense, due to complications introduced by initial subsequences (which begin with a numeral): these initial words can never occur as factors again in $w$ by the definition of a pin word, but (in order for $w$ to be recurrent) they \emph{must} appear infinitely often as pin factors. This means that it may not immediately be obvious whether an infinite pin word is recurrent or not, even in the periodic case, without drawing its pin diagram: for example, $2\overline{(ul)}$ is recurrent, but $1\overline{(ul)}$ is not, as $w_{1,2}=1u$ never reoccurs as a pin factor.

We wish to ignore, as far as possible, this subtle and often irritating distinction between regular factors and pin factors. First, we call a pin factor of $w$ \emph{fully internal} if it occurs in $w$ starting after the second place; that is, fully internal pin factors are those of the form $w_{i,j}$ where $3 \leq i < j$. Note that fully internal pin factors can be found somewhere in $w$ with a preceeding letter (as opposed to numeral). This letter is enough to reconstruct the initial numeral of the pin factor, thus leading to the following\footnote{Here we use the term \emph{$\overline{\mathcal{L}}_{P}$-factor} to refer to a factor of an infinite pin word in the usual subword sense that is in $\overline{\mathcal{L}}_{P}$. In other words, this is simply a substring of a pin word, starting after the first position, for which we do not change the first letter to a numeral.}:

\begin{lemma} \label{pfsubwcorr}
Let $w$ be an infinite pin word. Then the number of fully internal pin factors of $w$ of length $n \geq 2$ is equal to the number of $\overline{\mathcal{L}}_{P}$-factors of $w$ of length $n+1$.
\end{lemma}

Of course, the fully internal restriction can be removed if we further assume that $w$ is recurrent, as in this case any pin factor beginning in the first or second place can also be found later on; hence all pin factors are fully internal, giving us the following simplification:

\begin{cor}\label{recpfsubwcorr}
Let $w$ be a recurrent infinite pin word.
Then the number of pin factors of $w$ of length $n$ is equal to the number of $\overline{\mathcal{L}}_{P}$-factors of $w$ of length $n+1$.
\end{cor}

\subsubsection*{Left-truncations}

We shall also require an infinite analogue of a pin factor of $w$, the infinite pin word obtained by starting $w$ at a later point:

\begin{defn}[Left-truncation of an infinite pin word]
Let $w$ be an infinite pin word and $n \in \mathbb{N}$. The $n$th \emph{left-truncation} of $w$, $w_{\geq n}$, is the infinite pin word obtained from $w$ by replacing the symbol in the $n$th place with the number of the quadrant in which $p_{n}$ is placed, and then removing all of the previous symbols.
\end{defn} 

Informally, $w_{\geq n}$ is `$w$ but we start in the $n$th place'. For example, if $w = 3rurdlurur\dots$ then $w_{\geq 4} = 1dlurur\dots$ and $w_{\geq 7} = 2rur\dots$. Again, we can always work out the numeral to replace the letter in the $n$th place with by looking at the previous symbol.

Crucially, left-truncating an infinite pin word does not affect the upper of lower growth rate of its pin class:

\begin{lemma}[Finite Prefix Lemma] \label{fplem}
Let $w$ be an infinite pin word and $n \in \mathbb{N}$. Then $\overline{gr}(\mathcal{C}^{\circ}_{w}) = \overline{gr}(\mathcal{C}^{\circ}_{w_{\geq n}})$ and $\underline{gr}(\mathcal{C}^{\circ}_{w}) = \underline{gr}(\mathcal{C}^{\circ}_{w_{\geq n}})$.
\end{lemma}

\begin{proof}
Recall that we write $\mathcal{C}^{\circ+k}$ for the $k$-point extension of $\mathcal{C}^{\circ}$ (that is, the class of centred permutations from $\mathcal{C}^{\circ}$ with at most $k$ extra points added anywhere). Note that the (infinite) pin diagram generated by $w$ is simply that of $w_{\geq n}$ with $n-1$ extra points, hence:
\[
\mathcal{C}^{\circ}_{w_{\geq n}} \subseteq \mathcal{C}^{\circ}_{w} \subseteq \mathcal{C}^{\circ+(n-1)}_{w_{\geq n}}
\]
and by Lemma \ref{eqgrs2} the classes on the left and right have the same (upper and lower) growth rates.
\end{proof}

We call this the \emph{Finite Prefix Lemma}, as it tells us that any finite prefix of an infinite pin word cannot affect the growth rate of the associated pin class, a fact that we shall use often.

\subsection{The Pin Decomposition} \label{sec:3.3}

The key idea we will use to enumerate pin classes comes from the following observation: when we take a centred permutation generated by a pin word and remove any interior point, we decompose the resulting permutation as the $\boxplus$-sum of two smaller pin permutations. Figure \ref{fig:boxdecomposition} illustrates this process.

\begin{figure}[h]
\begin{center}
\begin{tikzpicture}[scale=0.3]

\node[circle, draw, fill=none, inner sep=0pt, minimum width=\plotptradius] (0) at (7,6) {};
\node[permpt] (1) at (8,8) {}; 
\node[permpt] (2) at (5,7) {}; \draw[thin] (2) -- ++ (3.5,0);
\node[permpt] (3) at (6,4) {}; \draw[thin] (3) -- ++ (0,3.5);
\node[permpt] (4) at (3,5) {}; \draw[thin] (4) -- ++ (3.5,0);
\node[permpt] (5) at (4,2) {}; \draw[thin] (5) -- ++ (0,3.5);
\node[permpt] (6) at (10,3) {}; \draw[thin] (6) -- ++ (-6.5,0);
\node[permpt] (7) at (9,10) {}; \draw[thin] (7) -- ++ (0,-7.5);
\node[permpt] (8) at (12,9) {}; \draw[thin] (8) -- ++ (-3.5,0);
\node[permpt] (9) at (11,12) {}; \draw[thin] (9) -- ++ (0,-3.5);
\node[permpt] (10) at (1,11) {}; \draw[thin] (10) -- ++ (10.5,0);
\node[permpt] (11) at (2,1) {}; \draw[thin] (11) -- ++ (0,10.5);

\draw[thick] (7,0.5) -- ++ (0,12);
\draw[thick] (0.5,6) -- ++ (12,0);

\draw[dotted] (9.5,2.5) rectangle (10.5,3.5);

\node at (15,6) {$\rightarrow$};

\begin{scope}[shift={(17,0)}]
\node[circle, draw, fill=none, inner sep=0pt, minimum width=\plotptradius] (0) at (7,6) {};
\node[permpt] (1) at (8,8) {}; 
\node[permpt] (2) at (5,7) {}; \draw[thin] (2) -- ++ (3.5,0);
\node[permpt] (3) at (6,4) {}; \draw[thin] (3) -- ++ (0,3.5);
\node[permpt] (4) at (3,5) {}; \draw[thin] (4) -- ++ (3.5,0);
\node[permpt] (5) at (4,2) {}; \draw[thin] (5) -- ++ (0,3.5);

\node[permpt] (7) at (9,10) {};
\node[permpt] (8) at (12,9) {}; \draw[thin] (8) -- ++ (-3.5,0);
\node[permpt] (9) at (11,12) {}; \draw[thin] (9) -- ++ (0,-3.5);
\node[permpt] (10) at (1,11) {}; \draw[thin] (10) -- ++ (10.5,0);
\node[permpt] (11) at (2,1) {}; \draw[thin] (11) -- ++ (0,10.5);

\draw[thick] (7,0.5) -- ++ (0,12);
\draw[thick] (0.5,6) -- ++ (12,0);

%\draw[dotted] (2.5,0.5) -- ++ (0,12);
%\draw[dotted] (8.5,0.5) -- ++ (0,12);
%\draw[dotted] (0.5,1.5) -- ++ (12,0);
%\draw[dotted] (0.5,8.5) -- ++ (12,0);

\draw[pattern=crosshatch,pattern color=black!80,draw=none] (0.5,1.5) rectangle (2.5,8.5);
\draw[pattern=crosshatch,pattern color=black!80,draw=none] (8.5,1.5) rectangle (11.5,8.5);
\draw[pattern=crosshatch,pattern color=black!80,draw=none] (2.5,8.5) rectangle (8.5,12.5);
\draw[pattern=crosshatch,pattern color=black!80,draw=none] (2.5,0.5) rectangle (8.5,1.5);

\node at (15,6) {=};

\end{scope}

\begin{scope}[shift={(34,2)}]
\node[circle, draw, fill=none, inner sep=0pt, minimum width=\plotptradius] (0) at (5,4) {};
\node[permpt] (1) at (6,6) {}; 
\node[permpt] (2) at (3,5) {}; \draw[thin] (2) -- ++ (3.5,0);
\node[permpt] (3) at (4,2) {}; \draw[thin] (3) -- ++ (0,3.5);
\node[permpt] (4) at (1,3) {}; \draw[thin] (4) -- ++ (3.5,0);
\node[permpt] (5) at (2,1) {}; \draw[thin] (5) -- ++ (0,2.5);

\draw[thick] (5,0.5) -- ++ (0,6);
\draw[thick] (0.5,4) -- ++ (6,0);

\node at (9,4) {$\boxplus$};

\end{scope}

\begin{scope}[shift={(45,2)}]
\node[circle, draw, fill=none, inner sep=0pt, minimum width=\plotptradius] (0) at (3,2) {};
\node[permpt] (1) at (4,4) {}; 
\node[permpt] (2) at (6,3) {}; \draw[thin] (2) -- ++ (-2.5,0);
\node[permpt] (3) at (5,6) {}; \draw[thin] (3) -- ++ (0,-3.5);
\node[permpt] (4) at (1,5) {}; \draw[thin] (4) -- ++ (4.5,0);
\node[permpt] (5) at (2,1) {}; \draw[thin] (5) -- ++ (0,4.5);

\draw[thick] (3,0.5) -- ++ (0,6);
\draw[thick] (0.5,2) -- ++ (6,0);

\end{scope}

\end{tikzpicture}
\end{center}
\caption{When we remove the point $p_{6}$ from the pin permutation generated by $1ldldruruld$ we remove the only point that slices the bounding rectangle of the first five points; we can thus contract this rectangle down to a single point and in doing so express the resulting centred permutation as the box sum of $\pi^{\circ}_{1ldld}$ and $\pi^{\circ}_{1ruld}$.}
\label{fig:boxdecomposition}
\end{figure}

We formalise this as follows:

\begin{lemma}[Removing an interior point from a pin permutation] \label{lemma:pindec}
Suppose that $w$ is a pin word of length $n$, and that $p_{1}, p_{2}, \dots, p_{n}$ are the corresponding points of the centred pin permutation $\pi_{w}^{\circ}$. Let $k \in \{2,3,\dots,n-1\}$. Then:
\[
\pi_{w}^{\circ} - \{p_{k}\} = \pi_{w_{1,k-1}}^{\circ} \boxplus \pi_{w_{k+1,n}}^{\circ}
\]
\end{lemma}

\begin{proof}
We note that by Observation \ref{slicing}, $p_{k}$ is the \emph{only} point after $p_{k-1}$ that slices the bounding rectangle of $p_{0}, p_{1}, \dots, p_{k-1}$. Hence when $p_{k}$ is removed, $rec(p_{0}, p_{1}, \dots, p_{k-1})$ becomes a $\circ$-interval $\mathcal{I}$ in the resulting permuation, enclosing the permutation $\pi^{\circ}_{w_{1,k-1}}$. Hence, by Observation \ref{intsum}:
\[
\pi_{w}^{\circ} - \{p_{k}\} = \pi_{w_{1,k-1}}^{\circ} \boxplus \tau^{\circ}
\]
for some centred permuation $\tau^{\circ}$. That $\tau^{\circ}$ is $\pi_{w_{k+1,n}}^{\circ}$ follows from Observation \ref{pinfactorcontainment}, as this corresponds to the remaining points.
\end{proof}

This process equips pin classes with an in-built structure theorem:

\begin{thm}[The Pin Decomposition] \label{pindec}
Suppose that $w$ is an infinite pin word and that $\C_{w}^{\circ}$ is the pin class it generates. Then
\[
\sigma^{\circ} \in \C_{w}^{\circ} \ \text{iff} \ \sigma^{\circ} = \pi_{w_{1}}^{\circ} \ \boxplus \ \pi_{w_{2}}^{\circ} \ \boxplus \ \dots \ \boxplus \ \pi_{w_{k}}^{\circ} \ ,
\]
where $w_{1}, w_{2}, \dots, w_{k}$ is a sequence of pin factors of $w$ that occur in that order, in non-overlapping instances and separated from each other by at least one letter in $w$.
\end{thm}

\begin{proof}
This is almost immediate from the process derived above: $\sigma^{\circ} \in \C_{w}^{\circ}$ must be contained in some centred permutation generated by an initial subsequence $\widetilde{w}$ of $w$, so we can obtain $\sigma^{\circ}$ by deleting letters from $\widetilde{w}$; but every time we do we split the resulting permutation into the $\boxplus$-sum of the pin factors directly before and after. Repeatedly doing this gives the required result.
\end{proof}

We note one immediate consequence of this theorem:

\begin{cor}[$\boxplus$-indecomposables in a pin class] \label{pinfactorsareindecs}
Let $w$ be an infinite pin word and $\C_{w}^{\circ}$ the associated pin class. Suppose $\pi^{\circ} \in \C_{w}^{\circ}$ is a $\boxplus$-indecomposable pin permutation. Then $\pi^{\circ} = \pi_{\widetilde{w}}^{\circ}$ for some pin factor $\widetilde{w}$ of $w$.
\end{cor}

Theorem \ref{pindec} is often awkward to apply due to the conditions on the pin factors $w_{i}$; it becomes much easier however, if we assume that $w$ is a \emph{recurrent} infinite pin word --- that is, every pin factor of $w$ occurs infinitely often. The theorem then becomes:

\begin{thm}[The Pin Decomposition: Recurrent Case]
Suppose that $w$ is a recurrent infinite pin word and that $\C_{w}^{\circ}$ is the pin class it generates. Then:
\[
\sigma^{\circ} \in \C_{w}^{\circ} \ \text{iff} \ \sigma^{\circ} \ = \ \pi_{w_{1}}^{\circ} \ \boxplus \ \pi_{w_{2}}^{\circ} \ \boxplus \ \dots \ \boxplus \ \pi_{w_{k}}^{\circ} \ ,
\]
where $w_{1}, w_{2}, \dots w_{k}$ is a sequence of pin factors of $w$, and $\pi_{w_{i}}^{\circ}$ is the centred permutation generated by $w_{i}$.
\end{thm}

\begin{proof}
Immediate from Theorem \ref{pindec} on noting that, as $w$ is recurrent, each $w_{i}$ occurs infinitely-often as a pin factor of $w$, and so some sequence of instantiations of the $w_{i}$ satisfying the position and overlap conditions required by Theorem \ref{pindec} can be found for any ordering of the set $\left\{w_{1}, w_{2}, \dots w_{k}\right\}$. \qedhere
\end{proof}

This has the following crucial corollary:

\begin{cor}[Recurrent pin classes are $\boxplus$-closed]\label{recbox}
Suppose that $w$ is a recurrent infinite pin word. Then the pin class $\mathcal{C}^{\circ}_{w}$ is $\boxplus$-closed.
\end{cor}

\begin{comment}
We note now that the converse of Corollary \ref{recbox} is true, but the proof involves results derived in the next section, specifically regarding the possibility of certain pin factors generating $\boxplus$-decomposable permutations (we know that every $\boxplus$-indecomposable in a pin class is of the form $\pi^{\circ}_{w_{i,j}}$, but, as we shall see, the converse is certainly not true). As we shall not require this converse we omit the proof.
\end{comment}

Corollary \ref{recbox} tells us that to enumerate a \emph{recurrent} pin class $\C_{w}^{\circ}$ it will suffice to enumerate its $\boxplus$-indecomposables and then apply the generating function specification given in Theorem \ref{genfuncspec}. It is thus important for us to know how to enumerate the $\boxplus$-indecomposables in a pin class. Corollary \ref{pinfactorsareindecs} suggests a way of doing this that we investigate in Section \ref{sec:3.5}.

\subsection{Recurrent Pin Classes Have Growth Rates} \label{sec:3.4}

We know that pin classes all have \emph{upper} growth rates by Proposition \ref{eqgrs}. In this section we show that \emph{recurrent} pin classes in fact have \emph{proper} growth rates. In fact, as our eventual aim is to show that \emph{all} pin classes have proper growth rates we will have to be slightly more general here and aim to prove that all $\boxplus$-closed subclasses of $\mathcal{P}^{\circ}$ (the class of all centred pin permutations) have proper growth rates.

Our strategy in proving this will be to follow Arratia~\cite{arratia:on-the-stanley-:} who proved that $\oplus$-closed permutation classes have proper growth rates. Arratia's strategy was to note that, in a $\oplus$-closed class $\mathcal{C}$, the map
\[
\left(\sigma, \tau\right) \mapsto \sigma \oplus \tau
\]
is an injection from $\mathcal{C}_{m} \times \mathcal{C}_{n}$ to $\mathcal{C}_{m+n}$, and so the enumeration sequence of $\mathcal{C}$ satisfies the \emph{supermultiplicative inequality:}
\[
C_{m+n} \geq C_{m}C_{n} \ .
\]
Finally, Arratia applies the supermultiplicative form of Fekete's Lemma~\cite{fekete:uber-die-vertei:} to deduce that $gr(\mathcal{C})$ exists.

We aim to apply this method to $\boxplus$-closed centred permutation classes, but we have a problem: as we have seen, the $\boxplus$-decomposition is certainly \emph{not} unique and so the map
\begin{equation}\label{pairinjection}
\begin{split}
\mathcal{C}^{\circ}_{m} \times \mathcal{C}^{\circ}_{n} & \rightarrow \mathcal{C}^{\circ}_{m+n} \\
 & \\
\left(\sigma^{\circ}, \tau^{\circ}\right) & \mapsto \sigma^{\circ} \boxplus \tau^{\circ}
\end{split}
\end{equation}
is not necessarily an injection and we will not, in general, be able to prove the supermultiplicative inequality for a $\boxplus$-closed subclass of $\mathcal{P}^{\circ}$.

We note, however, that even the weaker `supermultiplicative-like' inequality $C_{m+n} \geq kC_{m}C_{n}$ (for some constant $k<1$) would be sufficient to deduce the existence of $gr(\mathcal{C})$, on applying Fekete's Lemma to $D_{n}=kC_{n}$ instead. Unfortunately, the enumeration sequence of a $\boxplus$-closed subclass of $\mathcal{P}^{\circ}$ does not in general satisfy even this weaker form of supermultiplicativity: for example, if we take $\mathcal{C}^{\circ}$ to be the $\boxplus$-closure of $(\nept)$ and $(\swpt)$ we can easily show that
\[
\lim_{n\rightarrow\infty}|\mathcal{C}^{\circ}_{2n}|/|\mathcal{C}^{\circ}_{n}|^{2} = 0.
\]
The problem in this example is that there is \emph{too much commutativity} due to the class only containing points in the first and third quadrants, which are opposite to each other. As it turns out, this is really the only thing that can go wrong in proving a supermultiplicative-like identity: if we have any points from a pair of adjacent quadrants then we already have enough pairs that don't commute to conclude that $C_{m+n} \geq kC_{m}C_{n}$ for some constant $k$. We thus exclude the bad case in which we have only points from opposite quadrants via the following definition:

\begin{defn}[Adjacency Condition]
Let $p^{\circ}_{i}$ denote the centred permutation consisting of a single point in the $i$th quadrant (so $p^{\circ}_{1} = (\nept)$, etc.). We say that a centred permutation class $\mathcal{C}^{\circ}$ satisfies the \emph{adjacency condition} if it contains either precisely one of $p^{\circ}_{1}, p^{\circ}_{2}, p^{\circ}_{3}, p^{\circ}_{4}$ or it contains a pair $p^{\circ}_{i}, p^{\circ}_{j}$ from adjacent quadrants.
\end{defn}

Assuming this condition we may now prove a weak form of supermultiplicativity for the enumeration sequence of a $\boxplus$-closed subclass of $\mathcal{P}^{\circ}$:

\begin{prop}\label{supermultgen}
Let $\mathcal{C}^{\circ}$ be a $\boxplus$-closed subclass of $\mathcal{P}^{\circ}$ which satisfies the adjacency condition. Let $\mathcal{C}^{\circ}_{n}$ denote the set of centred permutations of length $n$ in $\mathcal{C}^{\circ}$ and write $C_{n} = |\mathcal{C}^{\circ}_{n}|$. Then, for all $m,n \in \mathbb{N}$:
\[
C_{m+n} \geq C_{m-1}C_{n-1} \ .
\]
\end{prop}

\begin{proof}
Note first that certainly $p_{i} \in \mathcal{C}^{\circ}$ for some $i$: now, if $\sigma^{\circ} \in \mathcal{C}^{\circ}_{n-1}$ then $p_{i} \boxplus \sigma^{\circ} \in \mathcal{C}^{\circ}_{n}$, and by left-cancellation the implied map from $\mathcal{C}^{\circ}_{n-1}$ to $\mathcal{C}^{\circ}_{n}$ is injective. Hence:
\[
C_{n} \geq C_{n-1} \ .
\]
We now split into cases based on how many of the one-point permutations $p^{\circ}_{1}, p^{\circ}_{2}, p^{\circ}_{3}, p^{\circ}_{4}$ are contained in $\mathcal{C}^{\circ}$:
\begin{itemize}
\item \textbf{Case 1:} Suppose $\mathcal{C}^{\circ}$ contains either one or two of $p^{\circ}_{1}, p^{\circ}_{2}, p^{\circ}_{3}, p^{\circ}_{4}$. Then, by the adjacency condition, $\mathcal{C}^{\circ}$ is contained entirely within one half-plane. Hence by Lemma \ref{commuters} there are no pairs of distinct elements in $\mathcal{C}^{\circ}$ which commute, and so, by Theorem \ref{uniqueness}, the map
\[
\begin{split}
\mathcal{C}^{\circ}_{m} \times \mathcal{C}^{\circ}_{n} & \rightarrow \mathcal{C}^{\circ}_{m+n} \\
\left(\sigma^{\circ}, \tau^{\circ}\right) & \mapsto \sigma^{\circ} \boxplus \tau^{\circ}
\end{split}
\]
is an injection, hence
\[
\begin{split}
C_{m+n} & \geq C_{m}C_{n} \\
 & \geq C_{m-1}C_{n-1}
\end{split}
\]
as required.
\item \textbf{Case 2:} Suppose $\mathcal{C}^{\circ}$ contains precisely three of $p^{\circ}_{1}, p^{\circ}_{2}, p^{\circ}_{3}, p^{\circ}_{4}$. Without loss of generality, assume that these are $p^{\circ}_{1}, p^{\circ}_{2}, p^{\circ}_{3}$. Note that $p^{\circ}_{2}$ commutes with no other element in $\mathcal{C}^{\circ}$. Thus
\[
\begin{split}
\mathcal{C}^{\circ}_{m} \times \mathcal{C}^{\circ}_{n-1} & \rightarrow \mathcal{C}^{\circ}_{m+n} \\
\left(\sigma^{\circ}, \tau^{\circ}\right) & \mapsto \sigma^{\circ} \boxplus p^{\circ}_{2} \boxplus \tau^{\circ}
\end{split}
\]
is an injection. Hence:
\[
\begin{split}
C_{m+n} & \geq C_{m}C_{n-1} \\
 & \geq C_{m-1}C_{n-1}
\end{split}
\]
as required.
\item \textbf{Case 3:} Suppose $\mathcal{C}^{\circ}$ contains all four of $p^{\circ}_{1}, p^{\circ}_{2}, p^{\circ}_{3}, p^{\circ}_{4}$. Then
\[
\begin{split}
\mathcal{C}^{\circ}_{m-1} \times \mathcal{C}^{\circ}_{n-1} & \rightarrow \mathcal{C}^{\circ}_{m+n} \\
\left(\sigma^{\circ}, \tau^{\circ}\right) & \mapsto \sigma^{\circ} \boxplus p^{\circ}_{2} \boxplus p^{\circ}_{3} \boxplus \tau^{\circ}
\end{split}
\]
is an injection, as no centred permutation in $\mathcal{C}^{\circ}$ can commute with both of the interior permutations. Hence:
\[
C_{m+n} \geq C_{m-1}C_{n-1}
\]
as required.
\end{itemize}
\end{proof}

This, as it stands, is slightly weaker than a supermultiplicativity result (which should relate $C_{m+n}$ with $C_{m}C_{n}$, not $C_{m-1}C_{n-1}$), but we could easily convert it into such if we could obtain a bound on the ratio between consecutive terms in the enumeration sequence of a $\boxplus$-closed subclass of $\mathcal{P}^{\circ}$. To find such a bound we notice that all pin permutations of length $n+1$ can be obtained as `extensions' of pin permutations of length $n$, motivating the following definition:

\begin{defn}[Pin Representations]
Let $\sigma^{\circ}$ be a pin permutation. We call a $k$-tuple of pin words $\left(w_{1}, w_{2}, \dots, w_{k}\right)$ a \emph{pin representation of $\sigma^{\circ}$} if
\[
\sigma^{\circ} = \pi^{\circ}_{w_{1}} \ \boxplus \ \pi^{\circ}_{w_{2}} \ \boxplus \ \dots \ \boxplus \ \pi^{\circ}_{w_{k}} \ ,
\]
where each $w_{i}$ is a pin word.
\end{defn}

Clearly, every pin permutation $\sigma^{\circ}$ has a pin representation (as every $\boxplus$-indecomposable is of the form $\pi^{\circ}_{\widetilde{w}}$; but note that we do not require that each $\pi^{\circ}_{w_{i}}$ is $\boxplus$-indecomposable). Note, however, that pin representations are in general highly non-unique (it may easily be checked, for example, that $\left(1, 3, 1ul\right)$, $\left(3, 1, 1ul\right)$ and $\left(1, 1uld\right)$ are all representations of the same centred permutation, namely $51\underline{2}364$). We can now formalise the notion of one pin permutation being an extension of another:

\begin{defn}[One-point Extensions]
Let $\sigma^{\circ}, \widetilde{\sigma}^{\circ} \in \mathcal{P}^{\circ}$ with $|\widetilde{\sigma}^{\circ}| = |\sigma^{\circ}| + 1$. We say that $\widetilde{\sigma}^{\circ}$ is a \emph{one-point extension of $\sigma^{\circ}$} if there is some pin representation $\left(w_{1}, w_{2}, \dots, w_{k}\right)$ of $\sigma^{\circ}$ for which there exists \emph{either:}
\begin{itemize}
\item some $L \in \{u,d,l,r\}$ such that the concatenation $w_{k}L$ is a valid pin word and
\[
\left(w_{1}, w_{2}, \dots, w_{k}L\right)
\]
is a pin representation of $\widetilde{\sigma}^{\circ}$; \ \ \ \emph{or:}
\item some $Q \in \{1,2,3,4\}$ such that
\[
\left(w_{1}, w_{2}, \dots, w_{k}, Q\right)
\]
is a pin representation of $\widetilde{\sigma}^{\circ}$.
\end{itemize}
\end{defn}

Informally, $\widetilde{\sigma}^{\circ}$ is a one-point extension of $\sigma^{\circ}$ if there is a pin representation of $\sigma^{\circ}$ which can be extended to a pin representation of $\widetilde{\sigma}^{\circ}$ by appending one final symbol, either as an extension of the final pin word in the representation, or as an additional pin word which consists of a single quadrant numeral. We note the following crucial facts about one-point extensions:

\begin{lemma}\label{onepointfacts}
Suppose $\mathcal{C}^{\circ}$ is a subclass of the complete pin class $\mathcal{P}^{\circ}$. Then:
\begin{enumerate}
\item Every $\widetilde{\sigma}^{\circ} \in \mathcal{C}^{\circ}$ of length $n+1$ is a one-point extension of some $\sigma^{\circ} \in \mathcal{C}^{\circ}$ of length $n$.
\item Let $\sigma^{\circ} \in \mathcal{C}^{\circ}$. Then $\sigma^{\circ}$ has at most $12$ one-point extensions in $\mathcal{C}^{\circ}$.
\end{enumerate}
\end{lemma}

\begin{comment}
Note that \textit{2.} is a statement about \emph{permutations} not representations: any given $\sigma^{\circ} \in \mathcal{C}^{\circ}$ of length $n$ may have many distinct representations, which may extend to many more than $12$ distinct representations of length $n+1$; the lemma tells us that these representations represent at most $12$ distinct permutations between them.
\end{comment}

\begin{proof}
\begin{enumerate}
\item Let $\widetilde{\sigma}^{\circ} \in \mathcal{C}^{\circ}$ have length $n+1$ and take a pin representation $\left(w_{1}, w_{2}, \dots, w_{k}\right)$ of $\widetilde{\sigma}^{\circ}$. Simply by deleting the final symbol of $w_{k}$ (which may be all of $w_{k}$ if $w_{k}$ is a quadrant numeral) we obtain a pin representation of a permutation $\sigma^{\circ}$ which must also be in $\mathcal{C}^{\circ}$ (by Observation \ref{pinfactorcontainment} and closure of $\mathcal{C}^{\circ}$ under the containment order $\leq$) and which has $\widetilde{\sigma}^{\circ}$ as a one-point extension.
\item Let $\sigma^{\circ} \in \mathcal{C}^{\circ}$ and suppose that $\widetilde{\sigma}^{\circ}$ is a one-point extension of $\sigma^{\circ}$ (which may or may not be in $\mathcal{C}^{\circ}$). We claim that \emph{regardless of the pin representation chosen for} $\sigma^{\circ}$, $\widetilde{\sigma}^{\circ}$ must be obtained by adding a point in one of the $12$ positions indicated in Figure \ref{fig:extensionpossibilities}: if we append a numeral $Q$  to the pin representation then $\widetilde{\sigma}^{\circ} = \sigma^{\circ} \boxplus p^{\circ}_{Q}$ and we have added a point in one of the extreme corners; if, on the other hand, we append a letter $L$ to the pin representation then we add a point which is the most extreme in the direction indicated by $L$ and the second-most extreme in another (perpendicular) direction, determined by the final symbol of the pin representation of $\sigma^{\circ}$. \qedhere
\end{enumerate}
\end{proof}

\begin{figure}[h]
\begin{center}
\begin{tikzpicture}[scale=0.35]

\node[circle, draw, fill=none, inner sep=0pt, minimum width=\plotptradius] (0) at (0,0) {};
\node[permpt] (1) at (5,3) {}; \draw[thin] (1) -- ++ (-2.5,0);
\node[permpt] (2) at (3,5) {}; \draw[thin] (2) -- ++ (0,-2.5);
\node[permpt] (3) at (-3,5) {}; \draw[thin] (3) -- ++ (0,-2.5);
\node[permpt] (4) at (-5,3) {}; \draw[thin] (4) -- ++ (2.5,0);
\node[permpt] (5) at (-5,-3) {}; \draw[thin] (5) -- ++ (2.5,0);
\node[permpt] (6) at (-3,-5) {}; \draw[thin] (6) -- ++ (0,2.5);
\node[permpt] (7) at (3,-5) {}; \draw[thin] (7) -- ++ (0,2.5);
\node[permpt] (8) at (5,-3) {}; \draw[thin] (8) -- ++ (-2.5,0);

\node[permpt] (9) at (5,5) {};
\node[permpt] (10) at (-5,5) {};
\node[permpt] (11) at (-5,-5) {};
\node[permpt] (11) at (5,-5) {};

\draw[thick] (-7,0) -- ++ (14,0);
\draw[thick] (0,-7) -- ++ (0,14);

\draw[thick] (-4,-4) rectangle (4,4);

\end{tikzpicture}
\end{center}
\caption{If $\sigma^{\circ}$ is enclosed by the square then \emph{any} one-point extension of $\sigma^{\circ}$ in $\mathcal{C}^{\circ}$ is formed by adding a point in one of the $12$ positions shown. As in all pin diagrams, a pin separates precisely one point from all the rest, so, for example, a point at the end of an upward pin in the right half-plane has precisely one point to its right, with all other points to its left.}
\label{fig:extensionpossibilities}
\end{figure}

Lemma \ref{onepointfacts} immediately implies the following:

\begin{prop} \label{upratiobound}
Let $\mathcal{C}^{\circ}$ be a subclass of the complete pin class $\mathcal{P}^{\circ}$. Let $\mathcal{C}^{\circ}_{n}$ denote the set of centred permutations of length $n$ in $\mathcal{C}^{\circ}$ and write $C_{n} = |\mathcal{C}^{\circ}_{n}|$. Then, for all $n \in \mathbb{N}$:
\[
C_{n} \leq 12C_{n-1} \ .
\]
\end{prop}

Finally, we combine Propositions \ref{supermultgen} and \ref{upratiobound} to deduce a supermultiplicativity-like identity for $\boxplus$-closed subclasses of $\mathcal{P}^{\circ}$ satisfying the adjacency condition:

\begin{cor}\label{supermultbound}
Let $\mathcal{C}^{\circ}$ be a $\boxplus$-closed subclass of $\mathcal{P}^{\circ}$ which satisfies the adjacency condition and let $C_{n} = |\mathcal{C}^{\circ}_{n}|$. Then, for all $m,n \in \mathbb{N}$:
\[
C_{m+n} \geq \frac{1}{144}C_{m}C_{n} \ .
\]
\end{cor}

This is finally enough to prove our main result:

\begin{thm} \label{grexist}
Suppose that $\mathcal{C}^{\circ}$ is a $\boxplus$-closed subclass of $\mathcal{P}^{\circ}$ which satisfies the adjacency condition. Then $\mathcal{C}^{\circ}$ has a proper growth rate.
\end{thm}

\begin{proof}
Let $C_{k}$ be the enumeration sequence of $\mathcal{C}^{\circ}$ and set $D_{k} = \frac{1}{144}C_{k}$ for all $k \in \mathbb{N}$. Then, by Corollary \ref{supermultbound},
\[
D_{m+n} \geq D_{m}D_{n} \ ,
\]
so we may apply Fekete's Lemma to deduce that $\lim_{n\rightarrow\infty}\sqrt[n]{D_{n}}$ exists and is equal to $\sup_{n\in\mathbb{N}}\sqrt[n]{D_{n}}$, which is finite by Marcus-Tardos. Hence:
\[
\begin{split}
\lim_{n\rightarrow\infty}\sqrt[n]{C_{n}} & = \lim_{n\rightarrow\infty}(\sqrt[n]{144})(\sqrt[n]{D_{n}}) \\
 & = \lim_{n\rightarrow\infty}\sqrt[n]{D_{n}} \\
 & < \infty,
\end{split}
\]
which is to say that $gr(\mathcal{C}^{\circ})$ exists. \end{proof}

The utility of Theorem \ref{grexist} for our present purposes stems from the following fact: 

\begin{prop}[Pin classes satisfy the adjacency condition] \label{pinclassessatisfytheadjacencycondition}
Suppose that $w$ is an infinite pin word. Then $\mathcal{C}^{\circ}_{w}$ satisfies the adjacency condition.
\end{prop}

\begin{proof}
It is impossible for a pin sequence to move between diagonally opposite quadrants without passing through one of the other two quadrants first. Thus if both $p^{\circ}_{1}$ and $p^{\circ}_{3}$ are contained in $\mathcal{C}^{\circ}_{w}$ then at least one of $p^{\circ}_{2}$ and $p^{\circ}_{4}$ is also in $\mathcal{C}^{\circ}_{w}$, and, similarly, if both $p^{\circ}_{2}$ and $p^{\circ}_{4}$ are contained in $\mathcal{C}^{\circ}_{w}$ then at least one of $p^{\circ}_{1}$ and $p^{\circ}_{3}$ is in $\mathcal{C}^{\circ}_{w}$. Thus $\mathcal{C}^{\circ}_{w}$ satisfies the adjacency condition. \qedhere
\end{proof}

We may now immediately deduce:

\begin{cor}[Recurrent pin classes have proper growth rates] \label{recgrs}
Let $w$ be a recurrent infinite pin word. Then the pin class $\mathcal{C}^{\circ}_{w}$ has a proper growth rate.
\end{cor}

\begin{proof}
Suppose $w$ is a recurrent infinite pin word. Then by Proposition \ref{pinclassessatisfytheadjacencycondition}, $\mathcal{C}^{\circ}_{w}$ satisfies the adjacency condition, and, by Corollary \ref{recbox}, $\mathcal{C}^{\circ}_{w}$ is $\boxplus$-closed. Thus Theorem \ref{grexist} implies that $\mathcal{C}^{\circ}_{w}$ has a proper growth rate. \qedhere
\end{proof}

Of course, Corollary \ref{recgrs} (along with Proposition \ref{eqgrs}) immediately implies that \emph{uncentred} pin classes generated by recurrent infinite pin words also have growth rates, which will equal that of their centred counterparts: $gr(\mathcal{C}_{w}) = gr(\mathcal{C}^{\circ}_{w})$.

We aim eventually to remove the assumption of recurrence from Corollary \ref{recgrs} --- this will be the content of Section \ref{sec:4}. For now, we move on to establish the theory that will allow us to calculate explicitly the growth rate of a given recurrent pin class.

\subsection{Classification of Collisions and $\boxplus$-decomposables} \label{sec:3.5}

We now know that any recurrent pin class $\mathcal{C}^{\circ}_{w}$ has a proper growth rate. In order to find the growth rate of a given recurrent pin class in practice we shall have to enumerate the $\boxplus$-indecomposables contained in $\mathcal{C}^{\circ}_{w}$: if we can do this then Theorem \ref{genfuncspec} will allow us to write down the generating function of the class, from which we can deduce the growth rate via the Exponential Growth Theorem. Accordingly, in this section we take up the problem of enumerating the $\boxplus$-indecomposables in a pin class.

We begin by noting that Corollary \ref{pinfactorsareindecs} immediately implies the following property of the restriction of the $\pi$-map to pin factors of length $n$:

\begin{prop}[Induced $\pi$-map contains all $\boxplus$-indecomposables in its image]
Let $w$ be an infinite pin word and $\mathcal{C}^{\circ}_{w}$ the associated centred pin class. Let $P_{n}$ denote the set of pin factors of $w$ of length $n$ and $\mathcal{B}^{\circ}_{n}$ denote the set of $\boxplus$-indecomposables in $\mathcal{C}^{\circ}_{w}$ of length $n$. Then the induced map
\[
\pi_{n}: P_{n} \rightarrow \mathcal{C}^{\circ}_{w}
\]
defined by 
\[
\pi_{n}: \tilde{w} \mapsto \pi^{\circ}_{\tilde{w}}
\]
contains the set $\mathcal{B}^{\circ}_{n}$ in its image.
\end{prop}

The study of this induced map will be vital for further progress as it suggests a natural method for counting the $\boxplus$-indecomposables in a given pin class $\mathcal{C}^{\circ}_{w}$: if (for some $n$) we had the following two properties:
\begin{itemize}
\item the image $\pi_{n}(P_{n})$ is equal to $\mathcal{B}^{\circ}_{n}$ (that is, $\pi_{n}(\tilde{w})$ is $\boxplus$-indecomposable for all pin factors $\tilde{w}$ of $w$);
\item $\pi_{n}$ is injective;
\end{itemize}
then $\pi_{n}$ would be a bijection between $P_{n}$ and $\mathcal{B}^{\circ}_{n}$, thus reducing the problem of enumerating the $\boxplus$-indecomposables of $\mathcal{C}^{\circ}_{w}$ to the (much simpler) combinatorial problem of counting the pin factors of $w$.

We do in fact have some intuitive reasons for suspecting that these two properties should hold, at least for sufficently large $n$: the definition of the $\pi$-map (specifically, the placing of each new point in order to intersect the bounding rectangle of all previous points) seems almost designed to ensure that $\pi_{\widetilde{w}}$ can have no proper non-trivial $\circ$-interval, and pin permutations have so much internal structure that it seems unlikely that two distinct pin words could conspire to generate the exact same centred permutation. Alas, both turn out not to be true for any $n$: for every length $n \geq 2$ there are pin words which generate $\boxplus$-decomposable permutations, as well as pairs of pin words which generate the same centred permutation. Fortunately, as our intuition suggests, both of these 'bad behaviours' are somewhat pathological and easily classified, so with some minor adjustments we will be able to make our strategy for counting $\boxplus$-indecomposables in a pin class work.

We begin by introducing some terminology for these pathological behaviours. First, those pin words whose image under the induced $\pi$-map is not contained in $\mathcal{B}^{\circ}_{n}$:

\begin{defn}[$\boxplus$-decomposable pin words]
We call a finite pin word $w$ a \emph{$\boxplus$-decomposable pin word} if the centred permution it generates, $\pi^{\circ}_{w}$, is $\boxplus$-decomposable.
\end{defn}

And those pin words which make injectivity of the induced $\pi$-map fail:

\begin{defn}[Collisions]
We call a set $\{w_{1}, w_{2}, \dots , w_{k}\}$ of pin words which generate the same centred permutation $\pi^{\circ}$ a \emph{collision} (sometimes a $k$-\emph{collision} when we wish to emphasise the size of the set) of pin words.
\end{defn}

It is fairly easy to find all $\boxplus$-decomposables and collisions of length $n\leq5$ through an exhaustive search, especially on applying symmetries of the square. This search reveals families of $\boxplus$-decomposables and collisions which generalise to all longer lengths. We can then, with considerable effort, prove that \emph{all} examples of lengths greater than or equal to $6$ belong to these families, giving us a full classification of these confounding behaviours. This is the content of the following two theorems: the statements will be crucial and will be quoted in full here; the proofs are technical and long and will be relegated to the appendix of this paper.

\newpage

\begin{thm}[Classification of $\boxplus$-decomposables and Collisions] \label{classification}
The following two tables form a complete list of $\boxplus$-decomposable pin words and collisions:
\end{thm}

\begin{tabular}{@{}|>{\centering}m{2em}|>{\centering}m{4em}|>{\centering\arraybackslash}m{10em}|>{\centering}m{5em}|>{\centering\arraybackslash}m{6.5em}|>{\centering\arraybackslash}m{6.5em}|} \hline
\multicolumn{6}{|c|}{}\\[0.5pt]
\multicolumn{6}{|c|}{\textbf{List of $\boxplus$-decomposable pin words:}}\\[6pt] \hline
\multicolumn{2}{|c|}{\textbf{Length}}&\textbf{Representative}&\textbf{Total number}&\multicolumn{2}{|c|}{\textbf{Full List}}\\ \hline
\multicolumn{2}{|c|}{$2$}
&
\begin{tikzpicture}[scale=0.25]

\node[circle, draw, fill=none, inner sep=0pt, minimum width=\plotptradius] (0) at (2,1) {};
\node[permpt] (1) at (3,3) {};
\node[permpt] (2) at (1,2) {}; \draw[thin] (2) -- ++ (2.5,0);

\node[permpt,white] (3) at (2,5) {};

\draw[thick] (0,1) -- ++ (4,0);
\draw[thick] (2,0) -- ++ (0,4);

%\node[] (3) at (5,0) {=};

\node[] (4) at (2,-1) {$1l$};

\draw[dashed] (0.5,0.5) rectangle (2.5,2.5);
\end{tikzpicture}
&
8
&
\multicolumn{2}{|c|}{$1l, 1d, 2r, 2d, 3u, 3r, 4l, 4u$}
 \\ \hline
%%%%%
\multicolumn{2}{|c|}{$3$}
&
\begin{tikzpicture}[scale=0.25]

\node[circle, draw, fill=none, inner sep=0pt, minimum width=\plotptradius] (0) at (3,2) {};
\node[permpt] (1) at (4,4) {};
\node[permpt] (2) at (1,3) {}; \draw[thin] (2) -- ++ (3.5,0);
\node[permpt] (3) at (2,1) {}; \draw[thin] (3) -- ++ (0,2.5);

\node[permpt,white] (4) at (3,6) {};

\draw[thick] (0,2) -- ++ (5,0);
\draw[thick] (3,0) -- ++ (0,5);

%\node[] (3) at (5,0) {=};

\node[] (4) at (3,-1) {$1ld$};

\draw[dashed] (1.5,0.5) rectangle (3.5,2.5);
\end{tikzpicture}
&
8
&
\multicolumn{2}{|c|}{$1ld, 1dl, 2dr, 2rd, 3ru, 3ur, 4ul, 4lu$}
 \\ \hline
%%%%%
\multirow{2}{*}{$n\geq4$}
%\left\{\begin{array}{@{}c@{}}\textbf{Type 1:}\\[40pt]\textbf{Type 2:}\end{array}\right.
&
\textbf{Type 1:}
&
\begin{tikzpicture}[scale=0.25]

\node[circle, draw, fill=none, inner sep=0pt, minimum width=\plotptradius] (0) at (3,2) {};
\node[permpt] (1) at (6,6) {}; \draw[thin] (1) -- ++ (0,-1.5);
\node[permpt] (2) at (8,5) {}; \draw[thin] (2) -- ++ (-2.5,0);
\node[permpt] (3) at (7,8) {}; \draw[thin] (3) -- ++ (0,-3.5);
\node[permpt] (4) at (1,7) {}; \draw[thin] (4) -- ++ (6.5,0);
\node[permpt] (5) at (2,1) {}; \draw[thin] (5) -- ++ (0,6.5);

\begin{scope}[shift={(0.5,0.5)}]
\node[circle,fill,inner sep=0.5pt] (6) at (3.5,2.5) {};
\node[circle,fill,inner sep=0.5pt] (7) at (4,3) {};
\node[circle,fill,inner sep=0.5pt] (8) at (4.5,3.5) {};
\node[circle,fill,inner sep=0.5pt] (8) at (5,4) {};
\end{scope}

\node[empty] (-1) at (3,10) {};

\draw[thick] (0,2) -- ++ (9,0);
\draw[thick] (3,0) -- ++ (0,9);

%\node[] (3) at (5,0) {=};

\node[] (4) at (4,-1.5) {$1(ur)^{k}uld$ \text{or} $1(ru)^{k}ld$};

\draw[dashed] (1.5,0.5) rectangle (3.5,2.5);
\end{tikzpicture}
&
8
&
\textbf{$n \geq 4$ even:}
\[
\begin{split}
1(ur)^{k}uld, \\
1(ru)^{k}rdl, \\
2(ul)^{k}urd, \\
2(lu)^{k}ldr, \\
3(dl)^{k}dru, \\
3(ld)^{k}lur, \\
4(dr)^{k}dlu, \\
4(rd)^{k}rul
\end{split}
\]
&
\textbf{$n \geq 5$ odd:}
\[
\begin{split}
1(ru)^{k}ld, \\
1(ur)^{k}dl, \\
2(lu)^{k}rd, \\
2(ul)^{k}dr, \\
3(ld)^{k}ru, \\
3(dl)^{k}ur, \\
4(rd)^{k}lu, \\
4(dr)^{k}ul
\end{split}
\]
\\ \cline{2-6}
%%%%%
&
\textbf{Type 2:}
&
\begin{tikzpicture}[scale=0.25]

\node[circle, draw, fill=none, inner sep=0pt, minimum width=\plotptradius] (0) at (5,4) {};
\node[permpt] (1) at (6,6) {};
\node[permpt] (2) at (3,5) {}; \draw[thin] (2) -- ++ (3.5,0);
\node[permpt] (3) at (4,2) {}; \draw[thin] (3) -- ++ (0,3.5);
\node[permpt] (4) at (1,3) {}; \draw[thin] (4) -- ++ (3.5,0);
\node[permpt] (5) at (2,1) {}; \draw[thin] (5) -- ++ (0,2.5);

\begin{scope}[shift={(-4,-2.25)}]
\node[circle,fill,inner sep=0.5pt] (6) at (3.5,2.5) {};
\node[circle,fill,inner sep=0.5pt] (7) at (4,3) {};
\node[circle,fill,inner sep=0.5pt] (8) at (4.5,3.5) {};
\node[circle,fill,inner sep=0.5pt] (8) at (5,4) {};
\end{scope}

\node[empty] (-1) at (5,8) {};

\draw[thick] (-2,4) -- ++ (9,0);
\draw[thick] (5,-2) -- ++ (0,9);

%\node[] (3) at (5,0) {=};

\node[] (4) at (3,-3.5) {$1l(dl)^{k}$ \text{or} $1l(dl)^{k}d$};

\draw[dashed] (-1,-0.25) rectangle (5.5,5.5);
\end{tikzpicture}
&
8
&
\textbf{$n \geq 4$ even:}
\[
\begin{split}
1l(dl)^{k}, \\
1d(ld)^{k}, \\
2r(dr)^{k}, \\
2d(rd)^{k}, \\
3r(ur)^{k}, \\
3u(ru)^{k}, \\
4l(ul)^{k}, \\
4u(lu)^{k}
\end{split}
\]
&
\textbf{$n \geq 5$ odd:}
\[
\begin{split}
1(ld)^{k}, \\
1(dl)^{k}, \\
2(rd)^{k}, \\
2(dr)^{k}, \\
3(ru)^{k}, \\
3(ur)^{k}, \\
4(lu)^{k}, \\
4(ul)^{k} \\
\end{split}
\]
\\ \hline
\end{tabular}

\newpage

%
%
%
%
%
%
%%%%%%%%%%%%%%%%%%%%%%%%%%%%%%%%%%
%
%
%
%
%

{
\centering
\begin{tabular}{@{}|>{\centering}m{1.5em}|>{\centering}m{5.5em}|>{\centering\arraybackslash}m{12em}|>{\centering}m{4em}|>{\centering\arraybackslash}m{17.5em}|} \hline
\multicolumn{5}{|c|}{}\\[0.4pt]
\multicolumn{5}{|c|}{\textbf{List of collisions of pin words:}}\\[6pt] \hline
\multicolumn{2}{|c|}{\textbf{Length}}&\textbf{Representative}&\textbf{Total number}&\textbf{Full List}\\ \hline
\multicolumn{2}{|c|}{$2$}
&
\begin{tikzpicture}[scale=0.2]

\begin{scope}[shift={(10,0)}]

\node[circle, draw, fill=none, inner sep=0pt, minimum width=\plotptradius] (0) at (0,0) {};
\node[permpt] (1) at (1,2) {};
\node[permpt] (2) at (2,1) {}; \draw[thin] (2) -- ++ (-1.5,0);

\node[empty] (-1) at (0,4) {};

\draw[thick] (-3,0) -- ++ (6,0);
\draw[thick] (0,-3) -- ++ (0,6);

\node[] (4) at (0,-5) {$1r$};

%\draw[dotted] (0.5,8.5) rectangle (1.5,9.5);

\end{scope}

\node[circle, draw, fill=none, inner sep=0pt, minimum width=\plotptradius] (0) at (0,0) {};
\node[permpt] (1) at (2,1) {};
\node[permpt] (2) at (1,2) {}; \draw[thin] (2) -- ++ (0,-1.5);

\draw[thick] (-3,0) -- ++ (6,0);
\draw[thick] (0,-3) -- ++ (0,6);

\node[] (3) at (5,0) {=};

%\node[] (3) at (5,0) {=};

\node[] (4) at (0,-5) {$1u$};

%\draw[dotted] (0.5,8.5) rectangle (1.5,9.5);

\end{tikzpicture}
&
4 \ \text{pairs}
&
\[
\{1u, 1r\}, \{2l, 2u\}, \{3d, 3l\}, \{4r, 4d\}
\]

 \\ \hline
%%%%%
\multicolumn{2}{|c|}{$3$}
&
\begin{tikzpicture}[scale=0.2]

\node[circle, draw, fill=none, inner sep=0pt, minimum width=\plotptradius] (0) at (1,0) {};
\node[permpt] (1) at (3,1) {};
\node[permpt] (2) at (2,3) {}; \draw[thin] (2) -- ++ (0,-2.5);
\node[permpt] (3) at (0,2) {}; \draw[thin] (3) -- ++ (2.5,0);

\draw[thick] (-2,0) -- ++ (6,0);
\draw[thick] (1,-3) -- ++ (0,7);

\node[] (3) at (6,0) {=};

\node[] (4) at (1,-4) {$1ul$};

%\draw[dotted] (0.5,8.5) rectangle (1.5,9.5);

\begin{scope}[shift={(10,0)}]

\node[circle, draw, fill=none, inner sep=0pt, minimum width=\plotptradius] (0) at (1,0) {};
\node[permpt] (1) at (0,2) {};
\node[permpt] (2) at (3,1) {}; \draw[thin] (2) -- ++ (-3.5,0);
\node[permpt] (3) at (2,3) {}; \draw[thin] (3) -- ++ (0,-2.5);

\node[empty] (-1) at (1,5) {};

\draw[thick] (-2,0) -- ++ (6,0);
\draw[thick] (1,-3) -- ++ (0,7);

\node[] (4) at (1,-4) {$2ru$};

\end{scope}

\end{tikzpicture}
&
8 \ \text{pairs}
&
\small
\[
\begin{split}
\{1ul, 2ru\}, \{1rd, 4ur\}, \{2ur, 1lu\}, \{2ld, 3ul\},\\ \{3lu, 2dl\}, \{3dr, 4ld\}, \{4ru, 1dr\}, \{4dl, 3rd\}
\end{split}
\]

 \\ \hline
%%%%%
\multicolumn{2}{|c|}{$4$}
&
\begin{tikzpicture}[scale=0.2]

\node[circle, draw, fill=none, inner sep=0pt, minimum width=\plotptradius] (0) at (3,3) {};
\node[permpt] (1) at (4,5) {}; \draw[thin] (1) -- ++ (0,-3.5);
\node[permpt] (2) at (1,4) {}; \draw[thin] (2) -- ++ (3.5,0);
\node[permpt] (3) at (2,1) {}; \draw[thin] (3) -- ++ (0,3.5);
\node[permpt] (4) at (5,2) {}; \draw[thin] (4) -- ++ (-3.5,0);

\node[empty] (4) at (3,7) {};

\draw[thick] (0,3) -- ++ (6,0);
\draw[thick] (3,0) -- ++ (0,6);

\node[] (4) at (3,-1) {$1ldr = 2dru = 3rul = 4uld$};

%\draw[dotted] (0.5,8.5) rectangle (1.5,9.5);

\end{tikzpicture}
&
2 \text{quads}
&
\[
\begin{split}
\{1ldr, 2dru, 3rul, 4uld\}, \\ \{1dlu, 2rdl, 3urd, 4lur\}
\end{split}
\]

 \\ \hline
%%%%%
\multirow{2}{*}{$5$}
%\left\{\begin{array}{@{}c@{}}\textbf{Type 1:}\\[40pt]\textbf{Type 2:}\end{array}\right.
&
\textbf{Irregular:}
&
\begin{tikzpicture}[scale=0.2]

\node[circle, draw, fill=none, inner sep=0pt, minimum width=\plotptradius] (0) at (4,3) {};
\node[permpt] (1) at (6,4) {};
\node[permpt] (2) at (5,6) {}; \draw[thin] (2) -- ++ (0,-2.5);
\node[permpt] (3) at (2,5) {}; \draw[thin] (3) -- ++ (3.5,0);
\node[permpt] (4) at (3,1) {}; \draw[thin] (4) -- ++ (0,4.5);
\node[permpt] (5) at (1,2) {}; \draw[thin] (5) -- ++ (2.5,0);

\node[empty] (-1) at (4,8) {};

\draw[thick] (0,3) -- ++ (8,0);
\draw[thick] (4,0) -- ++ (0,7);

\node[] (3) at (10,3) {=};

\node[] (4) at (4,-1) {$1uldl$};

%\draw[dotted] (0.5,8.5) rectangle (1.5,9.5);

\begin{scope}[shift={(12,0)}]

\node[circle, draw, fill=none, inner sep=0pt, minimum width=\plotptradius] (0) at (4,3) {};
\node[permpt] (1) at (6,4) {}; \draw[thin] (1) -- ++ (-4.5,0);
\node[permpt] (2) at (5,6) {}; \draw[thin] (2) -- ++ (0,-2.5);
\node[permpt] (3) at (2,5) {}; \draw[thin] (3) -- ++ (0,-3.5);
\node[permpt] (4) at (3,1) {}; 
\node[permpt] (5) at (1,2) {}; \draw[thin] (5) -- ++ (2.5,0);

\draw[thick] (0,3) -- ++ (8,0);
\draw[thick] (4,0) -- ++ (0,7);

\node[] (4) at (4,-1) {$3luru$};

\end{scope}

\end{tikzpicture}
&
4 \ \text{pairs}
&
\[
\begin{split}
\{1uldl, 3luru\}, \{1rdld, 3drur\}, \\ \{2urdr, 4rulu\}, \{2ldrd, 4dlul\}
\end{split}
\]
\\ \cline{2-5}
%%%%%
&
\textbf{Regular:}
&
\begin{tikzpicture}[scale=0.2]

\node[circle, draw, fill=none, inner sep=0pt, minimum width=\plotptradius] (0) at (5,3) {};
\node[permpt] (1) at (6,5) {};
\node[permpt] (2) at (3,4) {}; \draw[thin] (2) -- ++ (3.5,0);
\node[permpt] (3) at (4,1) {}; \draw[thin] (3) -- ++ (0,3.5);
\node[permpt] (4) at (1,2) {}; \draw[thin] (4) -- ++ (3.5,0);
\node[permpt] (5) at (2,6) {}; \draw[thin] (5) -- ++ (0,-4.5);

\draw[thick] (0,3) -- ++ (8,0);
\draw[thick] (5,0) -- ++ (0,7);

\node[] (3) at (10,3) {=};

\node[] (4) at (5,-1) {$1ldlu$};

%\draw[dotted] (0.5,8.5) rectangle (1.5,9.5);

\begin{scope}[shift={(12,0)}]

\node[circle, draw, fill=none, inner sep=0pt, minimum width=\plotptradius] (0) at (5,3) {};
\node[permpt] (1) at (6,5) {}; \draw[thin] (1) -- ++ (-4.5,0);
\node[permpt] (2) at (3,4) {};
\node[permpt] (3) at (4,1) {}; \draw[thin] (3) -- ++ (0,3.5);
\node[permpt] (4) at (1,2) {}; \draw[thin] (4) -- ++ (3.5,0);
\node[permpt] (5) at (2,6) {}; \draw[thin] (5) -- ++ (0,-4.5);

\draw[thick] (0,3) -- ++ (8,0);
\draw[thick] (5,0) -- ++ (0,7);

\node[] (4) at (5,-1) {$2dlur$};

\end{scope}

\end{tikzpicture}
&
8 \ \text{pairs}
&
\small
\[
\begin{split}
\{1ldlu, 2dlur\}, \{1dldr, 4ldru\}, \\ \{2rdru, 1drul\}, \{2drdl, 3rdlu\}, \\
\{3urul, 2ruld\}, \{3rurd, 4urdl\}, \\ \{4luld, 3uldr\}, \{4ulur, 1lurd\}
\end{split}
\]

\\ \hline

%%%%%
\multicolumn{2}{|c|}{$n \geq 6$ \ \textbf{even:}}
&
\begin{tikzpicture}[scale=0.2]

\node[circle, draw, fill=none, inner sep=0pt, minimum width=\plotptradius] (0) at (8,8) {};
\node[permpt] (1) at (9,10) {}; 
\node[permpt] (2) at (6,9) {}; \draw[thin] (2) -- ++ (3.5,0);
\node[permpt] (3) at (7,6) {}; \draw[thin] (3) -- ++ (0,3.5);
\node[permpt] (4) at (4,7) {}; \draw[thin] (4) -- ++ (3.5,0);
\node[permpt] (5) at (5,5) {};  \draw[thin] (5) -- ++ (0,2.5);
\node[permpt] (6) at (1,3) {}; \draw[thin] (6) -- ++ (1.5,0);
\node[permpt] (7) at (2,1) {}; \draw[thin] (7) -- ++ (0,2.5);
\node[permpt] (8) at (11,2) {}; \draw[thin] (8) -- ++ (-9.5,0);
\node[permpt] (9) at (10,11) {}; \draw[thin,dashed] (9) -- ++ (0,-9.5);

\node[empty] (-1) at (8,13) {};

\node[circle,fill,inner sep=0.5pt] (10) at (4.5,5.5) {};
\node[circle,fill,inner sep=0.5pt] (11) at (4,5) {};
\node[circle,fill,inner sep=0.5pt] (12) at (3.5,4.5) {};
\node[circle,fill,inner sep=0.5pt] (13) at (3,4) {};
\node[circle,fill,inner sep=0.5pt] (14) at (2.5,3.5) {};

\draw[thick] (8,0) -- ++ (0,12);
\draw[thick] (0,8) -- ++ (12,0);

\draw[dotted] (9.5,10.5) rectangle (10.5,11.5);

\node[] (15) at (6.5,-1) {$1(ld)^{k}r=2(dl)^{k}dru$};
\end{tikzpicture}
&
8 \ \text{pairs}
&
\small
\[
\begin{split}
\{1(ld)^{k}r, 2(dl)^{k}dru\}, \{1(dl)^{k}u, 4(ld)^{k}lur\}, \\
\{2(dr)^{k}u, 3(rd)^{k}rul\}, \{2(rd)^{k}l, 1(dr)^{k}dlu\}, \\
\{3(ru)^{k}l, 4(ur)^{k}uld\}, \{3(ur)^{k}d, 2(ru)^{k}rdl\}, \\
\{4(ul)^{k}d, 1(lu)^{k}ldr\}, \{4(lu)^{k}r, 3(ul)^{k}urd\}
\end{split}
\]

 \\ \hline

%%%%%
\multicolumn{2}{|c|}{$n \geq 7$ \ \textbf{odd:}}
&
\begin{tikzpicture}[scale=0.2]

\node[circle, draw, fill=none, inner sep=0pt, minimum width=\plotptradius] (0) at (8,6) {};
\node[permpt] (1) at (9,8) {};
\node[permpt] (2) at (6,7) {}; \draw[thin] (2) -- ++ (3.5,0);
\node[permpt] (3) at (7,4) {}; \draw[thin] (3) -- ++ (0,3.5);
\node[permpt] (4) at (5,5) {}; \draw[thin] (4) -- ++ (2.5,0);
\node[permpt] (5) at (3,1) {};  \draw[thin] (5) -- ++ (0,1.5);
\node[permpt] (6) at (1,2) {}; \draw[thin] (6) -- ++ (2.5,0);
\node[permpt] (7) at (2,10) {}; \draw[thin] (7) -- ++ (0,-8.5);
\node[permpt] (8) at (10,9) {}; \draw[thin,dashed] (8) -- ++ (-8.5,0);

\node[empty] (-1) at (8,12) {};

\node[circle,fill,inner sep=0.5pt] (8) at (5.5,4.5) {};
\node[circle,fill,inner sep=0.5pt] (9) at (5,4) {};
\node[circle,fill,inner sep=0.5pt] (10) at (4.5,3.5) {};
\node[circle,fill,inner sep=0.5pt] (11) at (4,3) {};
\node[circle,fill,inner sep=0.5pt] (12) at (3.5,2.5) {};

\draw[thick] (8,0) -- ++ (0,11);
\draw[thick] (0,6) -- ++ (11,0);

\draw[dotted] (9.5,8.5) rectangle (10.5,9.5);

\node[] (13) at (6,-1) {$1(ld)^{k}lu=2(dl)^{k}ur$};

\end{tikzpicture}
&
8 \ \text{pairs}
&
\small
\[
\begin{split}
\{1(ld)^{k}lu, 2(dl)^{k}ur\}, \{1(dl)^{k}dr, 4(ld)^{k}ru\}, \\
\{2(dr)^{k}dl, 3(rd)^{k}lu\}, \{2(rd)^{k}ru, 1(dr)^{k}ul\}, \\
\{3(ru)^{k}rd, 4(ur)^{k}dl\}, \{3(ur)^{k}ul, 2(ru)^{k}ld\}, \\
\{4(ul)^{k}ur, 1(lu)^{k}rd\}, \{4(lu)^{k}ld, 3(ul)^{k}dr\}
\end{split}
\]
 \\ \hline
\end{tabular}\par
}

Note: in the previous table we write $(f)^{k}$ to denote the factor $f$ repeating $k$ times for some $k \geq 0$.

\begin{obs}
We emphasise some noteworthy features of these lists:
\begin{enumerate}
\item There is no overlap between the two lists: no pin word which generates a $\boxplus$-decomposable is also involved in a collision.
\item All collisions are pairs except for the two quadruples at length $4$.
\item There are two distinct families of collisions at length $5$: the \emph{regular} length $5$ collisions, which generalise to a family of collisions for all $n \geq 6$; and the \emph{irregular} length $5$ collisions, which are a family unique to length $5$.
\item All even-length collisions of length $n \geq 4$ require all four quadrants, whereas the odd-length collisions of length $n \geq 5$ require only three.
\end{enumerate}
\end{obs}

\subsection{Enumeration of Recurrent Pin Classes} \label{sec:3.6}

Combining the results of the previous section with the generating function specification of a $\boxplus$-closed centred permutation class established in Section \ref{sec:2}, we now have a general strategy for obtaining the generating function of the $\boxplus$-closure of a pin class:

\begin{procedure}[$\boxplus\left(\mathcal{C}^{\circ}_{w}\right)$ Enumeration Procedure]\label{recproc}

Suppose that $w$ is a pin word. We can find the generating function of the centred pin class $\boxplus\left(\mathcal{C}^{\circ}_{w}\right)$ by the following procedure:

\begin{enumerate}
\item Find the generating function $h(z)$ of the pin factors of $w$;
\item Enumerate the pin factors of $w$ that generate $\boxplus$-decomposable pin permutations, as well as colliding sets of pin factors of $w$, by comparison with the lists in Theorem \ref{classification}. Subtract these from $h(z)$ to obtain $g(z)$, the generating function of $\boxplus$-indecomposables in $\mathcal{C}^{\circ}_{w}$;
\item Amend for commutativity: find the generating function $g_{i}(z)$ of the one-quadrant $\boxplus$-indecomposables of $\mathcal{C}^{\circ}_{w}$ in the $i$th quadrant, for each $i \in \{1,2,3,4\}$. (In other words: enumerate the one-quadrant oscillations.) Set:
\[
G(z) = g(z) - g_{1}(z)g_{3}(z) - g_{2}(z)g_{4}(z);
\]
this is the amended $G$-sequence of $\mathcal{C}^{\circ}_{w}$;
\item Then the generating function of $\boxplus\left(\mathcal{C}^{\circ}_{w}\right)$ is given by
\[
f(z) = \frac{1}{1 - G(z)} \ ,
\]
and the growth rate of $\boxplus\left(\mathcal{C}^{\circ}_{w}\right)$ is the reciprocal of the radius of convergence of $f(z)$.
\end{enumerate}

Of course, if $w$ is recurrent then $\mathcal{C}^{\circ}_{w}$ is $\boxplus$-closed and so $f(z)$ is in fact the generating function of $\mathcal{C}^{\circ}_{w}$.

\end{procedure}

\subsection{Examples in Two Quadrants} \label{sec:3.7}

We illustrate Procedure \ref{recproc} with a sequence of examples, beginning in this section with recurrent pin classes contained entirely in two quadrants, which we may without loss of generality take to be quadrants $1$ and $2$. These two-quadrant pin classes tend to be the simplest to enumerate due to the lack of commutativity of the $\boxplus$-sum in two quadrants, along with the fact that there are only finitely-many collisions and $\boxplus$-decomposables amongst the pin words restricted to two quadrants by Theorem \ref{classification}. We begin with the simplest pin class: the class of increasing oscillations $\mathcal{O}^{\circ}$, already introduced in Figure \ref{fig:oscclass} on page \pageref{fig:oscclass}:

\begin{example}[The Increasing Oscillations] \label{ex:defO}
The growth rate of $\mathcal{O}^{\circ}$ is $\kappa \approx 2.20557$, defined to be the reciprocal of the smallest positive real root of
\[
1 - 2z - z^3 = 0
\]
\end{example}

\begin{proof}
The defining infinite pin word of $\mathcal{O}^{\circ}$ is $w=1\overline{(ru)}$, which has two pin factors of every length $\geq 2$ (corresponding to the two distinct starting positions within the period) and one pin factor (namely $1$ itself) of length $1$, so
\[
h(z) = z + 2z^2 + 2z^3 + 2z^4 + 2z^5 \dots
\]
By comparison with the lists in Theorem \ref{classification}, there are no $\boxplus$-decomposables amongst the pin factors of $w$ and precisely one colliding pair, namely $\{1u,1r\}$, at length $2$. We thus subtract $1$ at length $2$ to obtain the generating function of the $\boxplus$-indecomposables in $\mathcal{O}^{\circ}$:
\[
\begin{split}
g(z) & = z + z^2 + 2z^3 + 2z^4 + 2z^5 + \dots \\
 & = \frac{z+z^3}{1-z} \ .
\end{split}
\]
Noting that $g_{2}(z) = g_{3}(z) = g_{4}(z) = 0$ we see that $G(z)=g(z)$ and the generating function of $\mathcal{O}^{\circ}$ is
\[
\begin{split}
f(z) & = \frac{1}{1 - g(z)} \\
 & = \frac{1-z}{1 - 2z - z^3} \ ,
\end{split}
\]
and we may now deduce the growth rate by the Exponential Growth Theorem \ref{EGT}. \end{proof}

We shall see later that $\kappa$ is in fact the smallest possible growth rate of a pin class. For now we move on to the next simplest pin class:

\begin{figure}[h]
\begin{center}
\reflectbox{\begin{tikzpicture}[scale=0.35]

\draw[thick] (9,2) -- ++ (0,17);
\draw[thick] (0,2) -- ++ (18,0);

\node[circle, draw, fill=none, inner sep=0pt, minimum width=\plotptradius] (0) at (9,2) {};

%\node[permpt] (1) at (11,2) {};
%\node[permpt] (2) at (10,4) {}; \draw[thin] (2) -- ++ (0,-2.5);
\node[permpt] (3) at (7,3) {};
\node[permpt] (4) at (8,6) {}; \draw[thin] (4) -- ++ (0,-3.5);
\node[permpt] (5) at (13,5) {}; \draw[thin] (5) -- ++ (-5.5,0);
\node[permpt] (6) at (12,8) {}; \draw[thin] (6) -- ++ (0,-3.5);
\node[permpt] (7) at (5,7) {};  \draw[thin] (7) -- ++ (7.5,0);
\node[permpt] (8) at (6,10) {}; \draw[thin] (8) -- ++ (0,-3.5);
\node[permpt] (9) at (15,9) {}; \draw[thin] (9) -- ++ (-9.5,0);
\node[permpt] (10) at (14,12) {}; \draw[thin] (10) -- ++ (0,-3.5);
\node[permpt] (11) at (3,11) {};  \draw[thin] (11) -- ++ (11.5,0);
\node[permpt] (12) at (4,14) {}; \draw[thin] (12) -- ++ (0,-3.5);
\node[permpt] (13) at (17,13) {}; \draw[thin] (13) -- ++ (-13.5,0);
\node[permpt] (14) at (16,16) {}; \draw[thin] (14) -- ++ (0,-3.5);
\node[permpt] (15) at (1,15) {};  \draw[thin] (15) -- ++ (15.5,0);
\node[permpt] (16) at (2,18) {}; \draw[thin] (16) -- ++ (0,-3.5);

\draw[thin,dashed] (19,17) -- ++ (-17.5,0);

\node[circle,fill,inner sep=0.5pt] () at (19,17.5) {};
\node[circle,fill,inner sep=0.5pt] () at (19.5,18) {};
\node[circle,fill,inner sep=0.5pt] () at (20,18.5) {};
	
\end{tikzpicture}}
\end{center}
\caption{The beginning of the infinite pin diagram defined by the infinite pin word $w = 1\overline{(ulur)}$. The centred pin class $\mathcal{V}^{\circ}_{0}$ consists of all centred permutations that can be found somewhere in this infinite diagram, whilst $\mathcal{V}_{0}$ is its underlying (uncentred) permutation class.}
\label{fig:2pinclassV}
\end{figure}

\begin{example}[Pin Class $\mathcal{V}_{0}$] \label{ex:defV}
The pin class $\mathcal{V}_{0}$ is the underlying permutation class of the centred class $\mathcal{V}^{\circ}_{0} = \mathcal{C}^{\circ}_{w}$, where $w = 1\overline{(ulur)}$ --- see Figure \ref{fig:2pinclassV}. Writing this infinite pin word out with letters subscripted by the corresponding quadrant number
\[
w = 1u_{1}l_{2}u_{2}r_{1}u_{1}l_{2}u_{2}r_{1}u_{1}l_{2}u_{2}r_{1}u_{1}l_{2}u_{2}r_{1} \dots \ ,
\]
we note immediately that $w$ has $2$ pin factors of length $1$ and $4$ of every greater length. By comparison with the list of $\boxplus$-decomposables and collisions stated in Theorem \ref{classification}, we note that we have to remove $2$ from this count at lengths $2$ and $3$ (due to the $\boxplus$-decomposables $1l$ and $2r$ at length $2$ and the collisions $\{1ul,2ru\}$ and $\{2ur,1lu\}$ at length $3$) to obtain the generating function of the $\boxplus$-indecomposables of $\mathcal{V}^{\circ}_{0}$:
\[
\begin{split}
g(z) & = 2z + 2z^2 + 2z^3 + 4z^4 + 4z^5 + 4z^6 + 4z^7 + \dots \\
 & = \frac{2z + 2z^4}{1-z} \ .
\end{split}
\]
As we are in the upper half-plane, $g_{3}(z)=g_{4}(z)=0$ so the amended $G$-sequence is just $g(z)$ itself, and the generating function of $\mathcal{V}^{\circ}_{0}$ is given by
\[
\begin{split}
f(z) & = \frac{1}{1 - g(z)} \\
 & = \frac{1-z}{1 - 3z - 2z^4} \ .
\end{split}
\]
Taking the Taylor expansion of this function about $z=0$ gives us the enumeration sequence of the centred class $\mathcal{V}^{\circ}_{0}$:
\[
f(z) = 2z + 6z^2 + 18z^3 + 56z^4 + 172z^5 + 528z^6 + 1620z^7 + 4972z^8 +\dots .
\]
Furthermore, we can apply Theorem \ref{EGT} to the generating function $f^{\circ}(z)$ to deduce that $gr(\mathcal{V}_{0})$, which we shall call $\nu$, is the reciprocal of the smallest real root of 
\[
1 - 3z - 2z^4 = 0;
\]
by calculation:
\[
\nu = gr(\mathcal{V}_{0}) \approx 3.06918 \ .
\]
\end{example}

The pin class $\mathcal{V}_{0}$ is the second-smallest pin class, in the sense of having the second-smallest growth rate after $\kappa$, though we shall not prove this here. Both of the pin classes we have enumerated so far are members of a more general family of pin classes we call $\mathcal{V}(1,k)$.

\begin{example}[Pin Classes $\mathcal{V}(1,k)$] \label{v1kcentem}
Let $k \geq 0$. We denote by $\mathcal{V}^{\circ}(1,k)$ the centred pin class $\mathcal{C}^{\circ}_{w_{1,k}}$ generated by the infinite pin word
\[
w_{1,k} = 1\overline{(u(lu)^{k}r)} = 1u\underbrace{\mbox{$lululu\ldots lulu$}}_{\mbox{{\small $2k$}}}ru\underbrace{\mbox{$lululu\ldots lulu$}}_{\mbox{{\small $2k$}}}ru\underbrace{\mbox{$lululu\ldots lulu$}}_{\mbox{{\small $2k$}}}r \dots
\]
We similarly write $\mathcal{V}(1,k)$ for its underlying (uncentred) pin permutation class. See Figure \ref{fig:v1k} for an illustration. Informally, the pin diagram stays for one `turn' (ie. two points) on the right, then takes $k$ turns on the left, then repeats. Note that $\mathcal{V}^{\circ}(1,0) = \mathcal{O}^{\circ}$, the class of increasing oscillations in the first quadrant, and $\mathcal{V}^{\circ}(1,1) = \mathcal{V}^{\circ}$, as enumerated in the previous example.
\end{example}

\begin{figure}[h]
\begin{center}
\begin{tikzpicture}[scale=0.35]

\node[circle, draw, fill=none, inner sep=0pt, minimum width=\plotptradius] (0) at (11,0) {};

\node[permpt] (1) at (13,1) {};
\node[permpt] (2) at (12,3) {}; \draw[thin] (2) -- ++ (0,-2.5);

\node[permpt] (3) at (9,2) {}; \draw[thin] (3) -- ++ (3.5,0);
\node[permpt] (4) at (10,4) {}; \draw[thin] (4) -- ++ (0,-2.5);

\node[circle,fill,inner sep=0.5pt] () at (9,3.5) {};
\node[circle,fill,inner sep=0.5pt] () at (8.5,4) {};
\node[circle,fill,inner sep=0.5pt] () at (8,4.5) {};

\node[permpt] (5) at (6,5) {}; \draw[thin] (5) -- ++ (1.5,0);
\node[permpt] (6) at (7,7) {}; \draw[thin] (6) -- ++ (0,-2.5);
\draw[decorate,decoration={brace,amplitude=10},-] (9,1.5) -- (5.5,5);
\node[] () at (6,2) {\small{2k}};

\node[permpt] (7) at (15,6) {}; \draw[thin] (7) -- ++ (-8.5,0);
\node[permpt] (8) at (14,9) {}; \draw[thin] (8) -- ++ (0,-3.5);

\begin{scope}[shift={(-5,6)}]
\node[permpt] (3) at (9,2) {}; \draw[thin] (3) -- ++ (10.5,0);
\node[permpt] (4) at (10,4) {}; \draw[thin] (4) -- ++ (0,-2.5);

\node[circle,fill,inner sep=0.5pt] () at (9,3.5) {};
\node[circle,fill,inner sep=0.5pt] () at (8.5,4) {};
\node[circle,fill,inner sep=0.5pt] () at (8,4.5) {};

\node[permpt] (5) at (6,5) {}; \draw[thin] (5) -- ++ (1.5,0);
\node[permpt] (6) at (7,7) {}; \draw[thin] (6) -- ++ (0,-2.5);
\draw[decorate,decoration={brace,amplitude=10},-] (9,1.5) -- (5.5,5);
\node[] () at (6,2) {\small{2k}};
\end{scope}

\node[permpt] (13) at (17,12) {}; \draw[thin] (13) -- ++ (-15.5,0);
\node[permpt] (14) at (16,15) {}; \draw[thin] (14) -- ++ (0,-3.5);

\draw[thick] (11,0) -- ++ (0,17);
\draw[thick] (-1,0) -- ++ (19,0);

\draw[thin,dashed] (0,14) -- ++ (16.5,0);

\begin{scope}[shift={(-9,11)}]
\node[circle,fill,inner sep=0.5pt] () at (9,3.5) {};
\node[circle,fill,inner sep=0.5pt] () at (8.5,4) {};
\node[circle,fill,inner sep=0.5pt] () at (8,4.5) {};
\end{scope}

%\draw[dotted] (8.5,7.5) rectangle (9.5,8.5);

%\node[] (13) at (6,-1) {$1dl(dl)^{*}dlu$};

\end{tikzpicture}
\end{center}
\caption{The beginning of the pin diagram for $\mathcal{V}(1,k)$; the $\mathcal{V}$-class $\mathcal{V}(1,k)$ consists of all permutations which can be found somewhere in this infinite diagram.}
\label{fig:v1k}
\end{figure}

We wish to determine the growth rates of the pin classes $\mathcal{V}(1,k)$, which we shall do by enumerating the centred classes $\mathcal{V}^{\circ}(1,k)$. As we have already calculated $gr(\mathcal{V}(1,0)) = \kappa \approx 2.20557$ and $gr(\mathcal{V}(1,1)) = \nu \approx 3.06918$ we shall assume that $k \geq 2$. By following Procedure \ref{recproc} we may now prove:

\begin{prop}[Enumeration of $\mathcal{V}^{\circ}(1,k)$] \label{ajfsd;kl}
Suppose $k \geq 2$. Then the generating function of the centred pin class $\mathcal{V}^{\circ}(1,k)$ is given by
\begin{equation} \label{kiosdfu8979i}
f(z) = \frac{(1-z)^2}{1 - 4z + 3z^2 - 2z^3 - z^4 + 2z^5 + z^{2k}}.
\end{equation}
\end{prop}

\begin{proof}
We begin by enumerating the pin factors of the (recurrent) infinite pin word $w_{1,k}$. In order to do this it is again useful to subscript each letter by the associated quadrant number:
\begin{equation} \label{fs8d0oufwui}
w_{1,k} = 1u_{1}\underbrace{\mbox{$l_{2}u_{2}l_{2}u_{2}\ldots l_{2}u_{2}l_{2}u_{2}$}}_{\mbox{{\small $2k$}}}r_{1}u_{1}\underbrace{\mbox{$l_{2}u_{2}l_{2}u_{2}\ldots l_{2}u_{2}l_{2}u_{2}$}}_{\mbox{{\small $2k$}}}r_{1}u_{1}\underbrace{\mbox{$l_{2}u_{2}l_{2}u_{2}\ldots l_{2}u_{2}l_{2}u_{2}$}}_{\mbox{{\small $2k$}}}r_{1} \dots
\end{equation}
We shall separately enumerate the pin factors of $w_{1,k}$ beginning with each quadrant numeral. Clearly there is one pin factor of length $1$ beginning with the numeral $1$ (namely the pin word $1$ itself), and two pin factors beginning with numeral $1$ of each length at least $2$, corresponding to the two possible starting positions at $u_{1}$ and $r_{1}$ within the period (both of which become distinct from length $2$). Hence the generating function of pin factors of $w_{1,k}$ beginning with the numeral $1$ is given by
\[
\begin{split}
h_{1}(z) & = z + 2z^2 + 2z^3 + 2z^4 + 2z^5 + \dots\\
 & = \frac{z+z^2}{1 - z} \\
 & = \frac{z-z^3}{(1-z)^2} \ .
\end{split}
\]
Next, we enumerate the pin factors of $w_{1,k}$ beginning with the numeral $2$. By an immediate count we may see that there is one such pin factor of length $1$ (namely the pin word $2$) and three of length $2$ (namely $2r$, $2u$ and $2l$). Longer lengths will prove somewhat more difficult: clearly at all sufficently-long lengths there will be $2k$ pin factors of $w_{1,k}$ beginning with the numeral $2$, corresponding to the $2k$ possible starting points in the period, those enveloped by the braces in (\ref{fs8d0oufwui}). But for shorter lengths these will not all be distinct. Precisely, a pin word starting at a point to the left of the rightmost $l_{2}$ and $u_{2}$ in a $2k$-block in (\ref{fs8d0oufwui}) will have to be long enough to reach the rightmost $l_{2}$ in order to be distinguished from the pin word of the same length starting two places to the right. Thus at each new length $n \geq 3$ one more point from inside the $2k$-block will `activate', in the sense of being the starting point for a pin word distinct from all beginning to the right, until eventually all $2k$ starting points generate distinct pin words. Hence the generating function of the pin factors of $w_{1,k}$ beginning with the numeral $2$ is given by
\[
\begin{split}
h_{2}(z) & = z + 3z^2 + 4z^3 + 5z^4 + 6z^5 \dots + (2k-1)z^{2k-2} + 2kz^{2k-1} + 2kz^{2k} + 2kz^{2k+1} + \dots\\
 & = \frac{z + 2z^2 + z^3 + z^4 + z^5 + \dots + z^{2k-1}}{1 - z} \\
 & = \frac{z + z^2 - z^3 - z^{2k}}{(1-z)^2} \ .
\end{split}
\]
Combining these expressions we obtain the generating function of all pin factors of $w_{1,k}$:
\[
\begin{split}
h(z) & = h_{1}(z) + h_{2}(z) \\
 & = \frac{2z + z^2 - 2z^3 - z^{2k}}{(1-z)^2} \ .
\end{split}
\]
Next, in order to obtain the generating function of $\boxplus$-indecomposables in $\mathcal{V}^{\circ}(1,k)$, we must remove from those counted in $h(z)$ all $\boxplus$-decomposable pin words along with the overcount casued by collisions. On inspecting the lists from Theorem \ref{classification} we see that amongst the pin factors of $w_{1,k}$ there are precisely two $\boxplus$-decomposables of length $2$ ($1l$ and $2r$), along with one colliding pair at length $2$ ($\left\{2l,2u\right\}$) and two more colliding pairs at length $3$ ($\{1ul,2ru\}$ and $\{2ur,1lu\}$); there are no more $\boxplus$-decomposables or collisions. Hence the generating function of $\boxplus$-indecomposables in $\mathcal{V}^{\circ}(1,k)$ is given by:
\[
\begin{split}
g(z) & = h(z) - 3z^2 - 2z^3 \\
 & = \frac{2z - 2z^2 + 2z^3 + z^4 - 2z^5 - z^{2k}}{(1-z)^2} \ .
\end{split}
\]
Again, $g_{3}(z)=g_{4}(z)=0$, so $G(z) = g(z)$, and so by the Generating Function Specification \ref{genfuncspec} the generating function of $\mathcal{V}^{\circ}(1,k)$ is given by
\[
\begin{split}
f(z) & = \frac{1}{1 - g(z)} \\
 & = \frac{(1-z)^2}{1 - 4z + 3z^2 - 2z^3 - z^4 + 2z^5 + z^{2k}} \ ,
\end{split}
\]
as required. \qedhere
\end{proof}

Finally, by the Exponential Growth Theorem \ref{EGT}, the growth rate of $\mathcal{V}^{\circ}(1,k)$ (and hence also of $\mathcal{V}(1,k)$) is the reciprocal of the smallest positive real root of the denominator of (\ref{kiosdfu8979i}), thus allowing us to extract from Proposition \ref{ajfsd;kl} the growth rates of the (uncentred) pin classes $\mathcal{V}(1,k)$:

\begin{cor}[Growth rate of $\mathcal{V}(1,k)$] \label{cor:v1kdefs}
Let $k \geq 0$ and define $\nu_{1,k} = gr(\mathcal{V}(1,k))$. Then
\begin{itemize}
\item $\nu_{1,0} = \kappa \approx 2.20557$, where $\kappa$ is defined as in Example \ref{ex:defO};
\item $\nu_{1,1} = \nu \approx 3.06918$, where $\nu$ is defined as in Example \ref{ex:defV};
\item For $k \geq 2$, $\nu_{1,k}$ is the reciprocal of the smallest positive real solution of the equation
\begin{equation} \label{kljdgfjsdl67222}
1 - 4z + 3z^2 - 2z^3 - z^4 + 2z^5 + z^{2k} = 0.
\end{equation}
\end{itemize}
\end{cor}

\begin{proof}
We have already established the cases $k=0,1$ in Examples \ref{ex:defO} and \ref{ex:defV}. For $k \geq 2$ we apply Theorem \ref{EGT} to Proposition \ref{ajfsd;kl}.
\end{proof}

We can now explicitly calculate approximate values of $gr(\mathcal{V}(1,k))$ for small $k$ (all values given correct to $7$ d.p.):

\begin{center}
\begin{tabular}{||c | c ||} 
 \hline
  $k$ & $\nu_{1,k} = gr(\mathcal{V}(1,k))$ \\ [0.8ex] 
 \hline \hline
0 & $\kappa \approx 2.2055694$ \\ 
 \hline
 1 & $\nu \approx 3.0691774$ \\ 
 \hline
 2 & $3.2479584$ \\
 \hline
 3 & $3.2796313$ \\
 \hline
 4 & $3.2824809$ \\
 \hline
 5 & $3.2827441$ \\
 \hline
 6 & $3.2827685$ \\
 \hline
 7 & $3.2827707$ \\
 \hline
\end{tabular}
\end{center}

We note that $(\nu_{1,k})_{k=0}^{\infty}$ appears to be a strictly increasing sequence which converges to a value at approximately $3.28277$. These facts can both be proved by basic real analysis on consideration of the defining equation (\ref{kljdgfjsdl67222}). In particular, the value $\lim_{k \rightarrow \infty}\nu_{1,k}$ is equal to the reciprocal of the smallest positive real root of the polynomial equation obtained by removing the $z^{2k}$-term from equation $(\ref{kljdgfjsdl67222})$, as this term becomes negligible on the interval $[0,1)$ as $k \rightarrow \infty$. We shall have further cause to refer to this constant later, so we formally state its definition:

\begin{defn} \label{def:nul}
We denote
\[
p_{\mathcal{L}}(z) = 1 - 4z + 3z^2 - 2z^3 - z^4 + 2z^5
\]
and define $\nu_{\mathcal{L}}$ to be the reciprocal of the smallest positive real root of $p_{\mathcal{L}}(z)$. By explicit calculation,
\[
\nu_{\mathcal{L}} \approx 3.28277097.
\]
\end{defn}

As $\lim_{k \rightarrow \infty}\nu_{1,k} = \nu_{\mathcal{L}}$, the constant $\nu_{\mathcal{L}}$ is an accumulation point in the set of pin class growth rates. It is natural to ask whether there is a pin class which achieves this growth rate. The answer to this is yes: there are in fact uncountably-many distinct pin classes with growth rate $\nu_{\mathcal{L}}$, and this is the smallest growth rate with this property, though we shall not prove this here. We will, however, be able to give one example of a (non-recurrent) pin class with this growth rate at the end of this paper in Section \ref{sec:4.4}.

\subsection{Examples in Three and Four Quadrants} \label{sec:3.8}

We now move on to consider some pin classes in three and four quadrants, which will force us to confront commutativity. We begin with what is, in a sense, the simplest pin class in three quadrants:

\begin{example}[The class $\mathcal{Y}$] \label{ex:gammadef}
Let $w$ be the infinite pin word $1\overline{(uldlur)}$. We refer to $\mathcal{C}_{w}$, the associated pin class, as $\mathcal{Y}$ --- see Figure \ref{fig:Y}. Then the growth rate of $\mathcal{Y}$ is $\gamma$, the reciprocal of the smallest positive real root of
\[
1 - 4z + 2z^2 + z^3 - z^4 - 2z^5 - 3z^6 = 0.
\]
Explicitly, $\gamma \approx 3.36637$.
\end{example}

\begin{figure}[h]
% Here you can include a picture (.pdf, .png,...) with \includegraphics[]{} or a tikz plot:
\begin{center}
\begin{tikzpicture}[scale=0.35]

\node[circle, draw, fill=none, inner sep=0pt, minimum width=\plotptradius] (0) at (9,5) {};

\node[permpt] (1) at (11,6) {}; 
\node[permpt] (2) at (10,8) {}; \draw[thin] (2) -- ++ (0,-2.5);
\node[permpt] (3) at (7,7) {}; \draw[thin] (3) -- ++ (3.5,0);
\node[permpt] (4) at (8,3) {}; \draw[thin] (4) -- ++ (0,4.5);
\node[permpt] (5) at (5,4) {}; \draw[thin] (5) -- ++ (3.5,0);
\node[permpt] (6) at (6,10) {}; \draw[thin] (6) -- ++ (0,-6.5);
\node[permpt] (7) at (13,9) {}; \draw[thin] (7) -- ++ (-7.5,0);
\node[permpt] (8) at (12,12) {}; \draw[thin] (8) -- ++ (0,-3.5);
\node[permpt] (9) at (3,11) {}; \draw[thin] (9) -- ++ (9.5,0);
\node[permpt] (10) at (4,1) {}; \draw[thin] (10) -- ++ (0,10.5);
\node[permpt] (11) at (1,2) {}; \draw[thin] (11) -- ++ (3.5,0);
\node[permpt] (12) at (2,13) {}; \draw[thin] (12) -- ++ (0,-11.5);

\draw[thick] (9,0) -- ++ (0,14);
\draw[thick] (0,5) -- ++ (14,0);

\end{tikzpicture}
\end{center}
\caption{The beginning of the pin diagram for $\mathcal{Y}^{\circ}$, generated by the infinite pin word $w = 1\overline{(uldlur)}$.}
\label{fig:Y}
\end{figure}

\begin{proof}
We shall find the generating function of the corresponding centred class $\mathcal{Y}^{\circ} = \mathcal{C}^{\circ}_{w}$. We note that $w$, being periodic, is recurrent and we can thus use Procedure \ref{recproc}. First, we enumerate the pin factors of $w=1\overline{(uldlur)}$; this will be easier if we first write out $w$ with letters subscripted by the corresponding quadrant number:
\[
w = 1u_{1}l_{2}d_{3}l_{3}u_{2}r_{1}u_{1}l_{2}d_{3}l_{3}u_{2}r_{1}u_{1}l_{2}d_{3}l_{3}u_{2}r_{1}u_{1}l_{2}d_{3}l_{3}u_{2}r_{1}\dots \ .
\]
Clearly there are three pin factors of length $1$, and for $n \geq 2$ the six distinct starting points in the period give six pin factors of length $n$. Hence the generating function of the pin factors of $w$ is
\[
h(z) = 3z + 6z^2 + 6z^3 + 6z^4 + 6z^5 + 6z^6 + 6z^7 + \dots \ .
\]
We now list the collisions and $\boxplus$-decomposables found amongst the pin factors of $w$ by comparison with the lists given in Theorem \ref{classification}:
\[
\begin{split}
\textbf{Collisions:} \  & \{1ul, 2ru\}, \{2dl, 3lu\} \ (\text{length} \ 3)\\
 & \{1uldl, 3luru\}, \{1ldlu, 2dlur\}, \{3urul, 2ruld\} \ (\text{length} \ 5) \\
\textbf{$\boxplus$-decomposables:} \ & 1l, 2d, 3u, 2r \ (\text{length} \ 2) \\
& 1ld, 3ur \ (\text{length} \ 3) \\
& 1uld, 1ldl, 3lur, 3uru \ (\text{length} \ 4)
\end{split}
\]
Hence we subtract these from the generating function of pin factors to obtain the generating function of $\boxplus$-indecomposables of $\mathcal{Y}^{\circ}$:
\[
\begin{split}
g(z) & = h(z) - (2z^3 + 3z^5) - (4z^2 + 2z^3 + 4z^4) \\
 & = 3z + 2z^2 + 2z^3 + 2z^4 + 3z^5 + 6z^6 + 6z^7 + 6z^8 + 6z^9 + \dots
\end{split}
\]
Next we must enumerate the one-quadrant $\boxplus$-indecomposables in each quadrant; this is easily-done by looking at the diagram: $g_{1}(z) = g_{3}(z) = z + z^2$, $g_{2}(z)=z$ and $g_{4}(z)=0$. Hence we can calculate the amended $G$-sequence:
\[
\begin{split}
G(z) & = g(z) - g_{1}(z)g_{3}(z) - g_{2}(z)g_{4}(z) \\
 & = 3z + z^2 + z^4 + 3z^5 + 6z^6 + 6z^7 + 6z^8 + 6z^9 + \dots \\
 & = \frac{3z - 2z^2 - z^3 + z^4 + 2z^5 + 3z^6}{1-z} \ .
\end{split}
\]
Finally, by the Generating Function Specification \ref{genfuncspec}, the generating function of $\mathcal{Y}^{\circ}$ is given by:
\[
\begin{split}
f(z) & = \frac{1}{1 - G(z)} \\
 & = \frac{1-z}{1-4z+2z^{2}+z^{3}-z^{4}-2z^{5}-3z^{6}} \ .
\end{split}
\]
This is a rational function, whose singularities are precisely the roots of
\[
1-4z+2z^{2}+z^{3}-z^{4}-2z^{5}-3z^{6} = 0.
\]
By computation, the smallest root of this equation is positive and real and has reciprocal $\gamma \approx 3.36637$, which by the Exponential Growth Theorem is therefore the growth rate of $\mathcal{Y}$.\end{proof}

Next, we consider a four-quadrant pin class that has in fact been studied before in the literature: the Widdershins Spiral, $\mathcal{W}^{\circ}$, generated by the infinite pin word $w = 1\overline{(ldru)}$ (see Figure \ref{fig:widdershins}). The uncentred class $\mathcal{W}$ was first introduced by Murphy~\cite{murphy:restricted-perm:}, who gave the generating function of the uncentred class by an explicit enumeration. Using our specification we can now give an alternative derivation of its growth rate:

\begin{figure}[h]
\begin{center}
\begin{tikzpicture}[scale=0.35]

\node[circle, draw, fill=none, inner sep=0pt, minimum width=\plotptradius] (0) at (7,7) {};
\node[permpt] (1) at (8,9) {};
\node[permpt] (2) at (5,8) {}; \draw[thin] (2) -- ++ (3.5,0);
\node[permpt] (3) at (6,5) {}; \draw[thin] (3) -- ++ (0,3.5);
\node[permpt] (4) at (10,6) {}; \draw[thin] (4) -- ++ (-4.5,0);
\node[permpt] (5) at (9,11) {}; \draw[thin] (5) -- ++ (0,-5.5);
\node[permpt] (6) at (3,10) {}; \draw[thin] (6) -- ++ (6.5,0);
\node[permpt] (7) at (4,3) {}; \draw[thin] (7) -- ++ (0,7.5);
\node[permpt] (8) at (12,4) {}; \draw[thin] (8) -- ++ (-8.5,0);
\node[permpt] (9) at (11,13) {}; \draw[thin] (9) -- ++ (0,-9.5);
\node[permpt] (10) at (1,12) {}; \draw[thin] (10) -- ++ (10.5,0);
\node[permpt] (11) at (2,1) {}; \draw[thin] (11) -- ++ (0,11.5);
\node[permpt] (12) at (14,2) {}; \draw[thin] (12) -- ++ (-12.5,0);
\node[permpt] (13) at (13,14) {}; \draw[thin] (13) -- ++ (0,-12.5);

%\node[circle,fill,inner sep=0.5pt] () at (3,5) {};
%\node[circle,fill,inner sep=0.5pt] () at (3.5,4.5) {};
%\node[circle,fill,inner sep=0.5pt] () at (4,4) {};
%\node[circle,fill,inner sep=0.5pt] () at (4.5,3.5) {};

%\node[circle,fill,inner sep=0.5pt] (8) at (5.5,4.5) {};

\draw[thick] (0,7) -- ++ (15,0);
\draw[thick] (7,0) -- ++ (0,15);

%\draw[decorate,decoration={brace,amplitude=10},-] (5.5,0.5) -- (0.5,4.5);
%\node[] () at (2,1.25) {\small{2k}};

%\draw[decorate,decoration={brace,amplitude=10},-] (15.5,10.5) -- (10.5,5.5);
%\node[] () at (14.25,6.75) {\small{2l}};

%\draw[thin,dashed] (0,1) -- ++ (17.5,0);

%\draw[dotted] (8.5,7.5) rectangle (9.5,8.5);

%\node[] (13) at (6,-1) {$1dl(dl)^{*}dlu$};

\end{tikzpicture}
\end{center}
\caption{The Widdershins Spiral: this is the smallest pin class in four quadrants.}
\label{fig:widdershins}
\end{figure}

\begin{example}[The Widdershins Spiral]\label{ex:Wdef}
Let $\mathcal{W} = \mathcal{C}_{w}$ be the pin class generated by the infinite pin word $w = 1\overline{(ldru)}$.
Then the generating function of the corresponding centred class $\mathcal{W}^{\circ}$ is given by:

\begin{equation*}
f(z) = \frac{1-z}{1 - 5z + 6z^2 - 2z^3 - z^4 - 3z^5} \ .
\end{equation*}

Hence $\omega_{0} = gr(\mathcal{W}^{\circ}) \approx 3.48806$.
\end{example}

\begin{proof}
The defining infinite pin word $w$ is recurrent so we can enumerate $\mathcal{W}^{\circ}$ using our standard procedure. We first determine the generating function of pin factors of $w$: as $w$ has period $4$ and all four pin factors are already distinct at length $1$ we have:
\[
h(z) = 4z + 4z^2 + 4z^3 + 4z^4 + 4z^5 + 4z^6 + \dots
\]
We must amend this for collisions and $\boxplus$-decomposables to obtain the generating function $g(z)$ of $\boxplus$-indecomposables in $\mathcal{W}^{\circ}$. By comparision with the lists in Theorem \ref{classification} we see that the only collision amongst the pin factors of $w$ is the colliding quadruple $\{1ldr, 2dru, 3rul, 4uld\}$ and that there are precisely four $\boxplus$-decomposables of length $2$ ($1l$, $2d$, $3r$ and $4u$), four $\boxplus$-decomposables of length $3$ ($1ld$, $2dr$, $3ru$ and $4ul$), and none of length $n \geq 4$. Hence the generating function of $\boxplus$-indecomposables in $\mathcal{W}^{\circ}$ is given by:
\[
\begin{split}
g(z) & = h(z) - 3z^4 - (4z^2 + 4z^3) \\
 & = 4z + z^4 + 4z^5 + 4z^6 + 4z^7 + 4z^8 + \dots \\
 & = \frac{4z - 4z^2 + z^4 + 3z^5}{1-z} \ .
\end{split}
\]
Next we note that there is precisely $1$ one-quadrant $\boxplus$-indecomposable of length $1$ in each quadrant and none of any greater length, so
\[
g_{1}(z) = g_{2}(z) = g_{3}(z) = g_{4}(z) = z.
\]
Hence the amended $G$-sequence of $\mathcal{W}^{\circ}$ is given by:
\[
\begin{split}
G(z) & = g(z) - g_{1}(z)g_{3}(z) - g_{2}(z)g_{4}(z) \\
 & = \frac{4z - 4z^2 + z^4 + 3z^5}{1-z} - 2z^2 \\
 & = \frac{4z - 6z^2 + 2z^3 + z^4 + 3z^5}{1-z} \ .
\end{split}
\]
And so, by the Generating Function Specification \ref{genfuncspec}, we can derive the generating function of $\mathcal{W}^{\circ}$:
\[
\begin{split}
f(z) & = \frac{1}{1 - G(z)} \\
 & = \frac{1-z}{1 - 5z + 6z^2 - 2z^3 - z^4 - 3z^5} \ ,
\end{split}
\]
as required. Finally, we can calculate the growth rate $\omega_{0}$ of $\mathcal{W}^{\circ}$ as the reciprocal of the smallest positive real root of the denominator: $\omega_{0} \approx 3.48806$. \end{proof}

We conclude this section by finding the growth rate of the permutation class consisting of \emph{all} pin permutations. The is the \emph{complete pin class}, mentioned earlier in this section. We reiterate its definition here:

\begin{defn}[The Complete Pin Class]
The \emph{complete pin class} $\mathcal{P}$ is the permutation class consisting of all pin permutations. Its centred counterpart $\mathcal{P}^{\circ}$ is the centred permutation class containing all centred pin permutations.
\end{defn}

This uncentred class $\mathcal{P}$ was enumerated by Bassino, Bouvel \& Rossin~\cite{bassino:enumeration-of-:}: we stated its generating function and growth rate $\omega_{\infty} \approx 5.24112$ in Section \ref{sec:1}. Clearly, $\mathcal{P}$ is a permutation class which contains every pin permutation class as a subclass, but it may not be immediately obvious that $\mathcal{P}$ is \emph{itself} a pin class (that is, that there is some infinite pin word for which $\mathcal{C}_{w} = \mathcal{P}$), but this is in fact true. We prove this below, as well as enumerating the \emph{centred} class $\mathcal{P}^{\circ}$, which will enable us to provide an alternative derivation of the growth rate $\omega_{\infty}$:

\begin{prop}[Complete Pin Class] \ \label{prop:comppinclassenum}
\begin{enumerate}
\item There is a pin sequence $w_{c}$ which contains every finite pin word as a recurrent pin factor.
\item For any such pin sequence $w_{c}$, $\mathcal{C}^{\circ}_{w_{c}} = \mathcal{P}^{\circ}$.
\item The complete pin class $\mathcal{P}$ has growth rate $\omega_{\infty} \approx 5.24112$, where $\omega_{\infty}$ is defined to be the reciprocal of the smallest positive real root of the equation

\[
1 - 8z + 19z^2 - 26z^3 + 14z^4 - 12z^5 - 8z^6 + 20z^7 - 8z^8 = 0 \ .
\]
\end{enumerate}
\end{prop}

\begin{proof}
\begin{enumerate}
\item Recall that we write $\overline{\mathcal{L}}_{P}$ to refer to the set of finite words which can occur after the initial quadrant numberal in a pin word. Equivalently, $\overline{\mathcal{L}}_{P}$ is the set of words over the alphabet $\left\{u,d,l,r\right\}$ which alternate between $\left\{u,d\right\}$ and $\left\{l,r\right\}$. Note that $\overline{\mathcal{L}}_{P}$ is countable, and so we can order all $\overline{\mathcal{L}}_{P}$-words and concatenate them (after any initial numeral) in this order, possibly placing a letter in between consecutive words to ensure the alignments alternate. This resulting infinite pin word $w_{c}$ will contain all $\overline{\mathcal{L}}_{P}$-words as subword factors, and will hence contain all possible finite pin words as pin factors by Lemma \ref{pfsubwcorr}.
\item This is clear by the definition of a pin permutation.
\item First, we enumerate the set of pin factors of $w_{c}$, which is to say the set of all finite pin words. Clearly, there are $4$ pin words of length $1$. For length $n \geq 2$, there are $4$ choices for the initial numeral, $4$ choices for the letter in second place, and then $2$ choices for every subsequent letter as the alignments must now alternate. Hence there are $2^{n+2}$ pin words of length $n \geq 2$, and the generating function of the set of all finite pin words is given by:
\[
h(z) = 4z + \sum_{n=2}^{\infty}2^{n+2}z^n.
\]
This will be an overcount of the $\boxplus$-indecomposables in $\mathcal{P}^{\circ}$ due to collisions and $\boxplus$-decomposable pin words, but we have already classified all of these in Theorem \ref{classification}. The generating function of the overcount due to collisions is
\[
h_{\text{Col}}(z) = 4z^2 + 8z^3 + 6z^4 + 12z^5 + 8z^6 + 8z^7 + 8z^8 + 8z^9 + \dots
\]
Note that this is the generating function of the \emph{overcount} due to collisions. At all $n \neq 4$ it is simply equal to the \emph{number} of collisions, as these are all pairs so we need only take away one from the count for each collision. But at $n=4$ there are two colliding \emph{quadruples} so the overcount is $6$.

Similarly, and also from Theorem \ref{classification}, the generating function of the $\boxplus$-decomposable pin words in $\mathcal{L}_{P}$ is
\[
h_{\boxplus\text{-Dec.}}(z) = 8z^2 + 8z^3 + 16z^4 + 16z^5 + 16z^6 + 16z^7 + 16z^8 + 16z^9 + \dots
\]
Subtracting $h_{\text{Col}}(z)$ and $h_{\boxplus\text{-Dec.}}(z)$ from $h(z)$ gives us the generating function of the set of all $\boxplus$-indecomposable pin permutations:
\[
\begin{split}
g(z) & = h(z) - h_{\text{Col}}(z) - h_{\boxplus\text{-Dec.}}(z) \\
 & = 4z + 4z^2 + 16z^3 + 42z^4 + 100z^5 + \sum_{n=6}^{\infty}(2^{n + 2} - 24)z^n \\
 & = \frac{4z - 8z^2 + 12z^3 + 2z^4 + 6z^5 + 16z^6 - 8z^7}{(1-z)(1-2z)} \ .
\end{split}
\]
Next, we shall require the generating functions $g_{i}(z)$ ($i \in \{1,2,3,4\}$) of the $\boxplus$-indecomposable pin permutations entirely contained in the $i$th quadrant. Clearly, these are all equal to the generating function of the one-quadrant oscillations in a specified quadrant; namely:
\[
\begin{split}
g_{1}(z) = g_{2}(z) = g_{3}(z) = g_{4}(z) & = z + z^2 + 2z^3 + 2z^4 + 2z^5 + 2z^6 + \dots \\
 & = \frac{z + z^3}{1-z} \ .
\end{split}
\]
Hence we obtain the amended $G$-sequence for the complete class of pin permutations:
\[
\begin{split}
G_{\infty}(z) & = g(z) - g_{1}(z)g_{3}(z) - g_{2}(z)g_{4}(z) \\
 & = \frac{4z - 14z^2 + 24z^3 - 14z^4 + 12z^5 + 8z^6 - 20z^7 + 8z^8}{(1-z)^{2}(1-2z)} \ .
\end{split}
\]
Finally, as $\mathcal{P}^{\circ}$ is $\boxplus$-closed\footnote{By the fact that $w_{c}$ is necessarily recurrent, or simply by the fact that the set of centred pin permutations is closed under the $\boxplus$-sum.}, we can apply the Generating Function Specification \ref{genfuncspec} to obtain the generating function for the class of all centred pin permutations:
\[
\begin{split}
f(z) & = \frac{1}{1 - G_{\infty}(z)} \\
 & = \frac{(1-z)^{2}(1-2z)}{1 - 8z + 19z^2 - 26z^3 + 14z^4 - 12z^5 - 8z^6 + 20z^7 - 8z^8} \ .
\end{split}
\]
By the Exponential Growth Theorem \ref{EGT}, and the fact that $f(z)$ has positive coefficients, the growth rate of $\mathcal{P}^{\circ}$, and hence also the uncentred class $\mathcal{P}$, is the reciprocal of the smallest positive real root of the denominator, as required.\qedhere
\end{enumerate}
\end{proof}

We note that the denominator of the generating function of the centred complete class $\mathcal{P}^{\circ}$ is identical to that of the uncentred class $\mathcal{P}$ derived by Bassino \emph{et al.}, stated in (\ref{eq:completepinclass}). This of course reflects the fact the the centred and uncentred complete classes have the same growth rate, both obtained by taking the reciprocal of the smallest positive real root of the respective denominators.

As every pin class is contained in the complete pin class, we immediately deduce the following upper bound on the upper growth rate of a pin class.

\begin{lemma} \label{pcupperbound}
Suppose $w$ is an infinite pin word. Then
\[
\overline{gr(\mathcal{C}_{w})} \leq \omega_{\infty},
\]
where $\omega_{\infty} \approx 5.24112$ is defined as in Proposition \ref{prop:comppinclassenum}.
\end{lemma}

\section{Non-recurrent Pin Classes} \label{sec:4}

In this section we move on to consider pin permutation classes generated by non-recurrent infinite pin words: these are not $\boxplus$-closed, and so the methods we have outlined in the previous section will not directly apply. We will not be able to find the generating functions of these (centred or uncentred) classes in general, but we can at least find their growth rates by sandwiching between some closely related $\boxplus$-closed classes. In Section \ref{sec:4.1} we introduce the \emph{$\boxplus$-interior} $\mathcal{C}^{\boxplus}_{w}$ of a centred pin class $\mathcal{C}^{\circ}_{w}$, which we may think of as the largest $\boxplus$-closed subclass of $\mathcal{C}^{\circ}_{w}$. We describe the close relationship between this class and the \emph{recurrent pin factors} of $w$, and use this to prove that $\boxplus$-interiors have proper growth rates. In Section \ref{sec:4.2} we state and prove some technical results pertaining to the amended $G$-sequences of $\boxplus$-interiors, and in Section \ref{sec:4.3} we apply these to prove the main result of this paper: that every pin class has a proper growth rate. We also outline a procedure for determining the growth rate of any pin class by reducing to the simpler combinatorial problem of enumerating the recurrent factors of its defining infinite pin word. We conclude in Section \ref{sec:4.4} by applying this process to determine the growth rate of a single non-recurrent pin permutation class, the \emph{Liouville} $\mathcal{V}$.

\subsection{The $\boxplus$-interior of a Pin Class} \label{sec:4.1}

We have now established a general method for determining the growth rate and generating function of a centred pin class $\mathcal{C}^{\circ}_{w}$ defined by a \emph{recurrent} infinite pin word $w$ in $\mathcal{L}^{\infty}_{P}$. This also allows us to determine the growth rate of an \emph{eventually} recurrent pin class, by the Finite Prefix Lemma \ref{fplem}. This leaves the non-eventually-recurrent case: these classes are not $\boxplus$-closed, so the methods outlined in the previous section will not enable us to determine their generating functions (in fact, the author does not know the generating function of \emph{any} not-eventually-recurrent pin class). Somewhat surprisingly, however, we \emph{can} at least determine the \emph{growth rate} of a general non-recurrent pin class $\mathcal{C}^{\circ}_{w}$. The key idea here is simple: the classes that we know how to enumerate are the $\boxplus$-closed classes, so we take better and better $\boxplus$-closed approximations of $\mathcal{C}^{\circ}_{w}$. The limiting behaviour of these approximations will give us the growth rate, but not the generating function, of $\mathcal{C}^{\circ}_{w}$.

We begin by noting that we already know how to bound $\mathcal{C}^{\circ}_{w}$ by a $\boxplus$-closed class from above: Procedure \ref{recproc} gives us the generating function of the $\boxplus$-closure of $\mathcal{C}^{\circ}_{w}$; in the non-recurrent case this is not equal to $\mathcal{C}^{\circ}_{w}$ but will function as an upper bound as $\mathcal{C}^{\circ}_{w} \subseteq \boxplus(\mathcal{C}^{\circ}_{w})$. We also note that, like recurrent pin classes, $\boxplus$-closures of pin classes have proper growth rates:

\begin{lemma}
Let $w$ be an infinite pin word. Then $\boxplus(\mathcal{C}^{\circ}_{w})$ has a proper growth rate.
\end{lemma}

\begin{proof}
This is an application of Theorem \ref{grexist}: $\boxplus(\mathcal{C}^{\circ}_{w})$ is a $\boxplus$-closed class contained in the complete pin class $\mathcal{P}^{\circ}$ and, as $\mathcal{C}^{\circ}_{w}$ satisfies the adjacency condition by Proposition \ref{pinclassessatisfytheadjacencycondition}, so must $\boxplus(\mathcal{C}^{\circ}_{w})$ (as a centred permutation class and its $\boxplus$-closure contain the same $\boxplus$-indecomposables, which of course include the one-point permutations).\end{proof}

Hence, for any infinite pin word $w$,
\[
\underline{gr}(\mathcal{C}^{\circ}_{w}) \leq \overline{gr}(\mathcal{C}^{\circ}_{w}) \leq \overline{gr}(\boxplus(\mathcal{C}^{\circ}_{w})) = gr(\boxplus(\mathcal{C}^{\circ}_{w})).
\]

We now wish to determine a \emph{lower} bound on the growth rate of a pin class, which we again do by comparison with a $\boxplus$-closed class whose generating function we can find. We thus consider the `largest $\boxplus$-closed subclass contained in $\mathcal{C}^{\circ}_{w}$', which we refer to as the $\boxplus$-\emph{interior} of $\mathcal{C}^{\circ}_{w}$. In order to define this rigorously, recall that if $w$ is an infinite pin word in $\mathcal{L}^{\infty}_{P}$, we use the notation $w_{i,j}$ to denote the pin factor of $w$ taken between the $i$th and $j$th places. Now we can define a pin factor analogue to recurrent factors of an infinite word:

\begin{defn}[Recurrent Pin Factors]
Let $w$ be an infinite pin word and $\widetilde{w}$ a finite pin word. We say that $\widetilde{w}$ is a \emph{recurrent pin factor} of $w$ if for all $n \in \mathbb{N}$ there exist $j \geq i \geq n$ such that $w_{i,j} = \widetilde{w}$.

We call a $\boxplus$-indecomposable permutation $\pi^{\circ}_{\widetilde{w}}$ generated by a recurrent pin factor $\widetilde{w}$ of $w$ a \emph{recurrent} $\boxplus$-indecomposable in $\mathcal{C}^{\circ}_{w}$.
\end{defn}

That is, a recurrent pin factor of $w$ is a pin factor that occurs in $w$ infinitely-often. We may now define:

\begin{defn}[The $\boxplus$-interior of a pin class] \label{def:boxinteriorofpinclass} \ \\
Let $w$ be an infinite pin word with associated pin class $\mathcal{C}^{\circ}_{w}$. We define the $\boxplus$-\emph{interior}, $\mathcal{C}^{\boxplus}_{w}$, of $\mathcal{C}^{\circ}_{w}$ to be the $\boxplus$-closure of the set of recurrent $\boxplus$-indecomposables in $\mathcal{C}^{\circ}_{w}$.
\end{defn}

We note that the set of recurrent pin factors of $w$ is closed under taking pin factors, and so by Observation \ref{pinfactorcontainment} every $\boxplus$-indecomposable of $\mathcal{C}^{\circ}_{w}$ which is contained in a recurrent $\boxplus$-indecomposable is itself a recurrent $\boxplus$-indecomposable. It follows from part \ref{lemma:boxprops2} of Lemma \ref{lemma:boxprops} that the $\boxplus$-interior of a pin class $\mathcal{C}^{\circ}_{w}$ can alternatively be characterised as:
\[
\mathcal{C}^{\boxplus}_{w} = \left\{\sigma^{\circ} = \pi^{\circ}_{1} \ \boxplus \ \pi^{\circ}_{2} \ \boxplus \ \dots \ \boxplus \ \pi^{\circ}_{n} \ \mid \ \text{each $\pi^{\circ}_{i}$ is a recurrent $\boxplus$-indecomposable of $\mathcal{C}^{\circ}_{w}$}\right\}.
\]

We now collect some basis properties of the $\boxplus$-interior:

\begin{prop}[Basic properties of the $\boxplus$-interior] \label{prop:boxprops}
For any $w \in \mathcal{L}^{\infty}_{P}$:
\begin{enumerate}
\item $\mathcal{C}^{\boxplus}_{w}$ is a (non-empty) $\boxplus$-closed subclass of $\mathcal{C}^{\circ}_{w}$; \label{boxprops:1}
\item $\mathcal{C}^{\boxplus}_{w}$ is the union of all $\boxplus$-closed subclasses of $\mathcal{C}^{\circ}_{w}$; \label{boxprops:2}
\item $\mathcal{C}^{\boxplus}_{w} = \mathcal{C}^{\circ}_{w}$ if and only if $\mathcal{C}^{\circ}_{w}$ is $\boxplus$-closed; \label{boxprops:3}
\item $\sigma^{\circ} \in \mathcal{C}^{\boxplus}_{w}$ if and only if for all $i \in \mathbb{N}$ there exists $j \geq i$ such that $\sigma^{\circ} \leq \pi^{\circ}_{w_{i,j}}$. \label{boxprops:4}
\end{enumerate}
\end{prop}

\begin{proof}
\begin{enumerate}
\item As $w$ is an infinite sequence whereas there are only four pin words of length $1$, at least one of these must occur infinitely-often as a pin factor of $w$. Hence the set of recurrent $\boxplus$-indecomposables is non-empty, and so its $\boxplus$-closure $\mathcal{C}^{\boxplus}_{w}$ is a non-empty $\boxplus$-closed centred permutation class. To see that $\mathcal{C}^{\boxplus}_{w}$ is contained in $\mathcal{C}_{w}$, note first that every element of $\mathcal{C}^{\boxplus}_{w}$ is of the form
\[
\sigma^{\circ} = \pi^{\circ}_{w_{1}} \ \boxplus \ \pi^{\circ}_{w_{2}} \ \boxplus \ \dots \ \pi^{\circ}_{w_{n}},
\]
where each $w_{i}$ is a recurrent pin factor of $w$. As each $w_{i}$ occurs as a pin factor infinitely-often in $w$, it is certainly the case that the sequence of pin factors $w_{1}, w_{2}, \dots, w_{n}$ occurs at some point in $w$ in that order, in non-overlapping instances and separated from each other by at least one letter in $w$. Hence by Theorem \ref{pindec}, $\sigma^{\circ}$ is in $\mathcal{C}^{\circ}_{w}$, and so $\mathcal{C}^{\boxplus}_{w} \subseteq \mathcal{C}^{\circ}_{w}$.
\item Suppose $\mathcal{D}^{\circ}$ is a $\boxplus$-closed subclass of $\mathcal{C}^{\circ}_{w}$, and let $\pi^{\circ}$ be a $\boxplus$-indecomposable in $\mathcal{D}^{\circ}$. As $\mathcal{D}^{\circ}$ is a subset of $\mathcal{C}^{\circ}_{w}$ it follows that $\pi^{\circ} = \pi^{\circ}_{\widetilde{w}}$, where $\widetilde{w}$ is a pin factor of $w$. As $\mathcal{D}^{\circ}$ is $\boxplus$-closed, for all $n \in \mathbb{N}$ the centred permutation
\[
\underbrace{\mbox{$\pi^{\circ}_{\widetilde{w}} \ \boxplus \ \pi^{\circ}_{\widetilde{w}} \ \boxplus \ \ldots \ \boxplus \ \pi^{\circ}_{\widetilde{w}}$}}_{\mbox{{\small $n$}}}
\]
is in $\mathcal{D}^{\circ}$ and hence also in $\mathcal{C}^{\circ}_{w}$. But then by Theorem \ref{pindec} we deduce that $\widetilde{w}$ occurs as a pin factor in $n$ non-overlapping instances in $w$. As this holds for all $n$, it follows that $\widetilde{w}$ is a recurrent pin factor of $w$, and so $\pi^{\circ}$ is a recurrent $\boxplus$-indecomposable of $\mathcal{C}^{\circ}_{w}$, and is therefore contained in $\mathcal{C}^{\boxplus}_{w}$. As $\mathcal{D}^{\circ}$ is therefore a $\boxplus$-closed class all of whose $\boxplus$-indecomposables are contained in the $\boxplus$-closed class $\mathcal{C}^{\boxplus}_{w}$, it follows that
\[
\mathcal{D}^{\circ} \subseteq \mathcal{C}^{\boxplus}_{w} \ .
\]
Hence $\mathcal{C}^{\boxplus}_{w}$ is a $\boxplus$-closed subclass of $\mathcal{C}^{\circ}_{w}$ which itself contains every $\boxplus$-closed subclass of $\mathcal{C}^{\circ}_{w}$ as a subclass, and the result follows.
\item Immediate from part \ref{boxprops:2}.
\item Suppose $\sigma^{\circ} \in \mathcal{C}^{\boxplus}_{w}$. Then $\sigma^{\circ}$ can be written in the form
\[
\sigma^{\circ} = \pi^{\circ}_{w_{1}} \ \boxplus \ \pi^{\circ}_{w_{2}} \ \boxplus \ \dots \ \boxplus \ \pi^{\circ}_{w_{n}},
\]
where each $w_{i}$ is a recurrent pin factor of $w$. As in part \ref{boxprops:1} we know that we can find the sequence of pin factors $w_{1}, w_{2}, \dots, w_{n}$ at some point in $w$ in that order, in non-overlapping instances and separated from each other by at least one letter in $w$. In fact, as all of these factors are recurrent, this happens infinitely-often, so we can choose such a sequence beginning after any fixed position $i$. Taking $j$ to be the position at the end of $w_{n}$, we see that $\sigma^{\circ}$ is contained in $w_{i,j}$ by the pin decomposition. For the converse, suppose that $\sigma^{\circ}$ is contained in some $w_{i,j}$ for all $i \in \mathbb{N}$. It follows that every $\boxplus$-indecomposable in the $\boxplus$-decomposition of $\sigma^{\circ}$ is also contained in some $w_{i,j}$ for all $i$, and is therefore a recurrent $\boxplus$-indecomposable of $\mathcal{C}^{\circ}_{w}$, and so $\sigma^{\circ}$ is contained in $\mathcal{C}^{\boxplus}_{w}$. \qedhere
\end{enumerate}
\end{proof}

Informally, the $\boxplus$-interior of $\mathcal{C}^{\circ}_{w}$ is the set of all centred permutations that can be found in the pin diagram of $w$ in infinitely many instances.

As the $\boxplus$-interior of a pin class is $\boxplus$-closed, we can determine its generating function in terms of its $G$-sequence, which we can obtain from $w$ as follows:

\begin{lemma} \label{lemma:boxgenfunction}
Suppose that $g(z)$ is the generating function of the recurrent $\boxplus$-indecomposables in $\mathcal{C}^{\circ}_{w}$ and $g_{1}(z), g_{2}(z), g_{3}(z), g_{4}(z)$ are the generating functions of the recurrent one-quadrant $\boxplus$-indecomposables in quadrants $1,2,3,4$, respectively. Then
\[
G(z) = g(z) - g_{1}(z)g_{3}(z) - g_{2}(z)g_{4}(z)
\]
is the amended $G$-sequence for $\mathcal{C}^{\boxplus}_{w}$, and 
\[
f(z) = \frac{1}{1 - G(z)}
\]
is the generating function of $\mathcal{C}^{\boxplus}_{w}$.
\end{lemma}

\begin{proof}
Immediate from the Generating Function Specification \ref{genfuncspec} on noting that $\mathcal{C}^{\boxplus}_{w}$ is a $\boxplus$-closed class and the recurrent $\boxplus$-indecomposables of $\mathcal{C}^{\circ}_{w}$ are its $\boxplus$-indecomposables. \end{proof}

As with the $\boxplus$-closure we briefly note the following:

\begin{lemma}
Let $w$ be an infinite pin word. Then $\mathcal{C}^{\boxplus}_{w}$ has a proper growth rate.
\end{lemma}

\begin{proof}
If $\mathcal{C}^{\boxplus}_{w}$ contains $(\nept)$ and $(\swpt)$ then $w$ contains $1$ and $3$ as recurrent pin factors, which is to say that $w$ visits quadrants $1$ and $3$ infinitely-often. Every time $w$ moves from quadrant $1$ to quadrant $3$ it must pass through either quadrant $2$ or $4$, and so must also visit at least one of these quadrants infinitely-often; hence $\mathcal{C}^{\boxplus}_{w}$ also contains either $(\nwpt)$ or $(\sept)$ and so $w$ satisfies the adjacency condition. The same reasoning works if $\mathcal{C}^{\boxplus}_{w}$ contains $(\nwpt)$ or $(\sept)$, and so we conclude that $\mathcal{C}^{\boxplus}_{w}$ satisfies the adjacency condition. As $\mathcal{C}^{\boxplus}_{w}$ is $\boxplus$-closed we can now apply Theorem \ref{grexist} to deduce that $\mathcal{C}^{\boxplus}_{w}$ has a proper growth rate.
\end{proof}

Recall that $\kappa \approx 2.20557$ is the growth rate of the class $\mathcal{O}^{\circ}$ of increasing oscillations. We shall require the following elementary bound:

\begin{lemma} \label{lemma:kappabound}
Let $w$ be an infinite pin word. Then $gr(\mathcal{C}^{\boxplus}_{w}) \geq \kappa$.
\end{lemma}

\begin{proof}

If $w$ visits only one quadrant recurrently, which by symmetry we may take to be the first quadrant, then $w=\widetilde{w}\overline{(ur)}$, where $\widetilde{w}$ is some finite prefix. Hence $\mathcal{C}^{\boxplus}_{w}$ is just $\mathcal{O}^{\circ}$ with a finite prefix, and so has growth rate $\kappa$ by the Finite Prefix Lemma \ref{fplem}.

Next, suppose $w$ visits precisely two (necessarily adjacent) quadrants recurrently: by symmetry we may take these to be quadrants $1$ and $2$. Then after some finite prefix, $w$ stays in the upper half-plane, and moves between quadrants $1$ and $2$ infinitely-often. Hence $w$ contains $1l$ and $2r$ as recurrent pin factors, and these must extend to recurrent pin factors $1lu$ and $2ru$ (as $w$ stays in the upper half-plane from this point). Hence $\pi^{\circ}_{1lu}$ and $\pi^{\circ}_{2ru}$ are contained in $\mathcal{C}^{\boxplus}_{w}$ and we note that $(\nwtwo)$ and $(\netwo)$ are contained in these two pin permutations, respectively. Hence $\mathcal{C}^{\boxplus}_{w}$ contains the $\boxplus$-closure $\boxplus\{(\nwtwo), (\netwo)\}$, which, by a calculation on similar lines to those given in Section \ref{sec:2.5}, has growth rate approximately $2.73205 > \kappa$.

Finally, suppose $w$ visits three or four quadrants recurrently --- by symmetry assume these include quadrants $1$, $2$ and $3$. Then $\mathcal{C}^{\boxplus}_{w}$ contains $\boxplus\{(\nept), (\nwpt), (\swpt)\}$, which by calculation has growth rate approximately $2.61803 > \kappa$. \qedhere
\end{proof}

We note that this now implies the following as a corollary:

\begin{cor} \label{cor:1quadbound}
Let $w$ be an infinite pin word. Then:
\begin{itemize}
\item $\overline{gr}(\mathcal{C}_{w}) \geq \kappa$
\item $\overline{gr}(\mathcal{C}_{w}) = \kappa$ if and only if $w = \widetilde{w}\overline{(ur)}$ or some symmetry of this form, where $\widetilde{w}$ is a finite prefix.
\end{itemize}
\end{cor}

\begin{proof}
Immediate from the proof of Lemma \ref{lemma:kappabound} as
\[
\mathcal{C}^{\boxplus}_{w} \subseteq \mathcal{C}^{\circ}_{w} \ ,
\]
and hence
\[
gr(\mathcal{C}^{\boxplus}_{w}) \leq \overline{gr}(\mathcal{C}^{\circ}_{w}) \ .
\]
\end{proof}

\subsection{$G$-sequence Properties} \label{sec:4.2}

Let $w$ be an infinite pin word, and let $\mathcal{C}^{\circ}$ be \emph{either} the pin class generated by $w$ \emph{or} its $\boxplus$-interior. We now know how to obtain the amended $G$-sequence $G(z)$ of $\mathcal{C}^{\circ}$, by enumerating either the pin factors or recurrent pin factors of $w$. Then the generating function of $\boxplus\left(\mathcal{C}^{\circ}\right)$ (note that this is $\mathcal{C}^{\circ}$ itself in the $\boxplus$-interior case) is given by:
\[
f(z) = \frac{1}{1 - G(z)} \ .
\]
We know that the growth rate $\rho$ of $\boxplus\left(\mathcal{C}^{\circ}\right)$ is equal to the reciprocal of the radius of convergence of $f(z)$. We should like to be able to deduce from this that $G(\rho^{-1}) = 1$, as this will give us a strategy for studying growth rates of pin classes in terms of analytic properties of $G(z)$, but this will require a slightly more thorough study of the properties of $G(z)$:

\begin{prop}[$G$-sequence of pin classes and $\boxplus$-interiors] \label{Gproperties}
Let $w$ be an infinite pin word and suppose that $G(z)= \sum_{n=1}^{\infty}a_{n}z^{n}$ ($a_{n} \in \mathbb{Z}$) is the amended $G$-sequence of a class $\mathcal{C}^{\circ}$, which is either the pin class generated by $w$ or its $\boxplus$-interior. Then:

\begin{enumerate}
\item $a_{1} \in \left\{1, 2, 3, 4\right\}$. \label{Gproperties1}
\item For all $n \geq 2$: \label{Gproperties2}
\[
-8n \leq a_{n} < 2^{n+2}
\]
\item $G(z)$ converges to a smooth function on the interval $\left[0,\frac{1}{2}\right)$, with $G(0) = 0$. \label{Gproperties3}
\item $G(z) = 1$ has a solution in $\left[0,\frac{1}{\kappa}\right]$. \label{Gproperties4}
\item Let $\alpha$ be the smallest positive real solution of $G(z)=1$. The growth rate of $\boxplus\left(\mathcal{C}^{\circ}\right)$ is equal to $\alpha^{-1}$. \label{Gproperties5}
\end{enumerate}

\end{prop}

\begin{proof}

Recall that 

\begin{equation} \label{amend}
G(z) = g(z) - g_{1}(z)g_{3}(z) - g_{2}(z)g_{4}(z)
\end{equation}

where $g(z)$ counts the $\boxplus$-indecomposables in $\mathcal{C}^{\circ}$ and $g_{i}(z)$ counts the one-quadrant $\boxplus$-indecomposables contained in the $i$th quadrant.

\begin{enumerate}
\item This is clear from (\ref{amend}): the $z$-term of $G(z)$ must agree with the $z$-term of $g(z)$ (as the products $g_{1}(z)g_{3}(z)$ and $g_{2}(z)g_{4}(z)$ have no term of degree lower than $z^2$), and this simply counts the number of quadrants that $\mathcal{C}^{\circ}$ visits (either at all or recurrently), which is clearly either $1$, $2$, $3$ or $4$.

\item Let $\preceq$ denote coefficient-wise ordering on formal power series, so $p(z) \preceq q(z)$ means that the $z^n$-coefficient of $p(z)$ is less than or equal to the $z^n$-coefficient of $q(z)$ for all $n \in \mathbb{N}$. Then, from (\ref{amend}) (and using the fact that the coefficients of $g(z),g_{i}(z)$ are all non-negative):
\[
\begin{split}
G(z) & = g(z) - g_{1}(z)g_{3}(z) - g_{2}(z)g_{4}(z) \\
 & \preceq g(z) \\
 & \preceq 4z + \sum_{n=2}^{\infty}2^{n+2}z^n,
\end{split}
\]
where the final inequality is simply the number of pin words of length $n$ in $\mathcal{L}_{P}$.

On the other hand, using the fact that the generating function of \emph{all} one-quadrant $\boxplus$-indecomposable pin permutations in the $i$th quadrant is
\[
g_{i}(z) = z + z^2 + 2z^3 + 2z^4 + 2z^5 + \dots = \frac{z + z^3}{1-z},
\]
we obtain:
\[
\begin{split}
G(z) & = g(z) - g_{1}(z)g_{3}(z) - g_{2}(z)g_{4}(z) \\
 & \succeq - g_{1}(z)g_{3}(z) - g_{2}(z)g_{4}(z) \\
 & \succeq -2\frac{(z+z^3)^{2}}{(1-z)^{2}} \\
 & = -2z^2 - 4z^3 - 10z^4 - 16z^5 - 24z^6 - 32z^7 - 40z^8 - 48z^9 - \dots \\
 & \succeq \sum_{n=1}^{\infty}-8nz^n.
\end{split}
\]
Combining these two inequalities gives the desired result.

\item This is from the Ratio Test, on noting that the coefficients of $G(z)$ have magnitudes bounded by powers of $2$.

\item We know that all the classes we are dealing with have upper growth rates greater than or equal to $\kappa$. Hence the generating function of $\boxplus\left(\mathcal{C}^{\circ}\right)$, namely
\[
f(z) = \frac{1}{1 - G(z)} ,
\]
has radius of convergence (in $\mathbb{C}$) $R$, where $R \leq \frac{1}{\kappa} < \frac{1}{2}$. As $f(z)$ has non-negative coefficients (regardless of whether $G(z)$ does), we can apply Pringsheim's Theorem~\cite[Theorem IV.7]{flajolet:analytic-combin:} to deduce that $f(z)$ has a singularity at $z = R$. Singularites of $f(z)$ correspond either to singularities of $G(z)$ or solutions of $G(z)=1$. But $G(z)$ converges on $\left[0,\frac{1}{2}\right)$ which includes $R$, so $G(R)=1$.

\item We saw that $R$ solves $G(z)=1$ in the proof of 4.; there cannot be a smaller solution to this equation as that would give a singularity with magnitude smaller than $R$. Now we may apply the Exponential Growth Theorem \ref{EGT}. \qedhere
\end{enumerate}
\end{proof}

We require one more property of $G$-sequences in the $\boxplus$-interior case only:

\begin{lemma} \label{G+}
Let $G(z)$ be the amended $G$-sequence of the $\boxplus$-interior $\mathcal{C}^{\boxplus}_{w}$ of some pin class $\mathcal{C}^{\circ}_{w}$. Let $\alpha$ be the smallest positive real root of the equation $G(z)=1$, guaranteed to exist by Proposition \ref{Gproperties}. Then $G(z)$ is positive on the interval $(0,\alpha)$.
\end{lemma}

\begin{proof}
We note that this property is obvious if $G(z)$ has non-negative coefficients, which immediately deals with the case in which $w$ visits only two of the quadrants (either recurrently or otherwise).

We next suppose that $\mathcal{C}^{\boxplus}_{w}$ has points in precisely three quadrants, which, without loss of generality, we take to be quadrants $1$, $2$ and $3$. Then $\mathcal{C}^{\boxplus}_{w}$ must contain the permutations $(\nept)$, $(\nwpt)$, $(\swpt)$, $(\netwo)$ and $(\swtwo)$ (as the defining infinite pin word $w$ must turn around in the first and third quadrants infinitely-often), and so
\[
3z + 2z^2 + 2z^3 \preceq g(z),
\]
where again we write $h_{1}(z) \preceq h_{2}(z)$ to mean that, for all $n \in \mathbb{N}$, the coefficient of $z^n$ in $h_{2}(z)$ is at least as large as the coefficient of $z^n$ in $h_{1}(z)$. At the other end of the spectrum, for each $i \in \left\{1,3\right\}$,
\[
\begin{split}
g_{i}(z) & \preceq z + z^2 + 2z^3 + 2z^4 + \dots \\
 & = \frac{z + z^3}{1-z} \ ,
\end{split}
\]
again, as this is the generating function of all one-quadrant pin permutations in the $i$th quadrant (namely, the one-quadrant oscillations).

Combining these two facts, we obtain
\begin{equation}
\begin{split}
G(z) & = g(z) - g_{1}(z)g_{3}(z) \\
 & \succeq (3z + 2z^2 + 2z^3) - \frac{(z+z^3)^2}{(1-z)^2} \\
 & = \frac{3z - 5z^2 + z^3 - 4z^4 + 2z^5 - z^6}{(1-z)^2} \label{skjdlkfj879sdkj}
\end{split}
\end{equation}
Hence:
\[
G(z) = \frac{3z - 5z^2 + z^3 - 4z^4 + 2z^5 - z^6}{(1-z)^2} + F(z) \ ,
\]
where $F(z)$ has non-negative coefficients. By direct computation the function (\ref{skjdlkfj879sdkj}) is positive on $(0,\kappa]$, and $F(z)$ is certainly positive here due to the non-negativity of its coefficients. Hence $G(z)$ is positive on $(0,\kappa]$, and as part \ref{Gproperties5} of Proposition \ref{Gproperties} tells us that $\alpha$ is in this interval, $G(z)$ is positive on $(0,\alpha)$.

For the final case, if $\mathcal{C}^{\boxplus}_{w}$ has points in all four quadrants then (as $g(z) \succeq 4z$):
\begin{equation}
\begin{split}
G(z) & = g(z) - g_{1}(z)g_{3}(z) - g_{2}(z)g_{4}(z)\\
 & \succeq 4z - 2\frac{(z+z^3)^2}{(1-z)^2} \\
 & = \frac{4z - 10z^2 + 4z^3 - 4z^4 - 2z^6}{(1-z)^2} \ , \label{sfkld7jsd2}
\end{split}
\end{equation}
and so $G(z)$ can be written in the form
\[
G(z) = \frac{4z - 10z^2 + 4z^3 - 4z^4 - 2z^6}{(1-z)^2} + F(z) \ ,
\]
where $F(z)$ has non-negative coefficients. Unlike in the three-quadrant case, (\ref{sfkld7jsd2}) is \emph{not} positive on $(0,\kappa]$, but it is positive on the shorter interval $(0,\frac{1}{3}]$; hence (by the non-negativity of the coefficients of $F(z)$) $G(z)$ itself is positive on the interval $(0,\frac{1}{3}]$. As $\mathcal{C}^{\boxplus}_{w}$ is a $\boxplus$-closed centred permutation class with points in all four quadrants, it contains the centred $\mathcal{X}$-class (as defined in Section \ref{sec:2.5}), which we demonstrated to have (upper) growth rate $2 + \sqrt{2} \approx 3.41421$ in Example \ref{ex:Xclass}. Hence $\mathcal{C}^{\boxplus}_{w}$ has growth rate greater than $2 + \sqrt{2} > 3$, and so by part \ref{Gproperties4} of Proposition \ref{Gproperties} the interval $(0,\alpha)$ is contained in $[0,\frac{1}{3})$, on which we have just shown $G(z)$ to be positive. \qedhere
 \end{proof}

\subsection{Pin Classes Have Proper Growth Rates} \label{sec:4.3}

The moral of the preceeding sections is that we can sandwich \emph{any} pin class (recurrent or otherwise) between two $\boxplus$-closed (centred) permutation classes that we know how to enumerate. If the pin class happens to be recurrent then these containments are in fact equalities and we have enumerated the pin class. If the pin class is \emph{not} recurrent then we at least have the following bound:

\begin{equation} \label{conts}
gr(\mathcal{C}^{\boxplus}_{w}) \leq \underline{gr}(\mathcal{C}^{\circ}_{w}) \leq \overline{gr}(\mathcal{C}^{\circ}_{w}) \leq gr(\boxplus(\mathcal{C}^{\circ}_{w}))
\end{equation}

The aim of this section is to prove the remarkable result that the two left-most inequalities in (\ref{conts}) are in fact equalities; that is, the upper and lower growth rates of a pin class are always equal to the growth rate of its $\boxplus$-interior. This will enable us to conclude that \emph{any} pin class (recurrent or otherwise) has a proper growth rate, and will enable us to determine this growth rate providing that we can enumerate the recurrent pin factors of the defining infinite pin word $w$. The core idea of the proof is relatively simple to understand: if $w$ is an infinite pin word then we know by the Finite Prefix Lemma \ref{fplem} that, for all $n \in \mathbb{N}$, $\mathcal{C}^{\circ}_{w}$ and $\mathcal{C}^{\circ}_{w_{\geq n}}$ have the same (upper) growth rate, and we may further note that, for any $k \in \mathbb{N}$, if we take $n$ sufficiently large then the pin factors of $w_{\geq n}$ will be precisely the \emph{recurrent} pin factors of $w$ up to length $k$. Thus we expect the growth rate of $\boxplus(\mathcal{C}^{\circ}_{w_{\geq n}})$, clearly an upper bound on $\overline{gr}(\mathcal{C}^{\circ}_{w_{\geq n}}) = \overline{gr}(\mathcal{C}^{\circ}_{w})$, to approach $gr(\mathcal{C}^{\boxplus}_{w})$ as $n \rightarrow \infty$. We formalise this as follows:

\begin{prop} \label{limprop}
Let $\mathcal{C}^{\circ}_{w}$ be a pin class generated by an infinite pin word $w$. Let $\mathcal{C}^{\boxplus}_{w}$ be the $\boxplus$-interior of $\mathcal{C}^{\circ}_{w}$ and write $w_{\geq n}$ for the left-truncation of $w$ starting in the $n$th position. Then:
\[
\lim_{n\rightarrow\infty}gr\left(\boxplus(\mathcal{C}^{\circ}_{w_{\geq n}})\right) = gr(\mathcal{C}^{\boxplus}_{w}).
\]
\end{prop}

\begin{proof}
Let $g(z)$ be the generating function of the $\boxplus$-indecomposables of $\mathcal{C}^{\boxplus}_{w}$, and $g_{i}(z)$ be the generating function of one-quadrant $\boxplus$-indecomposables of $\mathcal{C}^{\boxplus}_{w}$ in the $i$th quadrant. Then
\[
G(z) = g(z) - g_{1}(z)g_{3}(z) - g_{2}(z)g_{4}(z)
\]
is the amended $G$-sequence of $\mathcal{C}^{\boxplus}_{w}$, and $\alpha = gr(\mathcal{C}^{\boxplus}_{w})^{-1}$ is the smallest positive real root of the equation $G(z) = 1$.

Now, let $t \in \mathbb{N}$, and consider the set $\mathcal{W}^{*}_{t}$ of all non-recurrent pin factors of $w$ of length at most $t$. This is a finite set and each element of it occurs as a pin factor of $w$ only a finite number of times, so there is some $n(t) \in \mathbb{N}$ such that no element of $\mathcal{W}^{*}_{t}$ is contained as a pin factor of $w_{\geq n(t)}$ (for concreteness, we may take $n(t)$ to be the \emph{smallest} positive integer with this property). But of course \emph{all} of the recurrent pin factors of $w$ of length at most $t$ occur as pin factors of $w_{\geq n(t)}$, so we deduce that the pin factors of $w_{\geq n(t)}$ of length at most $t$ are precisely the recurrent pin factors of $w$ of length at most $t$. Hence if we write $g_{t}(z)$ for the generating function of the $\boxplus$-indecomposables of $\mathcal{C}^{\circ}_{w_{\geq n(t)}}$, and $g_{t, i}(z)$ for generating function of the one-quadrant $\boxplus$-indecomposables of $\mathcal{C}^{\circ}_{w_{\geq n(t)}}$ in the $i$th quadrant, then $g_{t}(z), g_{t,1}(z), g_{t,2}(z), g_{t,3}(z), g_{t,4}(z)$ will agree with $g(z), g_{1}(z), g_{2}(z), g_{3}(z), g_{4}(z)$, respectively, up to and including the $z^{t}$-term. Hence the amended $G$-sequence of $\mathcal{C}^{\circ}_{w_{\geq n(t)}}$, namely
\[
G_{t}(z) = g_{t}(z) - g_{t,1}(z)g_{t,3}(z) - g_{t,2}(z)g_{t,4}(z) ,
\]
agrees with $G(z)$ up to and including the $z^{t}$-term.

Fix a value $t$: we now consider some basic analytic facts about the functions $G(z)$ and $G_{t}(z)$. First, note that by Proposition \ref{Gproperties}, $G(z)$ and $G_{t}(z)$ are smooth functions defined on the interval $\left[0,\frac{1}{\kappa}\right]$. Further, $G(z) = 1$ and $G_{t}(z) = 1$ have solutions in this interval; we call the smallest solution to these equations in this interval $\alpha$ and $\alpha_{t}$, respectively. Then $\mathcal{C}^{\boxplus}_{w}$ and $\boxplus\mathcal{C}^{\circ}_{w_{\geq n(t)}}$ have growth rates $\rho = \alpha^{-1}$ and $\rho_{t} = \alpha^{-1}_{t}$, respectively. Note that we have:
\[
\rho = gr(\mathcal{C}^{\boxplus}_{w}) \leq \overline{gr}(\mathcal{C}^{\circ}_{w}) = \overline{gr}(\mathcal{C}^{\circ}_{w_{\geq n(t)}}) \leq gr(\boxplus\mathcal{C}^{\circ}_{w_{\geq n(t)}}) = \rho_{t}
\]
where the two inequalities follow from containment of the corresponding classes, and the middle equality follows from the Finite Prefix Lemma \ref{fplem}. Note that this implies that $\alpha_{t} \leq \alpha$. 

Now, write $G_{t}(z) = \sum_{n=1}^{\infty}a_{n}z^n$ and $G(z) = \sum_{n=1}^{\infty}b_{n}z^n$: then each $a_{n}, b_{n}$ is an integer (not necessarily positive) with magnitude bounded by $2^{n+2}$ and $a_{n} = b_{n}$ for all $n \leq t$. We combine these facts to deduce a bound on the difference between these two functions:
\[
\begin{split}
\left|G_{t}(z) - G(z)\right| & = \left|\sum_{n=t+1}^{\infty}(a_{n} - b_{n})z^{n}\right| \\
 & \leq \sum_{n=t+1}^{\infty}\left|a_{n} - b_{n}\right|z^n \\
 & \leq \sum_{n=t+1}^{\infty}(2^{n+2} + 8n)z^{n} \\
 & \leq \sum_{n=t+1}^{\infty}2^{n+3}z^{n} \\
 & = \frac{8(2z)^{t+1}}{1-2z} \ ,
\end{split}
\]
and so, for $z \in \left[0,\frac{1}{\kappa}\right]$:
\[
\begin{split}
\left|G_{t}(z) - G(z)\right| & \leq \sup_{z \in \left[0,\frac{1}{\kappa}\right]}\left\{\frac{8(2z)^{t+1}}{1-2z}\right\} \\
 & = \frac{8(\frac{2}{\kappa})^{t+1}}{1 - \frac{2}{\kappa}} \\
 & = \frac{16}{\kappa - 2} \cdot \left(\frac{2}{\kappa}\right)^{t}.
\end{split}
\]
Note, crucially, that (as $\kappa > 2$) this expression approaches $0$ as $t \rightarrow \infty$. We claim that this fact implies that $\alpha_{t} \rightarrow \alpha$:

Let $\epsilon > 0$ and consider $G(z)$ on the interval $\left[0, \alpha - \epsilon\right]$. As a continuous function on a closed interval, $G(z)$ achieves a maximum value $M$ on $\left[0, \alpha - \epsilon\right]$. Further, as $G(z)$ is positive on $(0,\alpha]$ (by Lemma \ref{G+}) and $G(z)=1$ does not have a root in $\left[0, \alpha - \epsilon\right]$ (as $\alpha$ is, by definition, the \emph{smallest} positive real root of $G(z)=1$), $M$ must be a positive number smaller than $1$. Now, choose a $K \in \mathbb{N}$ such that:
\[
\frac{16}{\kappa - 2} \cdot \left(\frac{2}{\kappa}\right)^{K} < 1 - M
\]
and take any $t \geq K$. Then, for $z \in \left[0, \alpha - \epsilon\right]$:
\[
\begin{split}
\left|G_{t}(z)\right| & = \left|(G_{t}(z) - G(z)) + G(z)\right| \\
 & \leq \left|G_{t}(z) - G(z)\right| + \left|G(z)\right| \\
 & \leq \left|G_{t}(z) - G(z)\right| + G(z) \\
 & \leq \frac{16}{\kappa - 2} \cdot \left(\frac{2}{\kappa}\right)^{t} + M \\
 & < (1 - M) + M \\
 & = 1.
\end{split}
\]
Hence, in particular, $G_{t}(z)=1$ does not have a root in $\left[0, \alpha - \epsilon\right]$. But $G_{t}(z)=1$ certainly does have a root, namely $\alpha_{t}$, which is smaller than $\alpha$. Hence $\alpha_{t} \in \left(\alpha - \epsilon, \alpha\right]$.

We have thus proved that for every $\epsilon > 0$ there exists $K \in \mathbb{N}$ such that $\alpha_{t} \in \left(\alpha - \epsilon, \alpha\right]$ for all $t \geq K$; hence $\alpha_{t} \rightarrow \alpha$ as $t \rightarrow \infty$. Thus
\[
gr(\boxplus\mathcal{C}^{\circ}_{\geq n(t)}) \rightarrow gr(\mathcal{C}^{\boxplus}_{w})
\]
as $t\rightarrow\infty$, as required. \qedhere \end{proof}

We may now finally deduce the main result of this paper:

\begin{thm} \label{greqbox}
Let $w$ be an infinite pin word. Then the associated pin class $\mathcal{C}^{\circ}_{w}$ (along with its uncentred counterpart $\mathcal{C}_{w}$) has a proper growth rate which is equal to that of its $\boxplus$-interior, $\mathcal{C}^{\boxplus}_{w}$.
\end{thm}

\begin{proof}
Suppose $w$ is an infinite pin word. Then $\mathcal{C}^{\boxplus}_{w} \subseteq \mathcal{C}^{\circ}_{w}$ and so
\begin{equation}\label{one}
gr(\mathcal{C}^{\boxplus}_{w}) =  \underline{gr}(\mathcal{C}^{\boxplus}_{w}) \leq \underline{gr}(\mathcal{C}^{\circ}_{w})
\end{equation}
where the left-hand equality is due to the fact that $\boxplus$-interiors have (proper) growth rates. Now, letting $n \in \mathbb{N}$ we have
\begin{equation}\label{two}
\overline{gr}(\mathcal{C}^{\circ}_{w}) = \overline{gr}(\mathcal{C}^{\circ}_{w_{\geq n}})
\end{equation}
by the Finite Prefix Lemma \ref{fplem}, and by containment we also have
\begin{equation} \label{three}
\overline{gr}(\mathcal{C}^{\circ}_{w_{\geq n}}) \leq \overline{gr}(\boxplus(\mathcal{C}^{\circ}_{w_{\geq n}})) = gr(\boxplus(\mathcal{C}^{\circ}_{w_{\geq n}}))
\end{equation}
where the right-hand equality follows from existence of proper growth rates of $\boxplus$-closures of pin classes.

Combining equations (\ref{one}), (\ref{two}) and (\ref{three}) yields:
\[
gr(\mathcal{C}^{\boxplus}_{w}) \leq \underline{gr}(\mathcal{C}^{\circ}_{w}) \leq \overline{gr}(\mathcal{C}^{\circ}_{w}) \leq gr(\boxplus\mathcal{C}^{\circ}_{w_{\geq n}})
\]
and taking limits (using Proposition \ref{limprop}) as $n \rightarrow \infty$ forces
\[
\underline{gr}(\mathcal{C}^{\circ}_{w}) = \overline{gr}(\mathcal{C}^{\circ}_{w}) = gr(\mathcal{C}^{\boxplus}_{w}),
\]
as required. Finally, Proposition \ref{eqgrs} demonstrates that the uncentred pin class $\mathcal{C}_{w}$ also has a proper growth rate, equal to this common value.\end{proof}

We have now established that every pin permutation class has a proper growth rate, equal to that of the $\boxplus$-interior of its centred counterpart. Combining this fact with Lemmas \ref{pcupperbound} and \ref{lemma:kappabound} we obtain the following bounds on this characteristic:

\begin{cor}[Pin class growth rate bounds] \label{pcgrbounds}
Suppose $w$ is an infinite pin word. Then $gr(\mathcal{C}_{w})$ exists and
\[
\kappa \leq gr(\mathcal{C}_{w}) \leq \omega_{\infty}.
\]
\end{cor}

\subsection{Example: The Liouville $\mathcal{V}$} \label{sec:4.4}

We conclude by illustrating the theory developed throughout this paper to determine the growth rate of a single example of a non-recurrent pin class. We shall consider one of the simplest examples of a pin class generated by a not-eventually-recurrent infinite pin word. Recall that for a word $f$ over any alphabet we write $(f)^{k}$ to mean $\underbrace{\mbox{fff \dots ff}}_{\mbox{{\small $k$}}}$. We now define:

\begin{defn}[The Liouville $\mathcal{V}$]
Let $w_{\mathcal{L}}$ denote the infinite pin word
\[
\begin{split}
w_{\mathcal{L}} & = 1(ul)^{1}ur(ul)^{2}ur(ul)^{3}ur(ul)^{4}ur(ul)^{5}ur(ul)^{6}ur(ul)^{7} \dots \\
 & = 1ulurululurulululurululululurulul \dots
\end{split}
\]

We write
\[
\mathcal{V}_{\mathcal{L}} = \mathcal{C}_{w_{\mathcal{L}}},
\]
and refer to the permutation class $\mathcal{V}_{\mathcal{L}}$ as the \emph{Liouville $\mathcal{V}$}\footnote{Named by analogy with Liouville's constant, whose decimal expansion consists of $1$'s spaced out by increasing strings of $0$'s.}. Similarly, we define the centred Liouville $\mathcal{V}$, $\mathcal{V}^{\circ}_{\mathcal{L}}$, to be $\mathcal{C}^{\circ}_{w_{\mathcal{L}}}$. See Figure \ref{fig:Liouville} for an illustration.
\end{defn}

The infinite pin word $w_{\mathcal{L}}$ is most certainly not recurrent: the pin word $1ulur$, for example, appears once as a pin factor at the beginning of the sequence and never again. In fact, $w_{\mathcal{L}}$ is not even eventually-recurrent: as the length of the string of between any two consecutive $r$'s is strictly increasing, any pin factor of $w_{\mathcal{L}}$ which contains more than one $r$ appears once and never again.

\begin{figure}[h]
\begin{center}
\reflectbox{\begin{tikzpicture}[scale=0.35]

\node[circle, draw, fill=none, inner sep=0pt, minimum width=\plotptradius] (0) at (9,0) {};
\node[permpt] (1) at (7,1) {};
\node[permpt] (2) at (8,3) {}; \draw[thin] (2) -- ++ (0,-2.5);
\node[permpt] (3) at (11,2) {}; \draw[thin] (3) -- ++ (-3.5,0);
\node[permpt] (4) at (10,5) {}; \draw[thin] (4) -- ++ (0,-3.5);
\node[permpt] (5) at (5,4) {};  \draw[thin] (5) -- ++ (5.5,0);
\node[permpt] (6) at (6,7) {}; \draw[thin] (6) -- ++ (0,-3.5);
\node[permpt] (7) at (13,6) {}; \draw[thin] (7) -- ++ (-7.5,0);
\node[permpt] (8) at (12,9) {}; \draw[thin] (8) -- ++ (0,-3.5);
\node[permpt] (9) at (15,8) {}; \draw[thin] (9) -- ++ (-3.5,0);
\node[permpt] (10) at (14,11) {}; \draw[thin] (10) -- ++ (0,-3.5);
\node[permpt] (11) at (3,10) {}; \draw[thin] (11) -- ++ (11.5,0);
\node[permpt] (12) at (4,13) {}; \draw[thin] (12) -- ++ (0,-3.5);
\node[permpt] (13) at (17,12) {}; \draw[thin] (13) -- ++ (-13.5,0);
\node[permpt] (14) at (16,15) {}; \draw[thin] (14) -- ++ (0,-3.5);
\node[permpt] (15) at (19,14) {}; \draw[thin] (15) -- ++ (-3.5,0);
\node[permpt] (16) at (18,17) {}; \draw[thin] (16) -- ++ (0,-3.5);
\node[permpt] (17) at (21,16) {}; \draw[thin] (17) -- ++ (-3.5,0);
\node[permpt] (18) at (20,19) {}; \draw[thin] (18) -- ++ (0,-3.5);
\node[permpt] (19) at (1,18) {}; \draw[thin] (19) -- ++ (19.5,0);
\node[permpt] (20) at (2,21) {}; \draw[thin] (20) -- ++ (0,-3.5);

%\node[circle,fill,inner sep=0.5pt] (8) at (5.5,4.5) {};

\draw[thick] (9,0) -- ++ (0,22);
\draw[thick] (0,0) -- ++ (22,0);

%\draw[decorate,decoration={brace,amplitude=10},-] (5.5,0.5) -- (0.5,4.5);
%\node[] () at (2,1.25) {\small{2k}};

%\draw[decorate,decoration={brace,amplitude=10},-] (15.5,10.5) -- (10.5,5.5);
%\node[] () at (14.25,6.75) {\small{2l}};

\draw[thin,dashed] (22,20) -- ++ (-20.5,0);

%\node[circle,fill,inner sep=0.5pt] () at (0,13.5) {};
%\node[circle,fill,inner sep=0.5pt] () at (-0.5,14) {};
%\node[circle,fill,inner sep=0.5pt] () at (-1,14.5) {};
%\node[circle,fill,inner sep=0.5pt] () at (-1.5,15) {};
%\node[circle,fill,inner sep=0.5pt] () at (4.5,2.5) {};

%\draw[dotted] (8.5,7.5) rectangle (9.5,8.5);

%\node[] (13) at (6,-1) {$1dl(dl)^{*}dlu$};

\end{tikzpicture}}
\end{center}
\caption{The Liouville $\mathcal{V}$; a non-recurrent pin class.}
\label{fig:Liouville}
\end{figure}

As a consequence of this fact, $\mathcal{V}^{\circ}_{\mathcal{L}}$ is not $\boxplus$-closed. For example, $\pi^{\circ}_{1lur}$ and $\pi^{\circ}_{1lululur}$ are both in $\mathcal{V}^{\circ}_{\mathcal{L}}$ but $\pi^{\circ}_{1lululur} \boxplus \pi^{\circ}_{1lur}$ is not: though $1lur$ and $1lululur$ are both pin factors of $w_{\mathcal{L}}$ they each occur only once and in that order. We will therefore not be able to follow Procedure \ref{recproc} to find the generating function of $\mathcal{V}^{\circ}_{\mathcal{L}}$, though we can nevertheless use Lemma \ref{lemma:boxgenfunction} to enumerate its $\boxplus$-interior, and hence determine its growth rate by Theorem \ref{greqbox}. Somewhat surprisingly, we have already encountered the growth rate of $\mathcal{V}_{\mathcal{L}}$: it is the constant $\nu_{\mathcal{L}}$, defined as in Section \ref{sec:3.7} as the limit of the sequence of constants $(\nu_{1,k})$ as $k \rightarrow \infty$.

\begin{prop}[$\mathcal{V}^{\boxplus}_{\mathcal{L}}$ enumeration] \label{prop:VLcentenum}
The $\boxplus$-interior of the (centred) Liouville $\mathcal{V}$, $\mathcal{V}^{\boxplus}_{\mathcal{L}}$, has generating function
\begin{equation} \label{kkio0}
f(z) = \frac{(1-z)^2}{p_{\mathcal{L}}(z)},
\end{equation}
where $p_{\mathcal{L}}(z) = 1 - 4z + 3z^2 - 2z^3 - z^4 + 2z^5$ as in Definition \ref{def:nul}.

Hence $\mathcal{V}_{\mathcal{L}}$ has growth rate $\nu_{\mathcal{L}} \approx 3.28277$.
\end{prop}

\begin{proof}
We aim to find the generating function of the $\boxplus$-interior of $\mathcal{V}^{\circ}_{\mathcal{L}}$, which by definition is the $\boxplus$-closure of the recurrent $\boxplus$-indecomposables of $\mathcal{V}^{\circ}_{\mathcal{L}}$, all of which are of the form $\pi^{\circ}_{\widetilde{w}}$ where $\widetilde{w}$ is a recurrent pin factor of $w_{\mathcal{L}}$. We thus begin by enumerating the recurrent pin factors of $w_{\mathcal{L}}$. It will again be useful to write out $w_{\mathcal{L}}$ with each letter subscripted by quadrant number:
\[
w_{\mathcal{L}} = 1u_{1}l_{2}u_{2}r_{1}u_{1}l_{2}u_{2}l_{2}u_{2}r_{1}u_{1}l_{2}u_{2}l_{2}u_{2}l_{2}u_{2}r_{1}u_{1}l_{2}u_{2}l_{2}u_{2}l_{2}u_{2}l_{2}u_{2}r_{1}u_{1}l_{2}u_{2}l_{2} \dots
\]
As the gap between consecutive $r_{1}u_{1}$-blocks increases each time, no recurrent pin factor of $w_{\mathcal{L}}$ can stretch between two of these blocks. Conversely, any pin factor of $w_{\mathcal{L}}$ which only contains points from (at most) one of the $r_{1}u_{1}$-blocks \emph{will} occur recurrently. Hence the recurrent pin factors of $w_{\mathcal{L}}$ are precisely the pin factors of the `bi-infinite pin word'
\begin{equation} \label{sfdjfl76tgh0}
\dots l_{2}u_{2}l_{2}u_{2}l_{2}u_{2}l_{2}u_{2}l_{2}u_{2} \ r_{1}u_{1} \ l_{2}u_{2}l_{2}u_{2}l_{2}u_{2}l_{2}u_{2}l_{2}u_{2} \dots
\end{equation}
Clearly, (\ref{sfdjfl76tgh0}) has two pin factors of length $1$ (the quadrant numerals $1$ and $2$). At all lengths $n \geq 2$ the pin factors of (\ref{sfdjfl76tgh0}) come in three distinct types. First, there are those beginning with a $1$: as there are only two letters in (\ref{sfdjfl76tgh0}) subscripted by a $1$ there are only two possible starting points for pin factors of this type, both of which give distinct pin factors at all lengths greater than $1$. Hence (\ref{sfdjfl76tgh0}) has precisely two pin factors of length $n$ in this category. Second, we have pin factors of (\ref{sfdjfl76tgh0}) which begin with a $2$ and do not include the single $r$: as these will alternate between $u$ and $l$ after the initial $2$, these are entirely determined by the choice of second letter, hence there are two pin factors of this form at all lengths $n \geq 2$. Finally, we have those pin factors of (\ref{sfdjfl76tgh0}) which begin with a $2$ and do contain the single $r$: this $r$ can be in any of the remaining $n-1$ positions, and this choice will then completely determine the pin word. Hence there are $n-1$ pin factors of (\ref{sfdjfl76tgh0}) of this type. In total, then, (\ref{sfdjfl76tgh0}) contains two pin factors of length $1$ and $n+3$ pin factors of all lengths $n \geq 2$. We have thus obtained the generating function of the recurrent pin factors of $w_{\mathcal{L}}$:
\[
\begin{split}
h(z) & = 2z + 5z^2 + 6z^3 + 7z^4 + 8z^5 + 9z^6 + \dots \\
 & = \frac{2z + z^2 - 2z^3}{(1-z)^2} \ .
\end{split}
\]
Next, we need to remove $\boxplus$-decomposables and collisions from this count. Comparing the list above with the lists given in Theorem \ref{classification}, we see that amongst the recurrent pin factors of $w_{\mathcal{L}}$ there are precisely two $\boxplus$-decomposables of length $2$ ($1l$ and $2r$), one colliding pair at length $2$ ($\left\{2l,2u\right\}$) and two more colliding pairs at length $3$ ($\{1ul,2ru\}$ and $\{2ur,1lu\}$). Hence we subtract $3$ from the count at length $2$ and $2$ at length $3$ to obtain the generating function of the recurrent $\boxplus$-indecomposables in $\mathcal{V}^{\circ}_{\mathcal{L}}$:
\[
\begin{split}
g(z) & = h(z) - 3z^2 - 2z^3 \\
 & = \frac{2z - 2z^2 + 2z^3 + z^4 - 2z^5}{(1-z)^2} \ .
\end{split}
\]
As the $\boxplus$-interior $\mathcal{V}^{\boxplus}_{\mathcal{L}}$ is, by definition, the $\boxplus$-closure of the recurrent $\boxplus$-indecomposables in $\mathcal{V}^{\circ}_{\mathcal{L}}$ we may apply the Generating Function Specification \ref{genfuncspec} (along with the obvious fact that $g_{3}(z) = g_{4}(z) = 0$) to obtain the generating function of $\mathcal{V}^{\boxplus}_{\mathcal{L}}$:
\[
\begin{split}
f(z) & = \frac{1}{1-g(z)} \\
 & = \frac{(1-z)^2}{1 - 4z + 3z^2 - 2z^3 - z^4 + 2z^5} \\
 & = \frac{(1-z)^2}{p_{\mathcal{L}}(z)} \ ,
\end{split}
\]
as required. The growth rate then follows from an application of the Exponential Growth Theorem \ref{EGT}. \qedhere
\end{proof}

Thus we have demonstrated that the constant $\nu_{\mathcal{L}} = \lim_{k \rightarrow \infty}\nu_{1,k}$ is an accumulation point in the set of pin class growth rates. This constant in fact represents an important phase transition in the set of pin class growth rates, a theme which will be taken up in the sequel to this paper~\cite{brignall2024pinclassesiismall}, in which we demonstrate that $\nu_{\mathcal{L}}$ is the smallest growth rate at which uncountably-many pin permutation classes appear.

\section{Concluding Remarks} \label{sec:5}

We have proved that all pin classes have growth rates and established a procedure for their calculation. We have also obtained bounds on the possible growth rate of a pin class in Corollary \ref{pcgrbounds}. A natural further question is what happens within these bounds: what are the possible growth rates of pin classes? We can also ask about bounds on growth rates of pin classes subject to certain characteristics: the number of quadrants visited (recurrently) by a pin class and the length of the longest oscillation contained in $\mathcal{C}^{\circ}_{w}$ are natural characteristics to consider. We can in fact state some answers in the former case already: the pin classes $\mathcal{V}$, $\mathcal{Y}$ and $\mathcal{W}$, which we enumerated in Sections \ref{sec:3.7} and \ref{sec:3.8}, are in fact the smallest pin classes which visit two, three and four quadrants recurrently, respectively. It is also relatively easy to derive an \emph{upper} bound on the growth rates of pin classes in a specified number of quadrants by considering the \emph{complete pin class} in those quadrants. For example, the \emph{complete class in two quadrants}, $\mathcal{V}_{c} = \mathcal{C}_{w_{c}}$ is the pin class generated by an infinite pin word $w_{c}$ that contains \emph{all} pin words in quadrants $1$ and $2$ as pin factors. Without too much difficulty (though we omit the proof here) we can calculate the growth rate of $\mathcal{V}_{c}$ to be $\nu_{c} \approx 3.51205$, where $\nu_{c}$ is the reciprocal of the smallest positive real root of
\[
1 - 2z - 4z^2 - 2z^3 - 8z^4 - 4z^5 = 0.
\]

In the sequel~\cite{brignall2024pinclassesiismall} we take up the question of what happens within this interval $[\nu,\nu_{c}]$: we move towards a classification of the growth rates of two-quadrant pin classes and observe some interesting structures in this set of growth rates. For example, we shall show that, as we have already alluded to, $\nu_{\mathcal{L}} \approx 3.28277$ is the first point at which there are uncountably many distinct pin classes and that $\nu_{\mathcal{L}}$ is in fact an accumulation point in the set of pin class growth rates from both above and below. This has potential consequences for the study of well-quasi-ordered permutation classes because two-quadrant pin class can be used to generate (by a similar construction as shown for the increasing oscillations in Figure \ref{fig:osc0} on page \pageref{fig:osc0}) infinite antichains with relatively small growth rates.

Potential further directions for study include:
\begin{itemize}
\item A systematic study of pin class growth rates in three and four quadrants.
\item A study of `uncentring' pin class generating functions: given the generating function of a centred pin class $\mathcal{C}^{\circ}_{w}$, is there a process by which we can obtain the generating function of the uncentred pin class $\mathcal{C}_{w}$?
\item The possibility of conjecturing a classification of `small' antichains (perhaps taking $\nu_{\mathcal{L}}$ as a cut-off) using pin classes in two quadrants.
\item The question of whether we can explicitly determine the generating function (not merely the growth rate) of \emph{any} not-eventually-recurrent pin class, centred or uncentred, such as the Liouville $\mathcal{V}$.
\end{itemize}

\appendix
\section{Appendix: Classification of $\boxplus$-decomposables and Collisions}
\renewcommand{\thesubsection}{(\Alph{subsection})}

We present as an appendix the proof of Theorem \ref{classification}, that the lists of collisions and $\boxplus$-decomposables given in the statement of that theorem are in fact complete. We begin with the collisions.

\subsection{Collisions Proof}

We can verify that the list of collisions given in the table is complete for lengths $n \leq 5$ by an exhaustive search (which can be done fairly quickly on applying symmetries). We thus aim to prove that the table is complete at lengths $n \geq 6$. We repeat the relevant section of the table for reference.

{
\centering
\begin{tabular}{@{}|>{\centering}m{1.5em}|>{\centering}m{5.5em}|>{\centering\arraybackslash}m{12em}|>{\centering}m{4em}|>{\centering\arraybackslash}m{17.5em}|} \hline
\multicolumn{5}{|c|}{}\\[0.4pt]
\multicolumn{5}{|c|}{\textbf{List of collisions of pin factors:}}\\[6pt] \hline
\multicolumn{2}{|c|}{\textbf{Length}}&\textbf{Representative}&\textbf{Total number}&\textbf{Full List}\\ \hline

%%%%%
\multicolumn{2}{|c|}{$n \geq 6$ \ \textbf{even:}}
&
\begin{tikzpicture}[scale=0.2]

\node[circle, draw, fill=none, inner sep=0pt, minimum width=\plotptradius] (0) at (8,8) {};
\node[permpt] (1) at (9,10) {}; 
\node[permpt] (2) at (6,9) {}; \draw[thin] (2) -- ++ (3.5,0);
\node[permpt] (3) at (7,6) {}; \draw[thin] (3) -- ++ (0,3.5);
\node[permpt] (4) at (4,7) {}; \draw[thin] (4) -- ++ (3.5,0);
\node[permpt] (5) at (5,5) {};  \draw[thin] (5) -- ++ (0,2.5);
\node[permpt] (6) at (1,3) {}; \draw[thin] (6) -- ++ (1.5,0);
\node[permpt] (7) at (2,1) {}; \draw[thin] (7) -- ++ (0,2.5);
\node[permpt] (8) at (11,2) {}; \draw[thin] (8) -- ++ (-9.5,0);
\node[permpt] (9) at (10,11) {}; \draw[thin,dashed] (9) -- ++ (0,-9.5);

\node[empty] (-1) at (8,13) {};

\node[circle,fill,inner sep=0.5pt] (10) at (4.5,5.5) {};
\node[circle,fill,inner sep=0.5pt] (11) at (4,5) {};
\node[circle,fill,inner sep=0.5pt] (12) at (3.5,4.5) {};
\node[circle,fill,inner sep=0.5pt] (13) at (3,4) {};
\node[circle,fill,inner sep=0.5pt] (14) at (2.5,3.5) {};

\draw[thick] (8,0) -- ++ (0,12);
\draw[thick] (0,8) -- ++ (12,0);

\draw[dotted] (9.5,10.5) rectangle (10.5,11.5);

\node[] (15) at (6.5,-1) {$1(ld)^{k}r=2(dl)^{k}dru$};
\end{tikzpicture}
&
8 \ \text{pairs}
&
\small
\[
\begin{split}
\{1(ld)^{k}r, 2(dl)^{k}dru\}, \{1(dl)^{k}u, 4(ld)^{k}lur\}, \\
\{2(dr)^{k}u, 3(rd)^{k}rul\}, \{2(rd)^{k}l, 1(dr)^{k}dlu\}, \\
\{3(ru)^{k}l, 4(ur)^{k}uld\}, \{3(ur)^{k}d, 2(ru)^{k}rdl\}, \\
\{4(ul)^{k}d, 1(lu)^{k}ldr\}, \{4(lu)^{k}r, 3(ul)^{k}urd\}
\end{split}
\]

 \\ \hline

%%%%%
\multicolumn{2}{|c|}{$n \geq 7$ \ \textbf{odd:}}
&
\begin{tikzpicture}[scale=0.2]

\node[circle, draw, fill=none, inner sep=0pt, minimum width=\plotptradius] (0) at (8,6) {};
\node[permpt] (1) at (9,8) {};
\node[permpt] (2) at (6,7) {}; \draw[thin] (2) -- ++ (3.5,0);
\node[permpt] (3) at (7,4) {}; \draw[thin] (3) -- ++ (0,3.5);
\node[permpt] (4) at (5,5) {}; \draw[thin] (4) -- ++ (2.5,0);
\node[permpt] (5) at (3,1) {};  \draw[thin] (5) -- ++ (0,1.5);
\node[permpt] (6) at (1,2) {}; \draw[thin] (6) -- ++ (2.5,0);
\node[permpt] (7) at (2,10) {}; \draw[thin] (7) -- ++ (0,-8.5);
\node[permpt] (8) at (10,9) {}; \draw[thin,dashed] (8) -- ++ (-8.5,0);

\node[empty] (-1) at (8,12) {};

\node[circle,fill,inner sep=0.5pt] (8) at (5.5,4.5) {};
\node[circle,fill,inner sep=0.5pt] (9) at (5,4) {};
\node[circle,fill,inner sep=0.5pt] (10) at (4.5,3.5) {};
\node[circle,fill,inner sep=0.5pt] (11) at (4,3) {};
\node[circle,fill,inner sep=0.5pt] (12) at (3.5,2.5) {};

\draw[thick] (8,0) -- ++ (0,11);
\draw[thick] (0,6) -- ++ (11,0);

\draw[dotted] (9.5,8.5) rectangle (10.5,9.5);

\node[] (13) at (6,-1) {$1(ld)^{k}lu=2(dl)^{k}ur$};

\end{tikzpicture}
&
8 \ \text{pairs}
&
\small
\[
\begin{split}
\{1(ld)^{k}lu, 2(dl)^{k}ur\}, \{1(dl)^{k}dr, 4(ld)^{k}ru\}, \\
\{2(dr)^{k}dl, 3(rd)^{k}lu\}, \{2(rd)^{k}ru, 1(dr)^{k}ul\}, \\
\{3(ru)^{k}rd, 4(ur)^{k}dl\}, \{3(ur)^{k}ul, 2(ru)^{k}ld\}, \\
\{4(ul)^{k}ur, 1(lu)^{k}rd\}, \{4(lu)^{k}ld, 3(ul)^{k}dr\}
\end{split}
\]
 \\ \hline
\end{tabular}\par
}

In order to prove this we first make the observation that in each of the pairs listed the two pin words end in different letters. We call a collision a \emph{minimal collision} if no two pin words in the tuple have the same final letter. Our aim is first to prove that the list above is a complete list of minimal collisions, and then deduce from this that there are no non-minimal collisions.

\begin{thm}[Classification of Minimal Collisions]

Any minimal collision of length $n \geq 6$ is one of the colliding pairs listed in the table above.

\end{thm}

\begin{proof}

Suppose we have a colliding pair $w_{1},w_{2}$ of pin words. Without loss of generality, by symmetry we may assume that one of these pin words ends in $dr$. We therefore first apply symmetries to the list above so that one pin word of each pair ends in $dr$ --- see Figures \ref{fig:oddcollisionsdr} and \ref{fig:evencollisionsdr}.

\begin{figure}[h]
\begin{center}
\begin{tikzpicture}[scale=0.35]

\node[circle, draw, fill=none, inner sep=0pt, minimum width=\plotptradius] (0) at (8,10) {};
\node[permpt] (1) at (10,11) {}; 
\node[permpt] (2) at (9,8) {}; \draw[thin] (2) -- ++ (0,3.5);
\node[permpt] (3) at (6,9) {}; \draw[thin] (3) -- ++ (3.5,0);
\node[permpt] (4) at (7,6) {}; \draw[thin] (4) -- ++ (0,3.5);
\node[permpt] (5) at (5,7) {};  \draw[thin] (5) -- ++ (2.5,0);
\node[permpt] (6) at (3,3) {}; \draw[thin] (6) -- ++ (0,1.5);
\node[permpt] (7) at (1,4) {}; \draw[thin] (7) -- ++ (2.5,0);
\node[permpt] (8) at (2,1) {}; \draw[thin] (8) -- ++ (0,3.5);
\node[permpt] (9) at (12,2) {}; \draw[thin] (9) -- ++ (-10.5,0);
\node[permpt] (10) at (11,12) {}; \draw[thin,dashed] (10) -- ++ (0,-10.5);

\node[circle,fill,inner sep=0.5pt] (11) at (5.5,6.5) {};
\node[circle,fill,inner sep=0.5pt] (12) at (5,6) {};
\node[circle,fill,inner sep=0.5pt] (13) at (4.5,5.5) {};
\node[circle,fill,inner sep=0.5pt] (14) at (4,5) {};
\node[circle,fill,inner sep=0.5pt] (15) at (3.5,4.5) {};

\draw[thick] (8,0) -- ++ (0,13);
\draw[thick] (0,10) -- ++ (13,0);

\draw[dotted] (10.5,11.5) rectangle (11.5,12.5);

\node[] (16) at (6.5,-1) {$1dldl \dots dldr$};
\node[] (17) at (6.65,-2.5) {$=4ldl \dots dldru$};

\begin{scope}[shift={(18,0)}]

\node[circle, draw, fill=none, inner sep=0pt, minimum width=\plotptradius] (0) at (10,5) {};
\node[permpt] (1) at (11,3) {};
\node[permpt] (2) at (8,4) {}; \draw[thin] (2) -- ++ (3.5,0);
\node[permpt] (3) at (9,7) {}; \draw[thin] (3) -- ++ (0,-3.5);
\node[permpt] (4) at (6,6) {}; \draw[thin] (4) -- ++ (3.5,0);
\node[permpt] (5) at (7,8) {};  \draw[thin] (5) -- ++ (0,-2.5);
\node[permpt] (6) at (3,10) {}; \draw[thin] (6) -- ++ (1.5,0);
\node[permpt] (7) at (4,12) {}; \draw[thin] (7) -- ++ (0,-2.5);
\node[permpt] (8) at (1,11) {}; \draw[thin] (8) -- ++ (3.5,0);
\node[permpt] (9) at (2,1) {}; \draw[thin] (9) -- ++ (0,10.5);
\node[permpt] (10) at (12,2) {}; \draw[thin,dashed] (10) -- ++ (-10.5,0);

\node[circle,fill,inner sep=0.5pt] (11) at (6.5,7.5) {};
\node[circle,fill,inner sep=0.5pt] (12) at (6,8) {};
\node[circle,fill,inner sep=0.5pt] (13) at (5.5,8.5) {};
\node[circle,fill,inner sep=0.5pt] (14) at (5,9) {};
\node[circle,fill,inner sep=0.5pt] (15) at (4.5,9.5) {};

\draw[thick] (10,0) -- ++ (0,13);
\draw[thick] (0,5) -- ++ (13,0);

\draw[dotted] (11.5,1.5) rectangle (12.5,2.5);

\node[] (16) at (6.5,-1) {$4lulu \dots luld$};
\node[] (17) at (6.5,-2.5) {$=3ulu \dots luldr$};

\end{scope}

\end{tikzpicture}
\end{center}
\caption{Odd collisions ending in $dr$ of length at least $6$.}
\label{fig:oddcollisionsdr}
\end{figure}

\begin{figure}[h]
\begin{center}
\begin{tikzpicture}[scale=0.35]

\node[circle, draw, fill=none, inner sep=0pt, minimum width=\plotptradius] (0) at (8,8) {};
\node[permpt] (1) at (9,10) {}; 
\node[permpt] (2) at (6,9) {}; \draw[thin] (2) -- ++ (3.5,0);
\node[permpt] (3) at (7,6) {}; \draw[thin] (3) -- ++ (0,3.5);
\node[permpt] (4) at (4,7) {}; \draw[thin] (4) -- ++ (3.5,0);
\node[permpt] (5) at (5,5) {};  \draw[thin] (5) -- ++ (0,2.5);
\node[permpt] (6) at (1,3) {}; \draw[thin] (6) -- ++ (1.5,0);
\node[permpt] (7) at (2,1) {}; \draw[thin] (7) -- ++ (0,2.5);
\node[permpt] (8) at (11,2) {}; \draw[thin] (8) -- ++ (-9.5,0);
\node[permpt] (9) at (10,11) {}; \draw[thin,dashed] (9) -- ++ (0,-9.5);

\node[circle,fill,inner sep=0.5pt] (10) at (4.5,5.5) {};
\node[circle,fill,inner sep=0.5pt] (11) at (4,5) {};
\node[circle,fill,inner sep=0.5pt] (12) at (3.5,4.5) {};
\node[circle,fill,inner sep=0.5pt] (13) at (3,4) {};
\node[circle,fill,inner sep=0.5pt] (14) at (2.5,3.5) {};

\draw[thick] (8,0) -- ++ (0,12);
\draw[thick] (0,8) -- ++ (12,0);

\draw[dotted] (9.5,10.5) rectangle (10.5,11.5);

\node[] (15) at (6.5,-1) {$1ldld \dots ldr$};
\node[] (16) at (6.5,-2.5) {$=2dld \dots ldru$};

\begin{scope}[shift={(18,0)}]

\node[circle, draw, fill=none, inner sep=0pt, minimum width=\plotptradius] (0) at (8,4) {};
\node[permpt] (1) at (10,3) {};
\node[permpt] (2) at (9,6) {}; \draw[thin] (2) -- ++ (0,-3.5);
\node[permpt] (3) at (6,5) {}; \draw[thin] (3) -- ++ (3.5,0);
\node[permpt] (4) at (7,8) {}; \draw[thin] (4) -- ++ (0,-3.5);
\node[permpt] (5) at (5,7) {};  \draw[thin] (5) -- ++ (2.5,0);
\node[permpt] (6) at (3,11) {}; \draw[thin] (6) -- ++ (0,-1.5);
\node[permpt] (7) at (1,10) {}; \draw[thin] (7) -- ++ (2.5,0);
\node[permpt] (8) at (2,1) {}; \draw[thin] (8) -- ++ (0,9.5);
\node[permpt] (9) at (11,2) {}; \draw[thin,dashed] (9) -- ++ (-9.5,0);

\node[circle,fill,inner sep=0.5pt] (10) at (5.5,7.5) {};
\node[circle,fill,inner sep=0.5pt] (11) at (5,8) {};
\node[circle,fill,inner sep=0.5pt] (12) at (4.5,8.5) {};
\node[circle,fill,inner sep=0.5pt] (13) at (4,9) {};
\node[circle,fill,inner sep=0.5pt] (14) at (3.5,9.5) {};

\draw[thick] (8,0) -- ++ (0,12);
\draw[thick] (0,4) -- ++ (12,0);

\draw[dotted] (10.5,1.5) rectangle (11.5,2.5);

\node[] (15) at (6.5,-1) {$4ulul \dots uld$};
\node[] (16) at (6.5,-2.5) {$=1lul \dots uldr$};

\end{scope}

\end{tikzpicture}
\end{center}
\caption{Even collisions ending in $dr$ of length at least $6$.}
\label{fig:evencollisionsdr}
\end{figure}

Now, suppose we have a minimal collision of pin words $w_{1}$ and $w_{2}$ of length at least $6$ and that $w_{1}$ ends in $dr$. Then the centred permutation $\pi^{\circ} = \pi^{\circ}_{w_{1}} = \pi^{\circ}_{w_{2}}$ is of the form given in Figure \ref{i}. We aim to use this to deduce facts about $w_{2}$ in order to show that this collision is in fact one of the pairs listed in Figures \ref{fig:oddcollisionsdr} and \ref{fig:evencollisionsdr}.

\begin{figure}[h]
\begin{center}
\begin{tikzpicture}[scale=0.35]

\node[] (0) at (2,9) {$w_{1} = \dots dr$};

%\node[circle, draw, fill=none, inner sep=0pt, minimum width=\plotptradius] (0) at (4,7) {};
\node[permpt] (1) at (2,0) {}; \draw[thin] (1) -- ++ (0,7);
\node[permpt] (2) at (6,1) {}; \draw[thin] (2) -- ++ (-5.5,0);
\node[] (3) at (2,-1) {\tiny{A}};
\node[] (4) at (6,0) {\tiny{B}};

\draw (0.5,2.5) rectangle (3.5,5.5);

\end{tikzpicture}
\end{center}
\caption{The centred permutation $\pi^{\circ}$ generated by the pin word $w_{1}=\dots dr$ has this form. Note that $A$ is in the lower half-plane and $B$ is in the fourth quadrant. All points other than $A$ and $B$ (of which there are at least $4$ by the assumption that the length of $\pi^{\circ}$ is at least $6$) are in the box, with precisely one point on one side of the pin attached to $A$ and all other points on the other.}
\label{i}
\end{figure}

As in Figure \ref{i}, we call the lowest and second-lowest points of $\pi^{\circ}$ $A$ and $B$, respectively. These are generated by the final two letters $dr$ of $w_{1}$. We split into three cases based on the position of the letter of $w_{2}$ which generates $A$: the final letter, the initial numeral, or an internal letter.

\subsubsection*{Case $1$: $A$ is the final point}

We begin by deducing various facts about the pin word $w_{2}$, using the fact that the centred permutation $\pi^{\circ}$ looks like Figure \ref{ii} (with all points other than $A$ and $B$ in the box), as well as the assumption that the final letter of $w_{2}$ corresponds to the point $A$.

\begin{figure}[h]
\begin{center}
\begin{tikzpicture}[scale=0.35]

%\node[circle, draw, fill=none, inner sep=0pt, minimum width=\plotptradius] (0) at (4,7) {};
\node[permpt] (1) at (2,0) {}; \draw[thin,dotted] (1) -- ++ (0,7);
\node[permpt] (2) at (6,1) {};
\node[] (3) at (2,-1) {\tiny{A}};
\node[] (4) at (6,0) {\tiny{B}};

\draw (0.5,2.5) rectangle (3.5,5.5);

\end{tikzpicture}
\end{center}
\caption{The fact that the centred permutation $\pi^{\circ}$ can be generated by the pin word $w_{1}=\dots dr$ means that it has this form, with at least one point on either side of the dotted line within the box.}
\label{ii}
\end{figure}

By assumption, the final letter of $w_{2}$ corresponds to the point $A$. The point corresponding to the final letter of a pin word must be the most extreme point in the direction indicated by that letter. Figure \ref{ii} shows that $A$ is the downmost point but not the most extreme point in any other direction (clearly, $B$ is further up and to the right, and the fact that there must be at least one point in the box on each side of the dotted line implies that there is a point further to the left). Hence the final letter of $w_{2}$ \emph{must} be a $d$.

Consider the pin word $w_{2}^{-1}$, formed by removing the final letter of $w_{2}$. This must correspond to a centred permutation of the shape given in Figure \ref{iii} (basically Figure \ref{ii} without the point $A$). Note that the point $B$ does not separate the bounding rectangle of all other points in the centred permutation. By the definition of a pin permutation, this can only happen if $B$ was the first point placed, so $B$ corresponds to the initial numeral of $w_{2}$. But as $B$ also corresponds to the final $r$ in the pin word $w_{1} = \dots dr$ it must be in the fourth quadrant. Hence $w_{2}$ must begin with the numeral $4$.

\begin{figure}[h]
\begin{center}
\begin{tikzpicture}[scale=0.35]

%\node[circle, draw, fill=none, inner sep=0pt, minimum width=\plotptradius] (0) at (4,7) {};
%\node[permpt] (1) at (2,0) {}; \draw[thin,dotted] (1) -- ++ (0,7);
\node[permpt] (2) at (6,1) {};
%\node[] (3) at (2,-1) {\tiny{(A)}};
\node[] (4) at (6,0) {\tiny{B}};

\draw (0.5,2.5) rectangle (3.5,5.5);

\end{tikzpicture}
\end{center}
\caption{The centred permutation generated by $w_{2}^{-1}$ has this form.}
\label{iii}
\end{figure}

Hence $w_{2} = 4 \dots d$, with the $4$ corresponding to $B$ and the final $d$ corresponding to $A$. Suppose that the ellipsis in the middle contained an $r$. Then this would correspond to a point to the right of all points placed before. But as $B$ was the first point placed, this would imply the existence of a point to the right of $B$. Figure \ref{ii} shows that no such point exists, and so $w_{2}$ contains no $r$.

Similarly, suppose that $w_{2}$ contained another $d$, in addition to the final one. Then this would correspond to a point below $B$. But the only point below $B$ is $A$, already accounted for by the final $d$. Hence $w_{2}$ contains no $d$ apart from its final letter.

Combining these facts, we see that $w_{2} = 4 \dots d$, with the letters in the ellipsis alternating between $u$ and $l$. Hence $w_{2} = 4(ul)^{\geq2}d$ or $w_{2} = 4l(ul)^{\geq2}d$, depending on whether the length of $\pi^{\circ}$ is even or odd, respectively\footnote{We write $(f)^{\geq k}$ to mean that the word $f$ is repeated some finite number of times which is at least $k$. Thus $w_{2} = 4(ul)^{\geq2}d$ is really just an abbreviation for `$w_{2}=4(ul)^{m}d$, where $m \geq 2$'.}. We thus now know what the centred permutation $\pi^{\circ}$ looks like, as shown in Figure \ref{iv}.

\begin{figure}[h]
\begin{center}
\begin{tikzpicture}[scale=0.35]

\node[circle, draw, fill=none, inner sep=0pt, minimum width=\plotptradius] (0) at (10,2) {};
\node[permpt] (1) at (12,1) {}; \draw[thin,dotted] (1) -- ++ (-10.5,0);
\node[permpt] (2) at (11,4) {}; \draw[thin] (2) -- ++ (0,-3.5);
\node[permpt] (3) at (8,3) {}; \draw[thin] (3) -- ++ (3.5,0);
\node[permpt] (4) at (9,6) {}; \draw[thin] (4) -- ++ (0,-3.5);
\node[permpt] (5) at (6,5) {};  \draw[thin] (5) -- ++ (3.5,0);
\node[permpt] (6) at (7,7) {}; \draw[thin] (6) -- ++ (0,-2.5);
\node[permpt] (7) at (3,8) {}; \draw[thin] (7) -- ++ (1.5,0);
\node[permpt] (8) at (4,10) {}; \draw[thin] (8) -- ++ (0,-2.5);
\node[permpt] (9) at (1,9) {}; \draw[thin] (9) -- ++ (3.5,0);
\node[permpt] (10) at (2,0) {}; \draw[thin] (10) -- ++ (0,9.5);

\node[circle,fill,inner sep=0.5pt] (11) at (5,8) {};
\node[circle,fill,inner sep=0.5pt] (12) at (5.5,7.5) {};
\node[circle,fill,inner sep=0.5pt] (13) at (6,7) {};
\node[circle,fill,inner sep=0.5pt] (14) at (6.5,6.5) {};

\node[] (15) at (2,-1) {\tiny{A}};
\node[] (16) at (12,0) {\tiny{B}};

\draw[thick] (10,-1) -- ++ (0,12);
\draw[thick] (0,2) -- ++ (13,0);

%\draw[dotted] (0.5,8.5) rectangle (1.5,9.5);

\begin{scope}[shift={(18,0)}]

\node[circle, draw, fill=none, inner sep=0pt, minimum width=\plotptradius] (0) at (10,3) {};
\node[permpt] (1) at (11,1) {}; \draw[thin,dotted] (1) -- ++ (-9.5,0);
\node[permpt] (2) at (8,2) {}; \draw[thin] (2) -- ++ (3.5,0);
\node[permpt] (3) at (9,5) {}; \draw[thin] (3) -- ++ (0,-3.5);
\node[permpt] (4) at (6,4) {}; \draw[thin] (4) -- ++ (3.5,0);
\node[permpt] (5) at (7,7) {};  \draw[thin] (5) -- ++ (0,-3.5);
\node[permpt] (6) at (5,6) {}; \draw[thin] (6) -- ++ (2.5,0);
\node[permpt] (7) at (3,9) {}; \draw[thin] (7) -- ++ (0,-1.5);
\node[permpt] (8) at (1,8) {}; \draw[thin] (8) -- ++ (2.5,0);
\node[permpt] (9) at (2,0) {}; \draw[thin] (9) -- ++ (0,9.5);

\node[circle,fill,inner sep=0.5pt] (10) at (4,8) {};
\node[circle,fill,inner sep=0.5pt] (11) at (4.5,7.5) {};
\node[circle,fill,inner sep=0.5pt] (12) at (5,7) {};
\node[circle,fill,inner sep=0.5pt] (13) at (5.5,6.5) {};

\node[] (14) at (2,-1) {\tiny{A}};
\node[] (15) at (11,0) {\tiny{B}};

\draw[thick] (10,-1) -- ++ (0,11);
\draw[thick] (0,3) -- ++ (12,0);

%\draw[dotted] (0.5,8.5) rectangle (1.5,9.5);

\end{scope}

\end{tikzpicture}
\end{center}
\caption{Even and odd cases for the centred permutation $\pi^{\circ}$, respectively.}
\label{iv}
\end{figure}

We now return to $w_{1} = \dots dr$. As we now know what the centred permutation $\pi^{\circ}$ looks like, we can deduce that the ellipsis here also does not contain a $d$ or $r$. First, if the ellipsis contained a $d$ then the points corresponding to both this and the next letter would be in the lower half-plane. But Figure \ref{iv} shows that there is at most one point in the lower half-plane in addition to $A$ and $B$ (which are already accounted for by the final two letters). Similarly, if the ellipsis contained an $r$ then the points corresponding to this and the next letter (neither of which can be the point $B$ as this is accounted for by the final $r$) would be in the right half-plane. But Figure \ref{iv} shows that there is at most one point in the right half-plane other than $B$. Hence $w_{1} = \dots dr$ contains no $d$ or $r$ apart from the final two letters. This is now enough to deduce all of $w_{1}$: $w_{1} = 1(lu)^{\geq1}ldr$ in the even case, and $w_{1} = 3(ul)^{\geq2}dr$ in the odd case.

This means that there is only one potential collision of each length $n \geq 6$ in Case $1$: $w_{1} = 1(lu)^{\geq1}ldr$ paired with $w_{2} = 4(ul)^{\geq2}d$ in the even case, and $w_{1} = 3(ul)^{\geq2}dr$ paired with $w_{2} = 4l(ul)^{\geq2}d$ in the odd case. These are the collisions on the right hand sides of Figures \ref{fig:oddcollisionsdr} and \ref{fig:evencollisionsdr}.

\subsubsection*{Case $2$: $A$ is the first point}

We now deal with the case in which $A$ is the first point placed according to the pin word $w_{2}$, corresponding to the initial numeral. We again use the shape of the centred permutation $\pi^{\circ}$ shown in Figure \ref{i} to deduce the form of $w_{2}$.

First note that, as $A$ corresponds to the letter $d$ in $w_{1}$, $A$ must be in the $3$rd or $4$th quadrant. Hence $w_{2} = (3|4)\dots$

Next, note that, as $B$ is not the first point placed, it must correspond to a letter. If this letter were $d$ or $l$ then $B$ would be below or to the left of $A$ (as $A$ has already been placed), which it is not (see Figure \ref{i}). If the letter were $u$, then $B$ would be the upmost of all points placed so far, and as Figure \ref{i} shows that all other points of $\pi$ are above $B$, this would mean that $B$ would have to be the second point placed, so $w_{2}$ begins with either $3u$ or $4u$. In the first case, both $A$ and $B$ would be in quadrant $3$ and in the second case $B$ would be to the left of $A$ --- Figure \ref{i} shows that neither of these is true, so $B$ cannot correspond to a $u$ in $w_{2}$. Hence, by elimination, $B$ corresponds to an $r$ in $w_{2}$.

Hence $B$ is at the end of a right-pin separating the previously placed point from the bounding rectangle of all other previously placed points and the origin. As the origin is above $B$ (as we know $B$ is in quadrant $4$), the previously placed point must be the only point below $B$, namely $A$. Thus $B$ corresponds to the first letter of $w_{2}$ after the numeral.

Hence $w_{2}=(3|4)r\dots $. We can now easily deduce that the ellipsis here contains no $r$ or $d$: if it contained an $r$ then this would correspond to a point to the right of $B$ (as this has already been placed), and it if contained a $d$ then this would correspond to a point below $A$ --- but Figure \ref{i} clearly shows that no point of either type exists.

Hence $w_{2} = (3|4)rulul \dots$, and so the centred permutation $\pi^{\circ}$ looks like one of the pictures in Figure \ref{v}.

\begin{figure}[h]
\begin{center}
\begin{tikzpicture}[scale=0.35]

\node[circle, draw, fill=none, inner sep=0pt, minimum width=\plotptradius] (0) at (5,3) {};
\node[permpt] (1) at (4,1) {};
\node[permpt] (2) at (7,2) {}; \draw[thin] (2) -- ++ (-3.5,0);
\node[permpt] (3) at (6,5) {}; \draw[thin] (3) -- ++ (0,-3.5);
\node[permpt] (4) at (2,4) {}; \draw[thin] (4) -- ++ (4.5,0);
\node[permpt] (5) at (3,7) {};  \draw[thin] (5) -- ++ (0,-3.5);
\node[permpt] (6) at (1,6) {}; \draw[thin] (6) -- ++ (2.5,0);

\node[circle,fill,inner sep=0.5pt] (7) at (1.5,7) {};
\node[circle,fill,inner sep=0.5pt] (8) at (1,7.5) {};
\node[circle,fill,inner sep=0.5pt] (9) at (0.5,8) {};
\node[circle,fill,inner sep=0.5pt] (10) at (0,8.5) {};

\node[] (11) at (4,0) {\tiny{A}};
\node[] (12) at (7,1) {\tiny{B}};
\node[] (13) at (7,5) {\tiny{C}};

\draw[thick] (5,0) -- ++ (0,8);
\draw[thick] (0,3) -- ++ (8,0);

%\draw[dotted] (0.5,8.5) rectangle (1.5,9.5);

\begin{scope}[shift={(13,0)}]

\node[circle, draw, fill=none, inner sep=0pt, minimum width=\plotptradius] (0) at (4,3) {};
\node[permpt] (1) at (5,1) {};
\node[permpt] (2) at (7,2) {}; \draw[thin] (2) -- ++ (-2.5,0);
\node[permpt] (3) at (6,5) {}; \draw[thin] (3) -- ++ (0,-3.5);
\node[permpt] (4) at (2,4) {}; \draw[thin] (4) -- ++ (4.5,0);
\node[permpt] (5) at (3,7) {};  \draw[thin] (5) -- ++ (0,-3.5);
\node[permpt] (6) at (1,6) {}; \draw[thin] (6) -- ++ (2.5,0);

\node[circle,fill,inner sep=0.5pt] (7) at (1.5,7) {};
\node[circle,fill,inner sep=0.5pt] (8) at (1,7.5) {};
\node[circle,fill,inner sep=0.5pt] (9) at (0.5,8) {};
\node[circle,fill,inner sep=0.5pt] (10) at (0,8.5) {};

\node[] (11) at (5,0) {\tiny{A}};
\node[] (12) at (7,1) {\tiny{B}};
\node[] (13) at (7,5) {\tiny{C}};

\draw[thick] (4,0) -- ++ (0,8);
\draw[thick] (0,3) -- ++ (8,0);

%\draw[dotted] (0.5,8.5) rectangle (1.5,9.5);

\end{scope}

\end{tikzpicture}
\end{center}
\caption{The two possible cases for the centred permutation $\pi^{\circ}$, as generated by the pin word $w_{2} = (3|4)rulul \dots$. As $n \geq 6$ all the points shown must exist; any further points are a continuation of the oscillation in the second quadrant.}
\label{v}
\end{figure}

We now return to $w_{1} = \dots dr$ which also generates $\pi^{\circ}$, with the final two letters corresponding to $A$ and $B$. Note first that all the points in the upper half-plane of Figure \ref{v} have been placed before $A$ in $w_{1}$, and in either case there are at least three points to the left of $A$ which have been placed already. This means that the previous letter to the $d$ corresponding to $A$ must have been an $r$, and this must correspond to the point marked $C$ (the only point other than $B$ in the right half-plane). Thus $w_{1} = \dots rdr$, with the final three letters corresponding to $C$, $A$, $B$, respectively. Note that the ellipsis here cannot contain a $d$, as $A$ and $B$ (already accounted for by the final two letters) are the only points in the lower half-plane. Thus $w_{1} = \dots urdr$. But this implies that $C$ is the second-upmost point, despite Figure \ref{v} showing that $C$ has at least two points above it. This contradiction implies that there are in fact no collisions in Case $2$.

\subsubsection*{Case $3$: $A$ is an internal point}

We now deal with the case where $A$ is generated by an internal letter of the pin word $w_{2}$ (ie., neither the initial numeral nor the final letter).

First, we determine which letter represents $A$ in $w_{2}$. By looking at Figure \ref{i}, we see that $A$ is in the lower half-plane and is indeed the lowest point in the entire centred permutation. If $A$ were represented by a $u$ in $w_{2}$ it would be in the upper half-plane, so we can exclude this possibility. If $A$ were represented by an $l$ or $r$, on the other hand, it would have to follow either the numeral $3$ or $4$ or the letter $d$ (as otherwise $A$ would be in the upper half-plane). But in this case $A$ would be above the previously-placed point, contradicting the fact that it is the lowest point in the entire permutation. Hence $A$ must be represented by the letter $d$ in $w_{2}$.

Next, note that the $d$ representing $A$ must in fact be the \emph{final} $d$ in $w_{2}$: if there were a $d$ after $A$ it would place a point below $A$, but $A$ is the lowest point in the entire centred permutation. As $A$ is generated by a letter internal to $w_{2}$ there is a letter immediately after the $d$ corresponding to $A$ in $w_{2}$, either an $l$ or an $r$. The point corresponding to this next letter, which we call $B'$, will be the second-lowest point \emph{so far} when it is placed. But since all points placed after $B'$ (if there are any) will be in the upper half-plane (as there is no further $d$ in the word), $B'$ is in fact the second-lowest point in the \emph{entire} centred permutation. Hence $B'$ must actually be $B$, which must correspond to an $r$ in $w_{2}$ as it is to the right of $A$.

We now know that the points $A$ and $B$ are generated by a consecutive $dr$ in $w_{2}$, and that there is no letter $d$ in $w_{2}$ after this $dr$ appears. By the same token there is also no further $r$ after this $dr$, as this would generate a point to the right of $B$. Hence, after the $dr$ in $w_{2}$ corresponding to $A,B$, the pin word $w_{2}$ alternates between $u$ and $l$ for any remaining points. We now ask how many points are remaining in $w_{2}$ after this final $dr$.

If there were \emph{no} points after the $dr$ in $w_{2}$ then $w_{1}$ and $w_{2}$ would both end in $dr$, so this would not be a pair of \emph{minimal} collisions. Hence we can assume that there is at least one further letter after the final $dr$, which must be a $u$.

Now suppose that there were at least two letters after the final $dr$. Then $w_{2}$ would have the form $\dots drul \dots$, with any further letters at the end alternating between $u$ and $l$. Thus $\pi^{\circ}$ would look like the centred permutation shown on the right-hand side of Figure \ref{vi}.

\begin{figure}[h]
\begin{center}
\begin{tikzpicture}[scale=0.35]

\node[] (0) at (2,9) {$w_{1} = \dots dr$};

%\node[circle, draw, fill=none, inner sep=0pt, minimum width=\plotptradius] (0) at (4,7) {};
\node[permpt] (1) at (2,0) {}; \draw[thin] (1) -- ++ (0,5.5);
\node[permpt] (2) at (5,1) {}; \draw[thin] (2) -- ++ (-3.5,0);
\node[] (3) at (2,-1) {\tiny{A}};
\node[] (4) at (5,0) {\tiny{B}};

\draw (0.5,2) rectangle (3.5,5);

\begin{scope}[shift={(13,2)}]

\node[] (0) at (1,7) {$w_{2} = \dots drul \dots$};

%\node[circle, draw, fill=none, inner sep=0pt, minimum width=\plotptradius] (0) at (4,7) {};
\node[permpt] (1) at (1,-2) {}; \draw[thin] (1) -- ++ (0,4.5);
\node[permpt] (2) at (4,-1) {}; \draw[thin] (2) -- ++ (-3.5,0);
\node[permpt] (3) at (3,4) {}; \draw[thin] (3) -- ++ (0,-5.5);
\node[permpt] (4) at (-2,3) {}; \draw[thin] (4) -- ++ (5.5,0);
%\node[permpt] (5) at (-1,5) {}; \draw[thin] (5) -- ++ (0,-2.5);
\node[] (6) at (1,-3) {\tiny{A}};
\node[] (7) at (4,-2) {\tiny{B}};
\node[] (8) at (3,5) {\tiny{C}};
\node[] (9) at (-2,2) {\tiny{D}};

\node[circle,fill,inner sep=0.5pt] (8) at (-1.5,4) {};
\node[circle,fill,inner sep=0.5pt] (9) at (-2,4.5) {};
\node[circle,fill,inner sep=0.5pt] (10) at (-2.5,5) {};
\node[circle,fill,inner sep=0.5pt] (11) at (-3,5.5) {};

\draw (0,0) rectangle (2,2);

\end{scope}

\end{tikzpicture}
\end{center}
\caption{The centred permutation $\pi^{\circ}$ is (under the assumption that there are at least two letters after the $r$ that generates $B$ in $w_{2}$) generated by both the words $w_{1} = \dots dr$ and $w_{2} = \dots drul \dots$, so has both of these forms. The points marked $A$ and $B$ must match up.}
\label{vi}
\end{figure}

Hence the centred permutation $\pi^{\circ}$ is simultaneously of both forms represented in Figure \ref{vi}. On closer inspection, however, these forms contradict each other, which we can see by counting the points other than $B$ to the left and right of $A$: looking first at the diagram generated by $w_{1}$, the box here is non-empty (as the length of $\pi^{\circ}$ is at least $6$), and the pin attached to $A$ must separate the previously-placed point from all others (including the origin). Hence, excluding $B$, there is precisely one point on one side of $A$ and all other points (including the origin) are on the other side. Looking at the diagram generated by $w_{2}$, however, we see that the box here is also non-empty (as $A$ is internal there is at least one non-origin point preceeding it), with the previous point on one side and all other points (including the origin) on the other. Including the points marked $C$ and $D$, we now see that there are at least two points (including the origin but excluding $B$) on either side of $A$, thus contradicting what the diagram generated by $w_{1}$ told us. We conclude that there cannot in fact be more than one point after the final $dr$ in $w_{2}$.

Thus there is in fact precisely one point after the final $dr$, corresponding to $A,B$ in $w_{2}$, and so $w_{2} = \dots dru$. Hence our centred permutation $\pi^{\circ}$ is of both the forms shown in Figure \ref{vii}.

\begin{figure}[h]
\begin{center}
\begin{tikzpicture}[scale=0.35]

\node[] (0) at (2,9) {$w_{1} = \dots dr$};

%\node[circle, draw, fill=none, inner sep=0pt, minimum width=\plotptradius] (0) at (4,7) {};
\node[permpt] (1) at (2,0) {}; \draw[thin] (1) -- ++ (0,5.5);
\node[permpt] (2) at (5,1) {}; \draw[thin] (2) -- ++ (-3.5,0);
\node[] (3) at (2,-1) {\tiny{A}};
\node[] (4) at (5,0) {\tiny{B}};

\draw (0.5,2) rectangle (3.5,5);

\begin{scope}[shift={(13,2)}]

\node[] (0) at (1,7) {$w_{2} = \dots dru$};

%\node[circle, draw, fill=none, inner sep=0pt, minimum width=\plotptradius] (0) at (4,7) {};
\node[permpt] (1) at (1,-2) {}; \draw[thin] (1) -- ++ (0,4.5);
\node[permpt] (2) at (4,-1) {}; \draw[thin] (2) -- ++ (-3.5,0);
\node[permpt] (3) at (3,4) {}; \draw[thin] (3) -- ++ (0,-5.5);
%\node[permpt] (5) at (-1,5) {}; \draw[thin] (5) -- ++ (0,-2.5);
\node[] (6) at (1,-3) {\tiny{A}};
\node[] (7) at (4,-2) {\tiny{B}};
\node[] (8) at (3,5) {\tiny{C}};

\draw (0,0) rectangle (2,2);

\end{scope}

\end{tikzpicture}
\end{center}
\caption{The centred permutation $\pi^{\circ}$ is generated by both the words $w_{1} = \dots dr$ and $w_{2} = \dots dru$, so has both of these forms. The points marked $A$ and $B$ must match up.}
\label{vii}
\end{figure}

But now consider the pin word $w_{1}^{-1}$, obtained by removing the final letter from $w_{1}$. This must generate a pin permutation corresponding to the centred permutations shown in Figure \ref{vii} with the point $B$ removed. In particular, it contains the point $C$ which (given the absence of $B$) will be both the highest  and rightmost point in the corresponding pin permutation $\pi^{\circ}_{w_{1}^{-1}}$. This can only happen if $C$ was the \emph{first} point placed in $\pi^{\circ}_{w_{1}^{-1}}$ (otherwise it separates the previously placed point from all others in one direction, making it only the second-most point in that direction). As $C$ is clearly in the first quadrant, this means that $w_{1}$ starts with a $1$. Hence $w_{1} = 1 \dots dr$. Further, given that there is no point above $C$ in $\pi^{\circ}$ and the only point to the right is $B$, accounted for by the final $r$, this means that there is no $u$ in $w_{1}$ and no $r$ except for the final letter. Hence $w_{1}$ is either $1l(dl)^{\geq 1}dr$ or $1(dl)^{\geq 2}dr$, depending on whether the length of $\pi^{\circ}$ is even or odd. Thus $\pi^{\circ}$ is one of the centred permutations on the left-hand sides of Figures \ref{fig:oddcollisionsdr} and \ref{fig:evencollisionsdr}, with $w_{1}$ being the upper word in the pair. Noting that $w_{2} = \dots dru$ cannot contain an $r$ or $u$ apart from the final two letters (all points on the upper and right half-planes are already accounted for) allows us to conclude that $w_{2}$ must be the lower word in the pair. \qedhere

\end{proof}

We can now quickly deduce that \emph{all} collisions are in fact minimal, and thus the list given above is complete:

\begin{lemma}

Every colliding tuple $\{w_{1},w_{2}, \dots, w_{k}\}$ of pin words is minimal (that is to say, $k \leq 4$ and each pair $w_{i},w_{j}$ differs in the final letter).

\end{lemma}

\begin{proof}

We write $w^{-n}$ for the pin word obtained from a pin word $w$ by removing the final $n$ letters.

Note that every collision of length $n \leq 5$ (which we have exhaustively listed) is minimal. Now suppose that $\{w_{1},w_{2}\}$ is a non-minimal colliding pair of some length $n \geq 6$. By applying symmetries, we can assume that both $w_{1}$ and $w_{2}$ ends in a $d$. But then this $d$ generates the lowest point in the shared generated centred pin permutation $\pi^{\circ}$ --- in particular, it represents the same point in each pin word. Thus $w_{1}^{-1}$ and $w_{2}^{-1}$ must also form a collision, indeed a collision of shorter length. Given that $w_{1}$ and $w_{2}$ are not the same word, we can continue this process until we arrive at a minimal collision $w_{1}^{-n},w_{2}^{-n}$. But we have already classified all minimal collisions, and can note that all minimal collisions have different alignment of the terminal letter: if one of the pair ends in $d$ or $u$ then the other ends in $l$ or $r$. This means that such a pair cannot be extended to a collision of longer length, forming a contradiction. \qedhere

\end{proof}

Thus all collisions are minimal and the list given in the table is complete.

\subsection{Classification of $\boxplus$-decomposables}

We now aim to prove that the list of $\boxplus$-decomposables given in Theorem \ref{classification} is complete:

\begin{proof}
First, we list the $\boxplus$-decomposables of lengths $2$ and $3$: these are $1l$ and $1ld$ along with their eight symmetries, as can be confirmed by an exhaustive list.

We now construct a $\boxplus$-decomposable pin permutation $\pi^{\circ} = \pi^{\circ}_{w}$ of length $n \geq 4$, generated by the pin word $w$, and show that it must be one of those listed in the statement of the theorem. We proceed by noting that every $\boxplus$-decomposable pin permutation can be drawn on the diagram in Figure \ref{fig:boxdec1}, with no points in the shaded region, at least one point in the inner region, and at least one point in one of the outer regions.

\begin{figure}[h]
\begin{center}
\begin{tikzpicture}[scale=0.25]

%\draw[step=1cm,gray!50,very thin](0.5,0.5)grid(17.5,16.5);

\draw[very thick] (0,-11) -- ++ (0,22);
\draw[very thick] (-11,0) -- ++ (22,0);

\draw[thin] (-4,-11) -- ++ (0,22);
\draw[thin] (4,-11) -- ++ (0,22);
\draw[thin] (-11,-4) -- ++ (22,0);
\draw[thin] (-11,4) -- ++ (22,0);

%\node[circle, draw, fill=none, inner sep=0pt, minimum width=4pt] (0) at (13,7) {};

\draw[pattern=crosshatch,pattern color=black!80,draw=none] (-4,4) rectangle (4,11);
\draw[pattern=crosshatch,pattern color=black!80,draw=none] (-11,-4) rectangle (-4,4);
\draw[pattern=crosshatch,pattern color=black!80,draw=none] (-4,-11) rectangle (4,-4);
\draw[pattern=crosshatch,pattern color=black!80,draw=none] (4,-4) rectangle (11,4);

%\draw [gray!50]  (-5,9.5) -- (-5,8) -- (-5,7) -- (-4,6) -- (-2,4)  -- (-2,2);

\end{tikzpicture}
\end{center}
\caption{The general shape of a $\boxplus$-decomposable permutation. The shaded region must be empty, whilst both the inner and at least one of the outer unshaded regions must be non-empty.}
\label{fig:boxdec1}
\end{figure}

By applying symmetries if necessary we can assume that the first point of $\pi^{\circ}$ placed was in the first quadrant. In fact, this point must be placed in the outer region of the first quadrant, as it is impossible to get out of the inner region once a point has been placed in there: if the point $p_{n}$ has been placed in the inner region then the point $p_{n+1}$ will be closer to the origin in the direction specified by the letter that placed $p_{n}$, and hence must also be in the inner region. Hence our centred permutation $\pi^{\circ}$ starts out like that in Figure \ref{fig:boxdec2}.

\begin{figure}[h]
\begin{center}
\begin{tikzpicture}[scale=0.25]

%\draw[step=1cm,gray!50,very thin](0.5,0.5)grid(17.5,16.5);

\draw[very thick] (0,-11) -- ++ (0,22);
\draw[very thick] (-11,0) -- ++ (22,0);

\draw[thin] (-4,-11) -- ++ (0,22);
\draw[thin] (4,-11) -- ++ (0,22);
\draw[thin] (-11,-4) -- ++ (22,0);
\draw[thin] (-11,4) -- ++ (22,0);

%\node[circle, draw, fill=none, inner sep=0pt, minimum width=4pt] (0) at (13,7) {};

\node[permpt,label={\tiny $p_{1}$}] (1) at (5,6) {};
%\uncover<3->{\node[permpt,label={\tiny $p_{2}$}] (2) at (7,5) {}; \draw[thin] (2) -- ++ (-2.5,0);}
%\uncover<6->{\node[permpt,label={\tiny $p_{3}$}] (3) at (6,8) {};  \draw[thin] (3) -- ++ (0,-3.5);}
%\uncover<7->{\node[permpt,label={\tiny $p_{4}$}] (4) at (9,7) {}; \draw[thin] (4) -- ++ (-3.5,0);}
%\uncover<8->{\node[permpt,label={\tiny $p_{5}$}] (5) at (8,10) {}; \draw[thin] (5) -- ++ (0,-3.5);}
%\uncover<9->{\node[permpt,label={\tiny $p_{6}$}] (6) at (-6,9) {}; \draw[thin] (6) -- ++ (14.5,0);}
%\uncover<12->{\node[permpt,label={\tiny $p_{7}$}] (7) at (-2,-2) {};}

%\uncover<10>{\node[permpt,red] (8) at (-5,11) {}; \draw[thin,red] (8) -- ++ (0,-2.5);}
%\uncover<11>{\node[permpt,red] (9) at (-5,-6) {}; \draw[thin,red] (9) -- ++ (0,15.5);}

\draw[pattern=crosshatch,pattern color=black!80,draw=none] (-4,4) rectangle (4,11);
\draw[pattern=crosshatch,pattern color=black!80,draw=none] (-11,-4) rectangle (-4,4);
\draw[pattern=crosshatch,pattern color=black!80,draw=none] (-4,-11) rectangle (4,-4);
\draw[pattern=crosshatch,pattern color=black!80,draw=none] (4,-4) rectangle (11,4);

%\uncover<4->{\draw[pattern=crosshatch,pattern color=red!80,draw=none] (0,-4) rectangle (4,4);}
%\uncover<5->{\draw[pattern=crosshatch,pattern color=red!80,draw=none] (-4,0) rectangle (0,4);}

%\draw [gray!50]  (-5,9.5) -- (-5,8) -- (-5,7) -- (-4,6) -- (-2,4)  -- (-2,2);

%\uncover<12->{\draw [thin] plot [smooth] coordinates {(-5,9.5) (-5,7) (-4,6) (-2,4) (-2,-2)};}
	
\end{tikzpicture}
\end{center}
\caption{First point of $\pi^{\circ}$ is placed in the outer region of the first quadrant.}
\label{fig:boxdec2}
\end{figure}

Next, note that it is impossible to get into the inner region of the first quadrant now: every point placed in the first quadrant from now on will be either to the right of or above $p_{1}$. Hence in order to construct a $\boxplus$-decomposable, we must at some point move to another quadrant. Again, by symmetry, we can assume that the next quadrant visited by the centred pin permutation $\pi^{\circ}$ is the second quadrant, and we now split into cases based on whether this first point in the second quadrant is in the inner or outer region:

\subsubsection*{Case $1$: First point in the second quadrant is in the outer region}

In this case, we may or may not begin by oscillating in the first quadrant for any length. If there is an oscillation of length $\geq 2$ this will ensure that the first point placed in the second quadrant will be in the outer region (it will be the second-highest point placed so far, and so will be above $p_{1}$, the first point placed, which is in the outer region). If, on the other hand, we leave the first quadrant on the second step, the second point may or may not be in the outer region. For now, we assume that the first point placed in the second quadrant is in the outer region (regardless of whether there is an oscillation of length greater than $1$ in the first quadrant) and deal with the the case in which it is in the inner region as Case $2$, below. We thus have a centred permutation that looks like that in Figure \ref{fig:boxdec3}.

\begin{figure}[h]
\begin{center}
\begin{tikzpicture}[scale=0.25]

%\draw[step=1cm,gray!50,very thin](0.5,0.5)grid(17.5,16.5);

\draw[very thick] (0,-11) -- ++ (0,22);
\draw[very thick] (-11,0) -- ++ (22,0);

\draw[thin] (-4,-11) -- ++ (0,22);
\draw[thin] (4,-11) -- ++ (0,22);
\draw[thin] (-11,-4) -- ++ (22,0);
\draw[thin] (-11,4) -- ++ (22,0);

%\node[circle, draw, fill=none, inner sep=0pt, minimum width=4pt] (0) at (13,7) {};

\node[permpt,label={\tiny $p_{1}$}] (1) at (5,6) {};
\node[permpt,label={\tiny $p_{2}$}] (2) at (7,5) {}; \draw[thin] (2) -- ++ (-2.5,0);
\node[permpt,label={\tiny $p_{3}$}] (3) at (6,8) {};  \draw[thin] (3) -- ++ (0,-3.5);
\node[permpt,label={\tiny $p_{4}$}] (4) at (9,7) {}; \draw[thin] (4) -- ++ (-3.5,0);
\node[permpt,label={\tiny $p_{5}$}] (5) at (8,10) {}; \draw[thin] (5) -- ++ (0,-3.5);
\node[permpt,label={\tiny $p_{6}$}] (6) at (-6,9) {}; \draw[thin] (6) -- ++ (14.5,0);
%\uncover<12->{\node[permpt,label={\tiny $p_{7}$}] (7) at (-2,-2) {};}

%\uncover<10>{\node[permpt,red] (8) at (-5,11) {}; \draw[thin,red] (8) -- ++ (0,-2.5);}
%\uncover<11>{\node[permpt,red] (9) at (-5,-6) {}; \draw[thin,red] (9) -- ++ (0,15.5);}

\draw[pattern=crosshatch,pattern color=black!80,draw=none] (-4,4) rectangle (4,11);
\draw[pattern=crosshatch,pattern color=black!80,draw=none] (-11,-4) rectangle (-4,4);
\draw[pattern=crosshatch,pattern color=black!80,draw=none] (-4,-11) rectangle (4,-4);
\draw[pattern=crosshatch,pattern color=black!80,draw=none] (4,-4) rectangle (11,4);

\draw[pattern=crosshatch,pattern color=blue!80,draw=none] (0,-4) rectangle (4,4);
\draw[pattern=crosshatch,pattern color=blue!80,draw=none] (-4,0) rectangle (0,4);

%\draw [gray!50]  (-5,9.5) -- (-5,8) -- (-5,7) -- (-4,6) -- (-2,4)  -- (-2,2);

%\uncover<12->{\draw [thin] plot [smooth] coordinates {(-5,9.5) (-5,7) (-4,6) (-2,4) (-2,-2)};}
	
\end{tikzpicture}
\end{center}
\caption{The first point of $\pi^{\circ}$ outside the first quadrant can be assumed to be in the second quadrant by symmetry.}
\label{fig:boxdec3}
\end{figure}

We note that the inner regions of the first, second and fourth quadrants (shaded in blue in Figure \ref{fig:boxdec3}) are now inaccessible: there are now at least two points in the outer region of the upper half-plane, and because any new point in the upper half-plane must be either the highest or second-highest placed so far, it is now impossible to place any points below those two, hence the inner region of the upper half-plane will remain empty. A similar argument shows that the inner region of the right-half plane is now also inaccessible: if we had oscillated in the first quadrant initially then there are now at least two points to the right of the inner region in the right half-plane, and the argument goes through exactly as before. If we went directly to the second quadrant with our second point then there would be only one point to the right of the inner region so far, but in order to place anything in the right half-plane again we would need to take a right step, which would create a second point to the right of the inner region, thus rendering the inner region of the right half-plane again inaccessible. In any case, the only part of the inner region which is now accessible is in the third quadrant, so we must be aiming to end up there.

If, after placing our first point in the second quadrant, we took either an upward step or a downward step into the outer region of the third quadrant, then the whole of the inner region would be rendered inaccessbile, by the same argument as in the previous paragraph. Hence we must now take a downstep into the inner region of the third quadrant, as in Figure \ref{fig:boxdec4}.

\begin{figure}[h]
\begin{center}
\begin{tikzpicture}[scale=0.25]

%\draw[step=1cm,gray!50,very thin](0.5,0.5)grid(17.5,16.5);

\draw[very thick] (0,-11) -- ++ (0,22);
\draw[very thick] (-11,0) -- ++ (22,0);

\draw[thin] (-4,-11) -- ++ (0,22);
\draw[thin] (4,-11) -- ++ (0,22);
\draw[thin] (-11,-4) -- ++ (22,0);
\draw[thin] (-11,4) -- ++ (22,0);

%\node[circle, draw, fill=none, inner sep=0pt, minimum width=4pt] (0) at (13,7) {};

\node[permpt,label={\tiny $p_{1}$}] (1) at (5,6) {};
\node[permpt,label={\tiny $p_{2}$}] (2) at (7,5) {}; \draw[thin] (2) -- ++ (-2.5,0);
\node[permpt,label={\tiny $p_{3}$}] (3) at (6,8) {};  \draw[thin] (3) -- ++ (0,-3.5);
\node[permpt,label={\tiny $p_{4}$}] (4) at (9,7) {}; \draw[thin] (4) -- ++ (-3.5,0);
\node[permpt,label={\tiny $p_{5}$}] (5) at (8,10) {}; \draw[thin] (5) -- ++ (0,-3.5);
\node[permpt,label={\tiny $p_{6}$}] (6) at (-6,9) {}; \draw[thin] (6) -- ++ (14.5,0);
\node[permpt,label={\tiny $p_{7}$}] (7) at (-2,-2) {};

%\node[permpt,red] (8) at (-5,11) {}; \draw[thin,red] (8) -- ++ (0,-2.5);
%\node[permpt,red] (9) at (-5,-6) {}; \draw[thin,red] (9) -- ++ (0,15.5);

\draw[pattern=crosshatch,pattern color=black!80,draw=none] (-4,4) rectangle (4,11);
\draw[pattern=crosshatch,pattern color=black!80,draw=none] (-11,-4) rectangle (-4,4);
\draw[pattern=crosshatch,pattern color=black!80,draw=none] (-4,-11) rectangle (4,-4);
\draw[pattern=crosshatch,pattern color=black!80,draw=none] (4,-4) rectangle (11,4);

\draw[pattern=crosshatch,pattern color=blue!80,draw=none] (0,-4) rectangle (4,4);
\draw[pattern=crosshatch,pattern color=blue!80,draw=none] (-4,0) rectangle (0,4);

%\draw [gray!50]  (-5,9.5) -- (-5,8) -- (-5,7) -- (-4,6) -- (-2,4)  -- (-2,2);

\draw [thin] plot [smooth] coordinates {(-5,9.5) (-5,7) (-4,6) (-2,4) (-2,-2)};
	
\end{tikzpicture}
\end{center}
\caption{This is $\boxplus$-decomposable, but will not be if any further point is added.}
\label{fig:boxdec4}
\end{figure}

Note that we cannot now place any further points: either a left or right step here will place a point into the shaded region. Hence the only $\boxplus$-decomposable pin permutations in Case $1$ are those shown in Figure \ref{fig:boxdec4}, which is to say those generated by a pin word of the form $1 \dots ld$, where the only letters in the ellipsis are $u$ and $r$.

\subsubsection*{Case $2$: First point in the second quadrant is in the inner region}

In this case, it is clear that we have to move into the second quadrant immediately after placing the first point in quadrant $1$ (otherwise the first point in quadrant $2$ will be above $p_{1}$ and hence not in the inner region). Hence we start off as in Figure \ref{fig:boxdec5}.

\begin{figure}[h]
\begin{center}
\begin{tikzpicture}[scale=0.25]

%\draw[step=1cm,gray!50,very thin](0.5,0.5)grid(17.5,16.5);

\draw[very thick] (0,-11) -- ++ (0,22);
\draw[very thick] (-11,0) -- ++ (22,0);

\draw[thin] (-7,-11) -- ++ (0,22);
\draw[thin] (5,-11) -- ++ (0,22);
\draw[thin] (-11,-7) -- ++ (22,0);
\draw[thin] (-11,4) -- ++ (22,0);

%\node[circle, draw, fill=none, inner sep=0pt, minimum width=4pt] (0) at (13,7) {};

\node[permpt,label={\tiny $p_{1}$}] (1) at (7,6) {};
\node[permpt,label={\tiny $p_{2}$}] (2) at (-2,2) {};

\draw[pattern=crosshatch,pattern color=black!80,draw=none] (-7,4) rectangle (5,11);
\draw[pattern=crosshatch,pattern color=black!80,draw=none] (-11,-7) rectangle (-7,4);
\draw[pattern=crosshatch,pattern color=black!80,draw=none] (-7,-11) rectangle (5,-7);
\draw[pattern=crosshatch,pattern color=black!80,draw=none] (5,-7) rectangle (11,4);

%\draw [gray!50]  (-5,9.5) -- (-5,8) -- (-5,7) -- (-4,6) -- (-2,4)  -- (-2,2);

\draw [thin] plot [smooth] coordinates {(7.5,5) (5,5) (3,4) (0,2) (-2,2)};
	
\end{tikzpicture}
\end{center}
\caption{This is $\boxplus$-decomposable, but will not be if any right or up step is taken.}
\label{fig:boxdec5}
\end{figure}

Note that we can now no longer take an up- or right-step: doing so would place a point above or to the right of $p_{1}$ and therefore within the shaded region. Conversely, if we only take down- and left-steps we can stay in the inner region indefinitely, as shown in Figure \ref{fig:boxdec6}.

\begin{figure}[h]
\begin{center}
\begin{tikzpicture}[scale=0.25]

%\draw[step=1cm,gray!50,very thin](0.5,0.5)grid(17.5,16.5);

\draw[very thick] (0,-11) -- ++ (0,22);
\draw[very thick] (-11,0) -- ++ (22,0);

\draw[thin] (-7,-11) -- ++ (0,22);
\draw[thin] (5,-11) -- ++ (0,22);
\draw[thin] (-11,-7) -- ++ (22,0);
\draw[thin] (-11,4) -- ++ (22,0);

%\node[circle, draw, fill=none, inner sep=0pt, minimum width=4pt] (0) at (13,7) {};

\node[permpt,label={\tiny $p_{1}$}] (1) at (7,6) {};
\node[permpt,label={\tiny $p_{2}$}] (2) at (-2,2) {};
\node[permpt,label={\tiny $p_{3}$}] (3) at (-1,-3) {};  \draw[thin] (3) -- ++ (0,5.5);
\node[permpt,label={\tiny $p_{4}$}] (4) at (-4,-2) {}; \draw[thin] (4) -- ++ (3.5,0);
\node[permpt,label={\tiny $p_{5}$}] (5) at (-3,-5) {}; \draw[thin] (5) -- ++ (0,3.5);
\node[permpt,label={\tiny $p_{6}$}] (6) at (-6,-4) {}; \draw[thin] (6) -- ++ (3.5,0);
\node[permpt,label={\tiny $p_{7}$}] (6) at (-5,-6) {}; \draw[thin] (6) -- ++ (0,2.5);

\draw[pattern=crosshatch,pattern color=black!80,draw=none] (-7,4) rectangle (5,11);
\draw[pattern=crosshatch,pattern color=black!80,draw=none] (-11,-7) rectangle (-7,4);
\draw[pattern=crosshatch,pattern color=black!80,draw=none] (-7,-11) rectangle (5,-7);
\draw[pattern=crosshatch,pattern color=black!80,draw=none] (5,-7) rectangle (11,4);

%\draw [gray!50]  (-5,9.5) -- (-5,8) -- (-5,7) -- (-4,6) -- (-2,4)  -- (-2,2);

\draw [thin] plot [smooth] coordinates {(7.5,5) (5,5) (3,4) (0,2) (-2,2)};
	
\end{tikzpicture}
\end{center}
\caption{The general form of a $\boxplus$-decomposable pin permutation in Case $2$}
\label{fig:boxdec6}
\end{figure}

Hence we have obtained another family of $\boxplus$-decomposable pin permutations, those generated by the pin words $1ldldl \dots$, and that these are all the possibilities in Case $2$.

Combining these two cases, we have obtained all possible $\boxplus$-decomposable pin permutations that begin in the first quadrant and first visit the second. By applying all eight symmetries we thus obtain all $\boxplus$-decomposable pin permutations and can see that these are precisely those described in the statement of Theorem \ref{classification}. \end{proof}

\def\cprime{$'$}

%\bibliographystyle{plain}
%{\footnotesize\bibliography{bib}}
%\bibliography{refs}

\end{document}